\documentclass[11pt]{amsart}

\usepackage[utf8]{inputenc}
\usepackage[T1]{fontenc}
\usepackage[english]{babel}

\usepackage{amsaddr}

\usepackage{amsthm}
\usepackage{amsmath}
\usepackage{amssymb}
\usepackage{mathrsfs}
\usepackage{dsfont}

\usepackage{graphicx}

\usepackage{color}
\definecolor{marin}{rgb}   {0.,   0.1,   0.5} 
\definecolor{rouge}{rgb}   {0.8,   0.,   0.} 
\definecolor{sepia}{rgb}   {0.4,   0.25,   0.} 
\definecolor{mag}{rgb}   {0.3,   0,   0.3} 

\usepackage[colorlinks,citecolor=sepia,linkcolor=marin,urlcolor=mag ,bookmarksopen,bookmarksnumbered]{hyperref}

\newcommand{\ic}{\mathrm{i}}

\newtheorem{theorem}{Theorem}[section]
\newtheorem{corollary}[theorem]{Corollary}
\newtheorem{lemma}[theorem]{Lemma}
\newtheorem{proposition}[theorem]{Proposition}
\newtheorem{definition}[theorem]{Definition}
\newtheorem{remark}[theorem]{Remark}

\newtheorem{result}[theorem]{Meta-theorem}

\begin{document}

\title[Almost global existence on compact manifolds]{Almost global existence for Hamiltonian PDEs on compact manifolds}

\author{Dario Bambusi$^*$, Joackim Bernier$^{\dagger\star}$, Beno\^it Gr\'ebert$^\dagger$, Rafik Imekraz$^\ddagger$}

\thanks{$^*$Dipartimento di Matematica, Universit\`a degli Studi di Milano, Via Saldini 50, I-20133, Milano}

\email{dario.bambusi@unimi.it}

\email{joackim.bernier@univ-nantes.fr}

\thanks{$^\dagger$Nantes Universit\'e, CNRS, Laboratoire de Math\'ematiques Jean Leray, LMJL, F-44000 Nantes, France}

\email{benoit.grebert@univ-nantes.fr}

\thanks{$^\ddagger$La Rochelle Universit\'e, Laboratoire MIA, EA 3165, France}

\email{rafik.imekraz@univ-lr.fr}

\thanks{$^\star$Corresponding author: joackim.bernier@univ-nantes.fr}

%We prove an abstract result of almost global existence of small
%solutions to semi-linear Hamiltonian partial differential equations on
%any smooth compact boundaryless Riemannian manifolds. As a main
%application we prove that, for any fixed $r\geq1$, the solution of the
%nonlinear Klein--Gordon equation on such manifolds whose initial datum
%is sufficiently small, says of size $\varepsilon\ll 1$, in the Sobolev
%space $H^s \times H^{s-1}$, with $s$ sufficiently large, exits and
%remains in $H^s \times H^{s-1}$ for polynomial times
%$|t|\leq\varepsilon^{-r}$.  Moreover, for any fixed $s_0$ independent of $r$ and $\varepsilon$,
%   the norm $H^{s_0} \times H^{s_0-1}$ of the solution remains of order
%  $\varepsilon$ for the same times. This is the first result of
%almost global existence without very specific assumptions on the
%compact manifold.  We also apply this abstract result to nonlinear
%Schr\"odinger equations close to ground states and nonlinear quantum
%harmonic oscillators.



\begin{abstract}  In this paper, we prove a general almost global existence result for semilinear Hamiltonian PDEs on compact boundaryless manifolds. As a main application, we prove the almost global existence of small solutions to nonlinear Klein--Gordon equations on such manifolds: for almost all masses, any arbitrarily large $r$ and sufficiently large $s$, solutions with initial data of sufficiently small size $\varepsilon \ll 1$ in the Sobolev space $H^s \times H^{s-1}$ exist and remain in $H^s \times H^{s-1}$ for polynomial times $|t| \leq \varepsilon^{-r}$. Surprisingly, it turns out that the geometry of the manifold has no influence.  The abstract result applies to equations satisfying very weak non resonance conditions, that are typically satisfied for PDEs with external parameters, and natural multilinear estimates.  We also apply it to nonlinear Schr\"odinger equations close to ground states and nonlinear Klein--Gordon equations on $\mathbb{R}^d$ with positive quadratic potentials.

\end{abstract}

\maketitle

\setcounter{tocdepth}{1} 
\tableofcontents

\section{Introduction}

We consider the dynamics of small and smooth solutions to spatially
confined nonlinear dispersive PDEs. We focus on models which,
  after diagonalizing the linear part can be rewritten as infinite systems of coupled
harmonic oscillators of the form
\begin{equation}
\label{eq:ham-pde}
\ic \partial_t u_j = \omega_j u_j + g_j(u), \quad \ j\in \mathbb{N}, \ t\in \mathbb{R},
\end{equation}
where $\mathbb{N} = \{ j \in \mathbb{Z} \ | \ j\geq 1\}$, $u_j(t)\in
\mathbb{C}$, $\omega_j \in \mathbb{R}$ and $g$ is at least of order
$\mathfrak{o}\geq 2$ at $u=0$ (at least formally). Since we are
interested in PDEs, we assume the
nonlinearity $g$ to fulfill
  some good properties like tame estimates. Smoothness and size of
solutions are measured in Sobolev norms defined by
$$
h^s := \big\{ u\in \mathbb{C}^{\mathbb{N}} \ | \ \| u\|_{h^s}^2 := \sum_{j \in \mathbb{N}}  j ^{2s} |u_j|^2 <\infty\big\}, \quad s\in \mathbb{R}.
$$
Local well posedness of \eqref{eq:ham-pde} usually implies that if $s$ is large enough and $\varepsilon := \| u(0) \|_{h^s}$ is small enough, then the solution exists for times of order $\varepsilon^{-\mathfrak{o}+1}$ (and remains of size $\varepsilon$ in $h^s$). For example, this property can be easily proven when $g$ is smooth from $h^s$ to $h^s$ and satisfies $\|g(u)\|_{h^s} = \mathcal{O}(\| u\|_{h^s}^{\mathfrak{o}})$. This time scale is usually called linear because the solution remains close to the solution of the linear system up to this "local" time. For longer times, nonlinear effects may become predominant and the solution could blow up: its existence is no more ensured. This raises the question of the lifespan of small and smooth solutions.

\medskip

Some nonlinear dispersive equations are globally well posed which completely solves the question (see e.g. \cite{BGT04}). Nevertheless these results require the dimension (of the spacial domain) and/or the
degree of the nonlinearity to be small enough. Here, we want to extend such results. To this end, we use a normal form approach to remove the nonlinear terms that could generate a growth of the $h^s$ norms. To avoid resonances we need two extra assumptions on the system.

\medskip

First, we need a geometric assumption on the system. Here, we focus on the Hamiltonian case, i.e. we assume the existence of a real valued function $G\in C^\infty(h^s;\mathbb{R})$ for $s$ large enough such that for all $j\in \mathbb{N}$
$$
g_j = \partial_{\overline{u_j}} G.
$$
Then, denoting by
$$
G(u) = \sum_{q\geq 3} \sum_{ \boldsymbol{j}\in \mathbb{N}^q} \sum_{\boldsymbol{\sigma}\in \{-1,1\}^q } G_{\boldsymbol{j}}^{\boldsymbol{\sigma}} u_{\boldsymbol{j}_1}^{\boldsymbol{\sigma}_1} \cdots u_{\boldsymbol{j}_q}^{\boldsymbol{\sigma}_q}
$$
the Taylor expansion of $G$ in $u=0$ where $G_{\boldsymbol{j}}^{\boldsymbol{\sigma}} \in \mathbb{C}$ and $u^{-1}_{\boldsymbol{j}_i} := \overline{u_{\boldsymbol{j}_i}}$, the point is to remove monomials $u_{\boldsymbol{j}_1}^{\boldsymbol{\sigma}_1} \cdots u_{\boldsymbol{j}_q}^{\boldsymbol{\sigma}_q}$ which do not commute with a well chosen $\widetilde{h}^s$ norm equivalent to the $h^s$ norm. This operation requires non resonance conditions on the frequencies. The two classical ones are
\begin{equation}
\label{eq:nr_1}
\tag{$\mathrm{NR}_1$}
|\boldsymbol{\sigma}_1 \omega_{\boldsymbol{j}_1} + \cdots + \boldsymbol{\sigma}_q \omega_{\boldsymbol{j}_q}  | \geq C_q |\boldsymbol{j}_1^\star|^{-a_q}
\end{equation}
\begin{equation}
\label{eq:nr_3}
\tag{$\mathrm{NR}_3$}
|\boldsymbol{\sigma}_1 \omega_{\boldsymbol{j}_1}+ \cdots + \boldsymbol{\sigma}_q \omega_{\boldsymbol{j}_q}  | \geq C_q |\boldsymbol{j}_3^\star|^{-a_q}
\end{equation}
where $C_q,a_q>0$ are constants depending only on $q$ and $\boldsymbol{j}_1^\star, \cdots, \boldsymbol{j}_q^\star$ denote the non-increasing arrangement of $\boldsymbol{j}_1, \cdots, \boldsymbol{j}_q$.
For systems with parameters (like the mass $m$ in the case of the Klein--Gordon equation \eqref{eq:KG}), the non resonance condition \eqref{eq:nr_1} is usually typically satisfied while condition \eqref{eq:nr_3} is much more restrictive and requires specific assumptions on the PDEs and its domain.

\medskip

For systems satisfying the non resonance condition  \eqref{eq:nr_3},
many results of \emph{almost global existence and stability} have been
proven. Mainly results of the following kind

\begin{result}[Almost global existence and stability]\label{result1} 
Assume \eqref{eq:nr_3}. For all $r\geq1$ if $u(0)\in h^s$ with $s$ large enough and $\|u(0)\|_{h^s}=\varepsilon$ small enough, then the solution $u(t)$ exists in $h^s$ for all time $|t|\leq \varepsilon^{-r}$ and satisfies $\|u(t)\|_{h^s}\lesssim_s \varepsilon$.
\end{result}

%That is a time of existence at least of order $\varepsilon^{-r}$, with $r$ arbitrarily large, for solutions $u$ issued from initial data of size $\varepsilon$ in $h^s$ provided that $s$ is large enough with respect to $r$ and with the additional property that $u(t)$ remains of size less than $2\varepsilon$ in $h^s$ during all this time. 
The first results in that direction was obtained by Bambusi in \cite{Bam03} and generalized in Bambusi-Gr\'ebert \cite{BG06} (and applied essentially to models in one space dimension). Then this technique was extended in a many different contexts, see for instance  
\cite{BDGS07,GIP09,FGL13,YZ14,Del15,BD17,FI21,BMM24,BG25}, but the applications always requires very special structure of the set of the frequencies. We notice that in all the cited results the authors used external parameters to ensure that  \eqref{eq:nr_3} holds true for almost all values of these external parameters.  Recently, but twenty years after  a pioneering work by Bourgain \cite{Bou00},  similar results have been obtained without external parameters \cite{BFG20a,BG21,BC24}. The bound $\|u(t)\|_{h^s}\lesssim_s \varepsilon$ ensures that the zero solution is stable in $h^s$ for times of order $\varepsilon^{-r}$, that is why we call it \emph{almost global existence and stability}. Note that this question of stability is interesting in itself, even for globally well-posed equations (like in \cite{Bou00,Bam03,GIP09,FGL13,YZ14,BFG20a,BG21}), but that it is strictly more specific than the one about the lifespan of the solutions.

%(for instance the mass $m$ in the case of the Klein--Gordon equation \eqref{eq:KG}).

\medskip

When we only have condition \eqref{eq:nr_1} (which is the situation typically occurring in higher space dimension), in general we can just extend a little the local time of existence and stability $\varepsilon^{-\mathfrak{o}+1}$, obtaining typically a time of order  $\varepsilon^{-A(\mathfrak{o}-1)}$ for some fixed $A>1$, but not arbitrary large (see in particular \cite{Del09,DI17,Brun23,FGI23}).  What we claim in this paper is that, if we just focus on the time of existence of the solutions, condition \eqref{eq:nr_1}, complemented with multilinear estimates, is sufficient to obtain, for semi-linear equations, a result of the kind
\begin{result}[Almost global existence]  \label{result2}
Assume \eqref{eq:nr_1}. 
For all $r\geq1$, if $u(0)\in h^s$ with $s$ large enough and
$\|u(0)\|_{h^s}=\varepsilon$ small enough, then the solution $u(t)$
exists in $h^s$ for all time $|t|\leq
\varepsilon^{-r}$. 
\end{result}
The precise abstract result is stated in Theorem \ref{thm:main} below. 
In fact it also includes a stability result but weaker than the one of Meta-Theorem \ref{result1}.
 Nevertheless, surprisingly, in the emblematic case where $u(0) = \epsilon v$ with $v\in h^\infty$ and $\epsilon \ll 1$, these stability results are equivalent (see Corollary \ref{cor:smooth} below). 
 The concrete applications are stated in Theorem \ref{thm:kg} for Klein-Gordon equations on compact manifolds, Theorem \ref{thm:nls} for nonlinear Schr\"odinger equations on compact manifolds close to ground states and Theorem \ref{thm:quatumoscillator} for nonlinear Klein--Gordon equations on $\mathbb{R}^d$ with quadratic potentials.
The important point is that condition \eqref{eq:nr_1} is much weaker than \eqref{eq:nr_3}. In particular, in the case of the Klein-Gordon equations, whereas \eqref{eq:nr_3} has been proved to hold true for almost all value of the mass for some particular manifolds (Zoll manifolds) while \eqref{eq:nr_1} holds true on any smooth compact boundaryless Riemannian manifolds (\cite{DS04}).

\medskip

The fact that \eqref{eq:nr_1} is sufficient to ensure almost global existence was already proved in \cite{BFG20b} for nonlinear Klein--Gordon equations on $d$-dimensional tori. Here we take advantage of this remark in a much more general situation.

\medskip

Finally, we notice that, recently, using Bourgain's cluster
decomposition, the almost global existence and stability of some
\eqref{eq:nr_1}-non-resonant Hamiltonian PDEs (not including
Klein--Gordon) on compact Riemannian manifolds with globally
integrable geodesic flow have been proved \cite{BFM24,BFLM24} {(see
also \cite{BL22})}. The point is that thanks to the existence of a Bourgain's cluster decomposition these systems almost satisfy the \eqref{eq:nr_3} non resonance condition (i.e. up to some quite trivial terms).

\subsection*{Basic notations} The Japanese bracket is defined by $\langle a \rangle := \sqrt{1+a^2}$ for $a\in \mathbb{R}$.  We shall use the notation $A\lesssim B$ to denote $A\le C B$ where $C$ is a positive constant
depending on  parameters fixed once for all. We will emphasize by writing $\lesssim_{\alpha}$ when the constant $C$ depends on some other parameter $\alpha$ and $A \sim_\alpha B$ if $A\lesssim_\alpha B \lesssim_\alpha A$.

 We rewrite as usual $L^2=\ell^2 := h^0$. For all $q\geq 3$ and $\boldsymbol{k} \in \mathbb{C}^q$,  $\boldsymbol{k}_1^\star, \cdots, \boldsymbol{k}_q^\star$ denotes the non-increasing arrangement of $\boldsymbol{k}_1, \cdots, \boldsymbol{k}_q$. For all $j\in \mathbb{N}$, the operators $\partial_{\overline{u_j}},\partial_{u_j}$ are defined by $2\partial_{\overline{u_j}} := \partial_{\Re u_j} + \ic\partial_{\Im u_j}$ and $2\partial_{u_j} := \partial_{\Re u_j} - \ic\partial_{\Im u_j}$. For all $j \in \mathbb{N}$, $\mathds{1}_{\{j\}} \in \mathbb{C}^\mathbb{N}$ denotes the sequence such that $(\mathds{1}_{\{j\}} )_i=1$ if $i=j$ and $(\mathds{1}_{\{j\}} )_i=0$ otherwise.

 %mal dÃƒÂ©fini et pas au bon endroit
 %Correspondingly we define $\Pi_{\left\{j\right\}}u=u_j$. Given a subset $\mathcal{C}\subset \mathbb{N}$ we define $\mathds{1}_{\mathcal{C}} $ to be the characteristic function of the set $\mathcal{C}$, namely the functions s.t. $(\mathds{1}_{\mathcal{C}} )_i=1$ if $i\in\mathcal{C}$ and zero otherwise. We also denote $\Pi_{\mathcal{C}}:=\sum_{j\in\mathcal{C}}\Pi_{\left\{j\right\}}$.
  
\subsection*{Acknowledgments} During the preparation of this work the authors benefited from the support of the Centre Henri Lebesgue ANR-11-LABX-0020-0 and J.B. was also supported by the region "Pays de la Loire" through the project "MasCan".
J.B. and B.G. were partially supported by the ANR project KEN ANR-22-CE40-0016. D.B was supported by the research projects PRIN 2020XBFL "Hamiltonian and dispersive PDEs" of the Italian Ministry of Education and Research (MIUR).

\section{Abstract theorem} We give an abstract theorem to prove the
\emph{almost global existence} of small and smooth solutions to
\eqref{eq:ham-pde}. First we present its setting and then we state the
result and one of its main corollary. As we will see in the next section, these assumptions are
natural for Hamiltonian PDEs on smooth compact boundaryless Riemannian
manifolds and more generally for Hamiltonian PDEs with a linear part
which only has imaginary pure point spectrum.

\medskip

Let $\omega \in (\mathbb{R}_+^*)^{\mathbb{N}}$ be non-decreasing, $s_0\geq 0$,
$\mathcal{U} \subset h^{s_0}$ be an open neighborhood of the origin,
$G\in  C^\infty(\mathcal{U};\mathbb{R})$ be a function of order larger
than or equal to $3$ at the origin and set, for all $j\in \mathbb{N}$,
$g_j := \partial_{\overline{u_j}} G$ (so that the system is
Hamiltonian). 

\medskip

Considering from now these objects as fixed, assume that  
\begin{itemize}
\item \emph{the nonlinearity is smooth, preserves the $h^s$ regularity and is tame}: for all $s\geq s_0$,
$g$ is a $C^\infty$ function from $h^s \cap \mathcal{U}$ into $h^s$, its derivative is uniformly bounded on bounded sets and it satisfies
\begin{equation}
\label{eq:lestimee_tame_pas_facile_a_pas_oublier}
\forall u \in h^s \cap \mathcal{U}, \quad \| g(u) \|_{h^s} \lesssim_{s}  \| u\|_{h^s}.
\end{equation}
\item {\it Weyl law}: there exists $\beta >0 $ such that provided that $\lambda$ is large enough
  \begin{equation}
    \label{weyl}
 \#\left\{j \in \mathbb{N} \ |\ \omega_j\leq\lambda\right\} \sim \lambda^{ \beta}
    \end{equation}

\item {\it Clustering}: there exist two positive numbers $\alpha,\Upsilon>0$ and a decomposition in disjoint subsets\footnote{some of them may be empty.}
$$
\mathbb{N} = \bigcup_{k\in \mathbb{N}} \mathcal{C}_k 
$$
such that
\begin{equation}
\label{eq:inclusion_clusters}
\sup_{k\in \mathbb{N}} \sup_{j\in \mathcal{C}_k} |\omega_j^{1/\alpha} - \Upsilon k| < \infty
\end{equation}
and
\begin{itemize}
\item \emph{the system is non resonant}: 
  for all $q\geq 1$, there exists $a_q > 0$ such that for $\boldsymbol{j} \in \mathbb{N}^q$, $\boldsymbol{\sigma}\in \{-1,1\}^q$ 
\begin{equation}
\label{eq:nr_thm}
\mathrm{if} \ \exists k\in \mathbb{N}, \quad  \sum_{i \ s.t. \ \boldsymbol{j}_i \in  \mathcal{C}_k } \boldsymbol{\sigma}_i \neq 0 \quad \mathrm{then} \quad |\boldsymbol{\sigma}_1 \omega_{\boldsymbol{j}_1}+ \cdots + \boldsymbol{\sigma}_q \omega_{\boldsymbol{j}_q}  | \gtrsim_q |\boldsymbol{j}_1^\star|^{-a_q},
\end{equation}

\item \emph{the nonlinearity satisfies the following multilinear estimate}: there exists $\nu \geq 0$ such that for all $q \geq 3$, all $n\geq 0$, all $\boldsymbol{k} \in \mathbb{N}^q$ and all $(u^{(\ell)})_{1\leq \ell \leq q} \in \prod_{1\leq \ell \leq q} E_{\boldsymbol{k}_{\ell}}$,
\begin{equation}
\label{eq:estimates_thm}
   \big| \mathrm{d}^{q} G(0) (u^{(1)},\cdots,u^{(q)}) \big|  \lesssim_{n,q}  \Gamma_{\boldsymbol{k}} \Big(   \frac{\boldsymbol{k}_2^\star }{ \boldsymbol{k}_1^\star } \Big)^n \Big(\prod_{3\leq \ell \leq q} \boldsymbol{k}_\ell^\star\Big)^\nu  \prod_{\ell=1}^q \| u^{(\ell)} \|_{\ell^2}
\end{equation}
where for all $k\in \mathbb{N}$, 
$E_k := \mathrm{Span}_{\mathbb{C}} \{ \mathds{1}_{\{j\}} \ | \ j\in \mathcal{C}_k \} \subset \ell^2:= h^0$ and 
\begin{equation}
\label{eq:def_gamma_k}
 \Gamma_{\boldsymbol{k}} := \sum_{\boldsymbol{\varsigma} \in \{-1,1\}^q} \langle  \boldsymbol{\varsigma}_1 \boldsymbol{k}_1 + \cdots + \boldsymbol{\varsigma}_q \boldsymbol{k}_q \rangle^{-3}.
\end{equation}
\end{itemize}
\end{itemize}

%%%%%%%%%%%%%%%%%%%%%%%%%%%%%%%%

For convenience, we set
$$
s_{\mathrm{min}} := \max\Big(s_0, \frac{\nu+2}{ \alpha\beta}\Big).
$$

\begin{theorem}
\label{thm:main} 
For all $r \geq 1$, $s_c \geq s_{\mathrm{min}} $, $s\gtrsim_{r,s_c} 1$, $u^{(0)} \in h^s$ satisfying $ \| u^{(0)} \|_{h^s} \leq  1$ and $\varepsilon := \| u^{(0)} \|_{h^{s_c}} \lesssim_{r,s_c} 1$, there exists a unique solution 
 $$
 u\in C^0\big((-\varepsilon^{-r},\varepsilon^{-r}) ;h^s \big) \cap  C^1\big((-\varepsilon^{-r},\varepsilon^{-r}) ;h^{s-\frac{1}{\beta}} \big) 
 $$
 to \eqref{eq:ham-pde} with initial datum $u(0) = u^{(0)}$, satisfying, as long as $|t| \leq \varepsilon^{-r}$, that
 \begin{equation}
 \label{eq:est_growth}
 \| u(t)\|_{h^{s_c}} \lesssim_{s_c} \| u(0)\|_{h^{s_c}}.
 \end{equation}
\end{theorem}

We begin with a few comments on the result.
\begin{itemize}
\item There are $3$ regularity indices in this theorem $s_{\mathrm{min}}\leq s_c \leq s$:
\begin{itemize}
\item $s_{\mathrm{min}}$ is the one on which our tame estimates are based,
\item $s_c$ is the one of the norm we aim at controlling,
\item $s$ represent the regularity needed on the initial data to get the almost global existence.
\end{itemize}   
\item The bound $s\gtrsim_{r,s_c} 1$ could be specified: the constant only depend on $r,\alpha,\beta,\nu,a,s_c$.
\item It would have been more standard to require that  $\varepsilon:= \| u^{(0)} \|_{h^{s}} \lesssim_{r,s} 1 $ instead of $ \| u^{(0)} \|_{h^s} \leq  1$ and $\varepsilon := \| u^{(0)} \|_{h^{s_c}} \lesssim_{r,s_c} 1$ (see e.g. Meta-Theorem \ref{result2}). These assumption are almost equivalent. Indeed, since $s$ is larger than $s_c$, if  $\varepsilon:= \| u^{(0)} \|_{h^{s}} \lesssim_{r,s} 1 $ then $\| u^{(0)} \|_{h^{s_c}} \leq \varepsilon$. Conversely, thanks to the H\"older inequality $\| \cdot \|_{h^{s_\theta}} \leq \| \cdot \|_{h^{s_c}}^{\theta} \| \cdot \|_{h^s}^{1-\theta}$ for $\theta \in [0,1]$ and $s_{\theta} = \theta s_c + (1-\theta) s$, our assumption allows to impose a smallness condition on any high regularity initial datum. Nevertheless, we believe that our assumptions are more natural in this setting because Theorem \ref{thm:main} only allows to control the growth of the $h^{s_c}$ norm.
\end{itemize}
Now, let us compare Theorem \ref{thm:main} with almost global and stability results (of the type Meta-Theorem \ref{result1}). 
With the same kind of non standard assumptions (i.e. $ \| u^{(0)} \|_{h^s} \leq  1$ and $\varepsilon := \| u^{(0)} \|_{h^{s_c}} \lesssim_{r,s_c} 1$ but assuming \eqref{eq:nr_3}), the proof of most of the "almost global and stability results" could be adapted to control the growth of all the $h^{\varsigma}$ norms, i.e. for all $\varsigma \geq s_{\mathrm{min}}$ and all $t\in (-\varepsilon^{-r}, 
\varepsilon^{-r})$, we have $\| u(t)\|_{h^{\varsigma}} \lesssim_{\varsigma} \| u(0)\|_{h^{\varsigma}}$ (see \cite{BG25} for a formulation of this kind). In particular, it would imply that $\| u(t)\|_{h^{s}}$ remains of order $1$. Here, the proof of Theorem \ref{thm:main}  does not allow to control the growth of the  $h^{\varsigma}$ norms when $\varsigma$ is too large with respect to $r$ and $s_c$. In particular, we do not know if  $\| u(t)\|_{h^{s}}$ remains of order $1$. Nevertheless, thanks to the tame estimate \eqref{eq:lestimee_tame_pas_facile_a_pas_oublier}, we still have a rough bound of the form $\| u(t) \|_{h^s} \leq \exp(C_s t) \varepsilon$ where $C_s>0$ is a constant depending only on $s$.

\medskip

In order to look for the apparition of possible energy cascades, it is natural to consider special initial data of the form  $ u^{(\epsilon)}(0) =\epsilon v$ with $v$ smooth and $\epsilon \ll 1$. Surprisingly, in this setting, the conclusions of Theorem  \ref{thm:main} are the same as those we would get from an almost global and stability results like Meta-Theorem \ref{result1}: forward energy cascades, if they exist, are necessarily very slow. This is the content of the following corollary whose proof is given in Appendix \ref{sec:appendix_B}.
\begin{corollary} \label{cor:smooth} Let $v \in h^\infty := \bigcap_{s\geq 0} h^s$. There exists $(T_\epsilon)_{\epsilon\in(0,1)} \in \mathbb{R}_+^{(0,1)}$ satisfying
 $$
 \forall r\geq 0,\quad \frac{T_\epsilon}{\epsilon^{-r}} \mathop{\longrightarrow}_{\epsilon \to 0} +\infty
 $$
 and, for all $ \epsilon \in (0,1)$, there exists a unique solution 
 $$
 u^{(\epsilon)} \in C^\infty\big([-T_\epsilon,T_\epsilon] ;h^\infty \big)
 $$
 to \eqref{eq:ham-pde} with initial datum $ u^{(\epsilon)}(0) =\epsilon v$ satisfying, as long as $|t|\leq T_\epsilon$,
 $$
\forall s\geq  0, \quad \|  u^{(\epsilon)}(t)\|_{h^{s}} \lesssim_{s,v} \|  u^{(\epsilon)}(0)\|_{h^{s}}.
 $$

\end{corollary}

On first reading, this corollary may give the impression that we have shown a result of  Meta-theorem \ref{result1} type, but this is not the case: here the shape of the initial condition, $v$, is fixed and the control we have over the $h^s$ norm of $u$ depends on $v$ (in fact, more precisely, it depends on the distribution of the high Fourier modes of $v$). This corollary can be seen as a "radial stability" result: it prevents growth, for very long times, of the $h^s$ norm (and so weak turbulence) for initial datum $u(0)=\epsilon v$ when $\epsilon$ goes to $0$.

\medskip

Let's make a few comments on the assumptions.
\begin{itemize}
\item In order to prove the multilinear estimate
 \eqref{eq:estimates_thm} typically one uses that the operator
diag$(\omega_j^{1/\alpha})$ is pseudodifferential of order $1$. That is
why we introduced the exponent $\alpha$.

\item By the Weyl law and the bound \eqref{eq:inclusion_clusters}, one has
  $$
 \# \bigcup_{k\leq \lambda}\mathcal{C}_k \sim
 \lambda^{\alpha \beta}.
 $$

\item The non resonance condition \eqref{eq:nr_thm}  allows to remove monomials which do not commute with the super-actions 
\begin{equation}
\label{eq:def_super_actions}
J_k(u) := \sum_{j \in \mathcal{C}_k} |u_j|^2.
\end{equation}  
This is the non resonance condition \eqref{eq:nr_1} of the introduction, in the sense that these monomials are exactly those do not commute with the $H^s$ norms defined in Definition \ref{def:Hs_space} below.

\item The multilinear estimates \eqref{eq:estimates_thm} are new but they are implied by the multilinear estimates of the type of those proved by Delort--Szeftel in \cite{DS06}, namely
\begin{equation}
\label{eq:DS-estimates}
   \big| \mathrm{d}^{q} G(0) (u^{(1)},\cdots,u^{(q)}) \big|  \lesssim_{n,q}   \frac{(\boldsymbol{k}_{3}^\star)^{\nu +n}  (\boldsymbol{k}_{4}^\star)^{\nu } \cdots (\boldsymbol{k}_{q}^\star)^{\nu }}{ (\boldsymbol{k}^\star_{1} - \boldsymbol{k}^\star_{2} + \boldsymbol{k}^\star_{3} )^{n}}  \prod_{\ell=1}^q \| u^{(\ell)} \|_{\ell^2}.
\end{equation}
See  Lemma \ref{A.6} for the fact that \eqref{eq:DS-estimates} implies  \eqref{eq:estimates_thm}.

\item If the nonlinearity $g$ is polynomial then the tame estimate \eqref{eq:lestimee_tame_pas_facile_a_pas_oublier} is a consequence of the multilinear estimates \eqref{eq:estimates_thm} (see Lemma \ref{lem:vf} below for a proof).

 \item The following equivalence holds true
\begin{equation*}
 \Gamma_{\boldsymbol{k}} \sim_q  \big( 1+ \min\limits_{\boldsymbol{\varsigma} \in \{-1,1\}^q} | \boldsymbol{\varsigma}_1 \boldsymbol{k}_1 + \cdots + \boldsymbol{\varsigma}_q \boldsymbol{k}_q  |^3 \big)^{-1} .
\end{equation*}
\item the term $\Upsilon k$ in \eqref{eq:inclusion_clusters} could be replaced by $b_k$, where $b\in (\mathbb{R}_+^*)^{\mathbb{N}}$ would be an increasing sequence satisfying $b_{k+1} - b_k \sim 1$ (i.e. uniformly in $k$).
\end{itemize}

Finally, let us comment some technical novelties of this paper. An important part of the work consists of combining the techniques developed in \cite{DS04,DS06,BFG20b}. Nevertheless, it is not direct. In particular, we point out two crucial points.
\begin{itemize}
\item The method introduced in \cite{BFG20b} allows to prove the existence for times of order $\varepsilon^{- c r_{\flat}/s_{\flat}}$ where $c>0$ is a universal constant, $r_{\flat}$ is the number of Birkhoff normal form steps we perform  and $s_{\flat}$ is the minimal regularity for which we have tame estimates at the end of the Birkhoff normal form process (see e.g. equation \eqref{eq:on_voit_bien_le_temps} below). Basing our construction on Delort--Szeftel multilinear estimates \eqref{eq:DS-estimates}, $s_{\flat}$ would grow linearly with respect to $r$ (see the exponent $\nu$ in \cite[Theorem 2.14]{DS06}) and we could not conclude the almost global existence of the solutions. The point is that these estimates contain too much information and so are too costly to propagate in the normal form process. A natural alternative would have been the tame-modulus estimates of \cite{BG06} but they are too weak: they do not allow to gain derivatives when computing brackets with $h^s$ norms (which is a crucial property for our proof, see Proposition \ref{prop:est_poisson}). That is why we proposed the new multilinear estimates \eqref{eq:estimates_thm} which are in between and overcome these obstructions.
\item In \cite{BFG20b}, the solutions were estimated in mixed norms: $h^{s_c}$ for high modes and $h^s$ with $s\gg r$ for low modes. Nevertheless, this disjunction does not allow for satisfactory treatment of remainder terms and generates technicalities. Here we simplified the approach by removing the mixed norms: the solution is just estimated in $h^{s_c}$ norm.
\item In this paper, $\varepsilon$ represents the $h^{s_c}$ norm of the initial datum, unlike in \cite{BFG20b}, where it represents its $h^{s}$ norm. This refinement is crucial as it enables us to derive Corollary \ref{cor:smooth}.
\end{itemize}

\section{Applications }\label{applications}
Now we present applications of this abstract result to two emblematic
Hamiltonian PDEs on an arbitrary boundaryless smooth Riemannian
manifold: the nonlinear Klein--Gordon equations and the nonlinear
Schr\"odinger equations close to ground states. As in
\cite{BFM24,BFLM24}, the result could be applied to other classical
semi-linear equations like the beam equation or the nonlinear
Schr\"odinger equations close to the origin with a spectral
multiplier. We chose the nonlinear Klein--Gordon equations 
because they are the most emblematic ones and the nonlinear Schr\"odinger equations to
emphasize that we do not need the nonlinearity to be
smoothing. We also provide an application to nonlinear Klein--Gordon equations on $\mathbb{R}^d$ with positive definite quadratic
  potentials to emphasize that our abstract theorem can even be applied beyond the framework of Hamiltonian PDEs on compact manifolds. 
 For brevity, we only state Corollary \ref{cor:smooth} in the context of the nonlinear Klein--Gordon equations on compact manifolds but, of course, it could also be applied to the other equations. Proofs are given in Section \ref{Sec:proofs_app} below.

\subsection{Nonlinear Klein--Gordon equations on compact manifolds}

We consider nonlinear Klein--Gordon equations of the form
\begin{equation}
\label{eq:KG} 
\tag{KG} \partial_t^2 \Psi(t,x) = (\Delta-m)  \Psi(t,x) - V(x) \Psi(t,x)+ f(x, \Psi(t,x)), \quad t\in \mathbb{R_+^*}, \ x\in \mathcal{M}  
\end{equation}
where $ \Psi(t,x)\in \mathbb{R}$,  $\mathcal{M}$ is a smooth compact
boundaryless Riemannian manifold of dimension $d\geq 1$, $V\in
C^\infty(\mathcal{M}; \mathbb{R}_+)$, $m>0$ is a parameter called
\emph{mass} and $f\in \mathcal{C}^\infty(\mathcal{M}\times \mathbb{R};
\mathbb{R})$ satisfies $f(\cdot,0) = \partial_ \Psi
f(\cdot,0)=0$ (to ensure that it is of order at least $2$ with respect to $\Psi$).

\medskip

%{\color{blue}The almost global existence of the small solutions to \eqref{eq:KG} on general manifolds is a longstanding question.  }

%\medskip

\begin{theorem}\label{thm:kg} Fix $s_1\gtrsim_d 1$. For almost all $m>0$, all $r \geq 1$, all $s\gtrsim_{r,m} s_1$, 
and any couple of real-valued functions $(\Psi_0, \Phi_0)\in H^{s+1}(\mathcal M ;\mathbb{R}) \times  H^{s}(\mathcal M;\mathbb{R})$  such that $\varepsilon := \| \Psi_0  \|_{H^{s_1+1}}+\| \Phi_0  \|_{H^{s_1}} \lesssim_{r,s_1,m} 1$ and $\| \Psi_0  \|_{H^{s+1}}+\| \Phi_0  \|_{H^{s}} \leq 1$, there exists a unique solution 
 $$
 \Psi\in C^0\big((-\varepsilon^{-r},\varepsilon^{-r}) ;H^{s+1}(\mathcal M;\mathbb{R}) \big) \cap  C^1\big((-\varepsilon^{-r},\varepsilon^{-r}) ;H^{s}(\mathcal M;\mathbb{R}) \big) 
 $$
 to \eqref{eq:KG} with initial datum $\Psi(0) = \Psi_0,\ \partial_t \Psi(0) = \Phi_0$. Furthermore, as long as $|t| \leq \varepsilon^{-r}$, one has
 $
\| \Psi(t)  \|_{H^{s_1+1}}+\| \partial_t\Psi(t)  \|_{H^{s_1}} \lesssim \varepsilon.
 $

\end{theorem}

\begin{remark} We recall that the Sobolev spaces $H^{s}(\mathcal M;  \mathbb{K})$, $\mathbb{K} \in \{ \mathbb{R},\mathbb{C}\}$ and $s\geq
  0$, are defined, as usual, by 
\begin{equation}
\label{eq:def_Hs_M}
H^s(\mathcal{M}; \mathbb{K}) := \Big\{ u \in L^2(\mathcal M; \mathbb{K}) \ | \ \| u\|_{H^s} := \| (1-\Delta)^{s/2}u\|_{L^2} <\infty  \Big\} .
\end{equation}
\end{remark}

Moreover, applying Corollary \ref{cor:smooth} in this setting, we deduce the following result.
\begin{corollary} \label{cor:kg}
For almost all $m>0$, any couple of real-valued functions $\Psi_0, \Phi_0\in C^\infty(\mathcal{M};\mathbb{R})$ and any $\varepsilon \in (0,1)$, there exists 
a unique solution  $\Psi^{(\varepsilon)}\in C^\infty ((-T_\varepsilon,T_\varepsilon)\times \mathcal{M};\mathbb{R} ) $ to \eqref{eq:KG} with  initial datum $ \Psi^{(\varepsilon)}(0) =\varepsilon \Psi_0,\ \partial_t  \Psi^{(\varepsilon)}(0) = \varepsilon \Phi_0$ satisfying
 $$
\forall r\geq 3, \quad \lim_{\varepsilon \to 0} \frac{T_\varepsilon}{\varepsilon^{-r}} = +\infty. 
 $$
 and
 $$
 \forall t\in (-T_\varepsilon,T_\varepsilon),\forall s\geq 0, \quad  \| \Psi^{(\varepsilon)}(t)  \|_{H^{s+1}}+\| \partial_t \Psi^{(\varepsilon)}(t)  \|_{H^{s}}  \lesssim_s \| \Psi^{(\varepsilon)}(0)  \|_{H^{s+1}}+\| \partial_t \Psi^{(\varepsilon)}(0)  \|_{H^{s}} .
 $$
\end{corollary}

\subsection{Nonlinear Schr\"odinger equations on compact manifolds}

We consider nonlinear Schr\"odinger equations of the form
\begin{equation}
\label{eq:NLS} 
\tag{NLS}\  \ic \partial_t z(t,x) = -\Delta z(t,x) + f(|z(t,x)|^2)z(t,x) \quad t\in \mathbb{R}, \ x\in \mathcal{M}
\end{equation}
where $z(t,x)\in \mathbb{C}$,  $\mathcal{M}$ is a smooth compact
boundaryless  connected Riemannian manifold of dimension $d\geq 2$ and $f\in
\mathcal{C}^\infty( \mathbb{R}; \mathbb{R})$ satisfies $f(0) = 0$. 

\medskip

It is immediate to verify that $z_*(t):=\sqrt{p_0}e^{-\ic\nu t }$ is a
solution of \eqref{eq:NLS} if and only if $\nu=f(p_0)$. To ensure linear stability of this solution, we denote by $\lambda_{2}^2 >0$
the second smallest  eigenvalue (with multiplicity) of $-\Delta$ on $\mathcal{M}$ and only consider values of $p_0>0$ such that  $\lambda_{2}^2 +2p_0f'(p_0)>0$.

\begin{theorem}\label{thm:nls}
  Fix $s_1 \gtrsim_d 1$. There exists a zero measure set $\mathcal{N} \subset \mathbb{R}$ such that for all $p_0 >0$ satisfying $2p_0f'(p_0) \in (-\lambda_{2}^2,+\infty ) \setminus \mathcal{N}$, all $r \geq
  1$, all $s\gtrsim_{r,p_0} s_1$, and any  functions $z_0\in
  H^{s}(\mathcal M)$ such that
  \begin{equation}
  \label{eq:hypopo}
\|z_0\|^2_{L^2}=p_0\ ,\quad \inf _{\theta\in
  \mathbb{T}}\|z_0-\sqrt{p_0}e^{-\ic\theta}\|_{H^{s_1}}=:\varepsilon\lesssim_{r,s,p_0}1 , \quad \|z_0\|_{H^s} \leq 2 \sqrt{p_0}
  \end{equation}
there exists a  unique solution
 $$
z\in C^0\big((-\varepsilon^{-r},\varepsilon^{-r}) ;H^{s}(\mathcal M) \big) \cap  C^1\big((-\varepsilon^{-r},\varepsilon^{-r}) ;H^{s-2}(\mathcal M) \big) 
 $$
 to \eqref{eq:NLS} with initial datum $z(0) =z_0$. Furthermore, as long
 as $|t| \leq \varepsilon^{-r}$, one has
 $$
\inf _{\theta\in   \mathbb{T}}\|z(t)-\sqrt{p_0}e^{-\ic\theta}\|_{H^{s_1}}\lesssim \varepsilon
$$
\end{theorem}
\begin{remark}
The constant $2 \sqrt{p_0}$ in \eqref{eq:hypopo} is purely arbitrary, it can be replaced by any constant larger than $\sqrt{p_0}$.
\end{remark}

\subsection{Nonlinear Klein--Gordon equations on $\mathbb{R}^d$ with positive definite quadratic potential.}

Here, we consider the following Klein-Gordon equation:
\begin{equation}\label{eq:KGQ}
\partial_t^2 \Psi(t,x)=(\Delta-m)\Psi(t,x)-Q(x)\Psi(t,x)+f(\Psi(t,x)),\qquad t\in \mathbb{R}, \ x\in \mathbb{R}^d
\end{equation}
in which $m>0$, $f\in \mathcal{C}^\infty(\mathbb{R};
\mathbb{R})$ satisfies $f(0) = 
f'(0)=0$ and $Q:\mathbb{R}^d\rightarrow \mathbb{R}$ is a positive definite quadratic form:
\begin{equation}
Q(x)=\sum_{i=1}^d \sum_{j=1}^d q_{ij} x_i x_j,\qquad \mbox{with} \quad q_{ij}=q_{ji} \quad \mbox{and}\quad Q(x)>0 \quad \forall x\in \mathbb{R}^d \setminus \{0 \}.
\end{equation}

\begin{remark}
  \label{eigen}
Denoting by $\varrho_d\geq\dots \geq \varrho_1>0$ the eigenvalues of
the quadratic form $Q$ one has that the spectrum of the quantum
operator $-\Delta+Q(x)$ is exactly
\begin{equation*}
\sum_{i=1}^d \sqrt{\varrho_i}(2\mathbb{N}+1).
\end{equation*}
In particular one has that if the numbers $\sqrt{\varrho_i}$ are
independent over the rational then their differences are dense on the
real axe. The same is true for the the eigenvalues of
$\sqrt{-\Delta+Q}$ and, generally speaking also those of
$\sqrt{-\Delta+Q+m}$ which are the frequencies in our problem. 
\end{remark}

As we shall see in subsection \ref{harm-sub}, the adequate Sobolev space is as follows:

\begin{equation}\label{harm-sob}
 \mathcal{H}^s(\mathbb{R}^d):=\Big\{u\in H^s(\mathbb{R}^d) \ | \ \quad \langle x \rangle^{s} u\in L^2(\mathbb{R}^d)\Big\}.
\end{equation}

We have exactly the same statement as Theorem \ref{thm:kg}:
\begin{theorem}
\label{thm:quatumoscillator}
Fix $s_1\gtrsim_d 1$. For almost all $m>0$, all $r \geq 1$, all $s\gtrsim_{r,m} s_1$, 
and any couple of real-valued functions $(\Psi_0, \Phi_0)\in \mathcal{H}^{s+1}(\mathbb{R}^d) \times  \mathcal{H}^{s}(\mathbb{R}^d)$  such that $\varepsilon := \| \Psi_0  \|_{\mathcal{H}^{s_1+1}(\mathbb{R}^d)}+\| \Phi_0  \|_{\mathcal{H}^{s_1}(\mathbb{R}^d)} \lesssim_{r,s_1,m} 1$ and $\| \Psi_0  \|_{\mathcal{H}^{s+1}(\mathbb{R}^d)}+\| \Phi_0  \|_{\mathcal{H}^s(\mathbb{R}^d)} \leq 1$, there exists a unique solution 
 $$
 \Psi\in C^0\big((-\varepsilon^{-r},\varepsilon^{-r}) ;\mathcal{H}^{s+1}(\mathbb{R}^d) \big) \cap  C^1\big((-\varepsilon^{-r},\varepsilon^{-r}) ;\mathcal{H}^{s}(\mathbb{R}^d) \big) 
 $$
 to \eqref{eq:KGQ} with initial datum $\Psi(0) = \Psi_0,\ \partial_t \Psi(0) = \Phi_0$. Furthermore, as long as $|t| \leq \varepsilon^{-r}$, one has
 $
\| \Psi(t)  \|_{\mathcal{H}^{s_1+1}(\mathbb{R}^d)}+\| \partial_t\Psi(t)  \|_{\mathcal{H}^{s_1}(\mathbb{R}^d)} \lesssim \varepsilon.
 $
\end{theorem}

\section{Proofs of the applications}
\label{Sec:proofs_app}
\subsection{Klein--Gordon on compact Riemannian manifold}\label{KG_proof}

Let $\mathcal{M}$ be a smooth compact boundaryless Riemannian manifold of dimension $d\geq 1$, $s_2> \max(d/2 - 1,0)$ and  $V\in C^\infty(\mathcal{M}; \mathbb{R}_+)$.
In order to prove Theorem \ref{thm:kg}, we have first to explain how \eqref{eq:KG} rewrites in the framework of Theorem \ref{thm:main} and then why it satisfies its assumptions.

\medskip

\noindent \underline{\emph{Step 1: Equivalence of the formalisms.}}  First we recall that the spectrum of the operator ${-\Delta+V}$ acting on $L^2(\mathcal{M};\mathbb{R})$  is pure point.
We denote by $(\lambda_j^2)_{j\geq 1}$ the nondecreasing sequence of its
eigenvalues (with multiplicities) and $(e_j)_{j\in \mathbb{N}}$ an associated real Hilbertian basis. The eigenvalues satisfy the Weyl law \cite{Hor68}
\begin{equation}
\label{eq:true_Weyl}
\forall y\gg 1, \quad \# \{ j\in \mathbb{N} \ | \ \lambda_j \leq y \} \sim y^d.
\end{equation}
 We identify any function $u\in L^2(\mathcal{M};\mathbb{C})$ with its associated sequence of coefficients $(u_j)_{j\in \mathbb{N}} \in \ell^2$ in this basis, i.e. such that
$$
(u_j)_{j\in \mathbb{N}}  \equiv u =: \sum_{j\in \mathbb{N}} u_j e_j.
$$  
Then, we prove in the following lemma that the standard Sobolev space $H^s(\mathcal{M};\mathbb{K})$ (defined by \eqref{eq:def_Hs_M}) we used to state Theorem \ref{thm:kg} are equivalent to the discrete Sobolev spaces $h^s$ we used to state our abstract theorem.
\begin{lemma} \label{eq:norm_equiv} For all $s\geq 0$ and all $\mathbb{K}\in \{\mathbb{R} ,\mathbb{C}\}$, we have that
$$
H^s(\mathcal{M};\mathbb{K}) = h^{s/d} \cap \mathbb{K}^{\mathbb{N}} \quad \mathrm{and} \quad \| \cdot \|_{H^s} \sim_{s,V} \| \cdot \|_{h^{s/d}}.
$$
\end{lemma}
\begin{proof} First, we note that since $V$ is bounded and nonnegative, for all $u\in \mathcal{C}^\infty(\mathcal{M};\mathbb{C})$
$$
\| u \|_{H^s}^2 = \| (\mathrm{Id} - \Delta)^{s/2} u \|_{L^2}^2 \sim_{s,V}  \| (\mathrm{Id} - \Delta + V)^{s/2} u \|_{L^2}^2 = \sum_{j\in \mathbb{N}} \langle \lambda_j \rangle^{2s} |u_j|^2.
$$
Since by the Weyl law, we have $\langle \lambda_j \rangle \sim j^{1/d} $ (it suffices to apply \eqref{eq:true_Weyl} with $y=\lambda_j$ and  $y=\lambda_j^-$), we deduce that $\| u\|_{H^s} \sim_{s,V} \|u\|_{h^{s/d}}$. The rest of the proof follows directly by density.
\end{proof}
Therefore, by setting, for all $m>0$,
$$
u=\Lambda_m^{1/2}\Psi  + i\Lambda_m^{-1/2}\partial_t \Psi ,
$$ 
where $\Lambda_m$ is the spectral multiplier defined by
$$
 \Lambda_m = \sqrt{-\Delta + V + m},
$$
we deduce that $(\Psi,\partial_t \Psi)$  is solution of \eqref{eq:KG} in $H^{s+1}(\mathcal{M};\mathbb{R}) \times H^{s}(\mathcal{M};\mathbb{R})  $ if and only if $u$ is solution of \eqref{eq:ham-pde} in $H^{s+1/2}(\mathcal{M};\mathbb{C})  $ with
\begin{equation}\label{def-g}
\forall j\in \mathbb{N}, \quad \omega_j :=  \sqrt{\lambda_j^2+m} \quad \mathrm{and} \quad g(u) := \Lambda_m^{-1/2}f\left(x,\Lambda_m^{-1/2}\left(\frac{u+\bar u}2\right)\right).
\end{equation} 
Finally, we set naturally, $s_0:= s_2/d$ so that $H^{s_2}(\mathcal{M};\mathbb{C}) = h^{s_0} $ .

\medskip

\noindent \underline{\emph{Step 2: Validity of the assumptions.}} For simplicity, we choose $\mathcal{U} = B_{H^{s_2}}(0,1)$. The fact that $g$ is smooth from $H^s \cap \mathcal{U}$ to $H^s$ for all $s\geq s_2$ is just a consequence of the algebra property of the Sobolev spaces (because we chose $s_2 > d/2-1$).  Due to \eqref{def-g}, the tame estimate \eqref{eq:lestimee_tame_pas_facile_a_pas_oublier} will be a consequence of the following one 
\begin{equation}\label{tame-u}
\forall s>d/2,\forall \Psi \in H^s(\mathcal{M};\mathbb{R}) \quad \|f(\cdot,\Psi)\|_{H^s(\mathcal{M})}\lesssim_{s,\|  \Psi  \|_{L^\infty}} \| \Psi \|_{H^s(\mathcal{M})}
\end{equation}
where $f$ is given in \eqref{eq:KG} and we thus refer to \cite[Thm 2.87 page 104]{BCD11} for a proof. Finally, for the Hamiltonian structure, it suffices to note that for all $j\in \mathbb{N}$, $g_j = \partial_{\overline{u_j}} G$ where
$$
G(u) := \int_{\mathcal{M}} F\Big(x,\Lambda_m^{-1/2} \frac{u(x) + \overline{u(x)}}2\Big) \mathrm{d}x
$$
where $F\in C^\infty(\mathbb{R}\times \mathbb{R} ; \mathbb{R})$ is the function satisfying $F(\cdot,0) =0$ and $\partial_{\Psi} F = f$.

\medskip

Recalling that $(\lambda_j)_{j\in \mathbb{N}}$ are the eigenvalues of the operator $-\Delta+V$, the estimate \eqref{weyl} with $\beta :=d$ is a direct consequence of the Weyl law \eqref{eq:true_Weyl} for this operator.

\medskip

Finally, it only remains to construct the clustering and to check the associated properties. The construction relies on the following lemma.
\begin{lemma}
\label{lem:seq_c}
There exists an increasing sequence $(c_k)_{k\in \mathbb{N}} \in \mathbb{R}_+^{\mathbb{N}}$ such that 
$$
\sigma(\sqrt{-\Delta+V}) = \{ \lambda_j \ | \ j\in \mathbb{N} \} \subset \bigcup_{k\in \mathbb{N}} [c_{2k-1},c_{2k})
$$
and
\begin{equation}
\label{eq:bound_int_c}
\forall k \in \mathbb{N}, \quad   k^{-d} \lesssim |c_{2k+1}-c_{2k}|  \quad \mathrm{and} \quad |c_{2k}-k|+|c_{2k}-c_{2k-1}| \lesssim 1.
\end{equation}
\end{lemma}
\begin{proof}
Recalling that by the Weyl law there exists $C\in \mathbb{N}$ such that for all $k \geq 1$
$$
\# \{ j \in \mathbb{N} \ | \ \lambda_j\leq k \} \leq C k^d
$$
Now split the interval $\big[k-\frac{1}{2},k\big)$ into $2Ck^d$ subintervals of the same length (that is $\frac{1}{4Ck^d}$) and of the form 
$\big[ k-\frac{\ell-1}{4Ck^d}   , k-\frac{\ell}{4Ck^d}    \big)$ with $\ell$ being an integer.
 Due to the last inequality, there is at least one subinterval that contains no eigenvalue and we may define it to be at a number $\ell=\ell_k$.
Thus, it suffices to set $c_1 = 0$ and for all $k\in \mathbb{N}$
$$
c_{2k} :=  k - \frac{\ell_k-1}{4 C k^d} \quad \mathrm{and} \quad c_{2k+1} :=k - \frac{\ell_k}{4C k^d} .
$$
In particular, we get the inequalities $\frac{1}{2}\leq c_{2k+2}-c_{2k+1}\leq \frac{3}{2}$ as illustrated in the figure below:
\begin{center}
\includegraphics[scale=0.5]{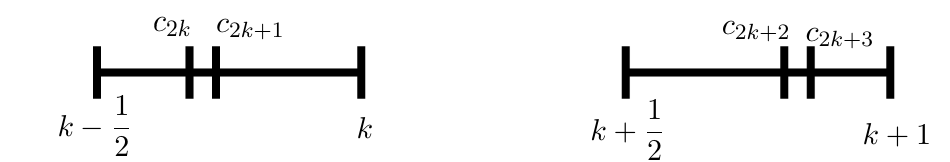}
\end{center}
\end{proof}
Then we set
$$
\mathcal{C}_k = \{ j\in \mathbb{N} \ | \ c_{2k-1} \leq \lambda_j \leq c_{2k} \}.
$$
and, recalling that $\omega_j=\sqrt{\lambda_j^2+m}$, we note that thanks to the second estimate of \eqref{eq:bound_int_c}, the bound \eqref{eq:inclusion_clusters} on the frequencies and the clusters holds with $\alpha=\Upsilon=1$. The non-resonance estimate \eqref{eq:nr_thm}, was proved in \cite{DS04} (see also \cite{DI17}). Here we do not repeat the proof, we just emphasize that it is based on the fact that the frequencies are analytic functions of
$m$ and of the eigenvalues $\lambda_j$ and furthermore the intervals actually containing frequencies are separated among them by a quantity that goes to zero not faster than $k^{-n}$ with some $n$ (given by \eqref{eq:bound_int_c} here).

\medskip

The proof of the multilinear estimate \eqref{eq:estimates_thm} is
done in the appendix. The
proof consists of first proving an estimate of the kind of \eqref{eq:DS-estimates} and then deducing the estimate
\eqref{eq:estimates_thm}. We emphasize that the estimate \eqref{eq:DS-estimates} was essentially proved in \cite{DS06}. Here we
give a slightly simplified version of its proof which does not make use of the Helffer--Sj\"ostrand formula.

\medskip

Finally, due to constants depending on $s$, the assumption $\| \Psi_0  \|_{H^{s+1}}+\| \Phi_0  \|_{H^{s}} \leq 1$ does not necessarily implies that $\|u(0)\|_{h^{s/d}} \leq 1$. Nevertheless, since $\|u(0)\|_{h^{s_1/d}}$ is small, it suffices to choose $s$ a little larger in Theorem \ref{thm:kg} and to interpolate with the $h^{s_1/d}$ norm (i.e. to apply H\"older's inequality) to get the smallness condition  of Theorem \ref{thm:main} on the high regularity norm.

\medskip

Therefore all the hypothesis of Theorem \ref{thm:main} are satisfied
and, when we apply it, we obtain Theorem \ref{thm:kg}. Similarly, applying Corollary \ref{cor:smooth}, we get Corollary \ref{cor:kg}.

\subsection{Schr\"odinger} Let $\mathcal{M}$ be a smooth compact  connected boundaryless Riemannian manifold of dimension $d\geq 1$ and $s_2> d/2$.
In order to prove Theorem \ref{thm:nls}, we have first to explain how, thanks to the Gauge symmetry, \eqref{eq:NLS} rewrites in the framework of Theorem \ref{thm:main} and then why it satisfies its assumptions. These explanations are similar to those given in \cite{FGL13,BFLM24}. To avoid normalization constants, we assume without loss of generality that 
$$
\mathrm{Vol}(\mathcal{M}) :=\int_{\mathcal{M}} 1 \, \mathrm{d}x = 1.
$$

\medskip

\noindent \underline{\emph{Step 1: Formalism.}} As previously, we denote by $(\lambda_j^2)_{j\in \mathbb{N}}$ the non-decreasing sequence of eigenvalues of the Laplace--Beltrami operator $-\Delta$ acting on $L^2(\mathcal{M};\mathbb{R})$ and by $(e_j)_{j\in \mathbb{N}}$ an associated real Hilbertian basis such that $e_1 = 1$. We denote by $ L^2_0(\mathcal{M};\mathbb{C}) $ the space of the zero integral functions in $L^2$, i.e. 
	\begin{equation*}
		L^2_0( \mathcal{M};\mathbb{C}):=\Big\{v\in L^2(\mathcal{M};\mathbb{C})\ \big| \ \int_{\mathcal{M}} v(x) \, \mathrm{d}x=0\Big\}.
	\end{equation*}
	  We identify any function $u\in L^2_0(\mathcal{M};\mathbb{C})$ with its sequence of coefficients $(u_j)_{j\in \mathbb{N}} \in \ell^2$, i.e.
	  $$
	  u = \sum_{j\in \mathbb{N}} u_j e_{j+1}.
	  $$ 
It follows that by Lemma \ref{eq:norm_equiv}, we have
$$
\forall s\geq 0, \quad H^s_0(\mathcal{M};\mathbb{C}) := H^s(\mathcal{M};\mathbb{C}) \cap L^2_0(\mathcal{M};\mathbb{C})  \equiv h^{s/d} 
$$
and that the associated norms are equivalent.

\medskip

Now, we summarize in the following proposition how the construction of Faou--Glauckler--Lubich in  \cite{FGL13} (which is also recalled in \cite{BFLM24}) reduces the analysis of \eqref{eq:NLS} close to a plane wave to the analysis of an equation of the form \eqref{eq:ham-pde}.
\begin{proposition}[Faou--Glauckler--Lubich \cite{FGL13}] For all $p_0>0$ satisfying $2p_0f'(p_0)> \lambda_2^2$, there exist a function $K \in C^\infty( \Omega ; \mathbb{R})$ defined on an open neighborhood $\Omega$ of the origin in $ \mathbb{C} \times \mathbb{R}$ and a $\mathbb{R}$-linear isomorphism\footnote{which is also actually a symplectomorphism.} $S : \ell^2 \to \ell^2 $ which is diagonal in the sense that 
\begin{equation}
\label{eq:S_is_diag}
\forall i\geq 2, \quad S \mathbb{C} e_i = \mathbb{C} e_i
\end{equation}
 such that setting, for all $j\geq 1$,
\begin{itemize}
\item  $\omega_ j := \sqrt{\lambda_{j+1}^4 + 2p_0f'(p_0)\lambda_{j+1}^2}$
\item  $g_j := \partial_{\overline{u}_j} G$ where $G$ denotes the function defined on a neighborhood of the origin in $H^{s_1}_0$ by
\begin{equation}
\label{eq:joli_forme}
G(u) = \int_{\mathcal{M}} K( S u(x), \| S u\|_{L^2}^2 ) \, \mathrm{d}x,
\end{equation}
\end{itemize}
we have the following property. For any $T>0$ and any solution $u \in C^0([0,T]; H^{s_2}_0)$ to \eqref{eq:ham-pde}, there exists a continuous function $\theta \in  C^0([0,T]; \mathbb{R})$ such that
\begin{equation}
\label{eq:z_de_u}
z := e^{- \ic\theta}\left(\sqrt{p_0-\left\|S u\right\|_{L^2}^2}+ Su \right) \in C^0([0,T]; H^{s_2}) \quad \mathrm{is \ a\  solution\ to\ } \eqref{eq:NLS}.
\end{equation}
\end{proposition}
\begin{remark} The proof is fully constructive. In particular there are explicit formulas for $K,S$ and $\theta$. However, they are not useful for us. To explain why $S$ enjoys these good properties let us just mention that it satisfies 
$$
\forall i\geq 2,\forall a,b\in \mathbb{R}, \quad S(a+ib)e_i = (\alpha_ia  + i \alpha_i^{-1} b )e_i \quad \mathrm{where} \quad \alpha_i \sim 1.
$$
The proof in \cite{FGL13} is done in $\mathbb{T}^d$ with a cubic nonlinearity but, as explained in \cite{BFLM24}, there is no difficulty to extend it in our setting. The key point is to note that the map $(p,\theta,u)\mapsto z$ defined by \eqref{eq:z_de_u} is symplectic (in some classical sense not recalled here).
\end{remark}
\begin{remark}
Since $S$ is diagonal, it is also an isomorphism from $H^s_0$ to $H^s_0$ for all $s\in \mathbb{R}$. 
\end{remark}
\begin{remark} Let us note that $G$  is a well defined smooth function on a bounded neighborhood of the origin in $H^{s_2}_0$ and so that $g$ is well defined. Indeed, it is a consequence of the fact that $S$ is continuous on $H^{s_2}$ and of the algebra property of $H^{s_2}$ (because $s_2>d/2$).
\end{remark}

Then it suffices to note that if $u$ is small enough in $L^2$ and $z$ is given by \eqref{eq:z_de_u} then for all $s\geq 0$ we have
$$
\| u\|_{H^s} \sim_{s,p_0} \inf_{\varphi \in \mathbb{T}} \| e^{i \varphi} \sqrt{p_0} - z \|_{H^s}
$$
to deduce that if \eqref{eq:ham-pde} is almost globally well posed (in the sense that Theorem \ref{thm:main} applies) then Theorem \ref{thm:nls} holds.

\medskip

\noindent \underline{\emph{Step 2: Validity of the assumptions.}} Most of the discussions are the same as the ones we presented for Klein--Gordon, so we just explain briefly the constructions.	The open set $\mathcal{U}$ can be chosen as a centered open ball in $H^{s_2}_0$ of sufficiently small radius. The fact that $g$ is smooth from $H^s_0 \cap \mathcal{U}$ to $H^s_0$ for all $s\geq s_2$ is just a consequence of the algebra property of the Sobolev spaces and of the continuity of $S$ in $H^s_0$. As previously, the tame estimate \eqref{eq:lestimee_tame_pas_facile_a_pas_oublier} relies on para-differential calculus techniques and we refer to \cite[Thm 2.87 page 104]{BCD11} for a proof. 
Recalling that $(\lambda_j)_{j\in \mathbb{N}}$ are the eigenvalues of the operator $-\Delta$, the estimate \eqref{weyl} with $\beta :=d/2$ is a direct consequence of the Weyl law \eqref{eq:true_Weyl} for this operator.

\medskip

Now we focus on the clustering property. Let $(c_k)_{k\geq 1}$ be the sequence given by Lemma \ref{lem:seq_c} in the case $V=0$. As previously we define the cluster by
$$
\mathcal{C}_k = \{ j\in \mathbb{N} \ | \ c_{2k-1} \leq \lambda_{j+1} \leq c_{2k} \}.
$$
We note that by construction of the sequence $(c_k)_k$ the cluster satisfies the bound \eqref{eq:inclusion_clusters} with $\alpha := 2$ and $\Upsilon=1$.
The nonresonance condition \eqref{eq:nr_thm} was proved (following \cite{DS06}) in \cite{BFLM24}, see Section 7.2 and in particular Lemma 7.5 (using $ 2p_0f'(p_0)$ as parameter).

\medskip

We come to the multilinear estimate \eqref{eq:estimates_thm}.  We have
to prove that if one considers the Taylor expansion of a functions of
the form \eqref{eq:joli_forme}, each Taylor polynomial fulfills the
estimate \eqref{eq:estimates_thm}. The spectral multiplier $S$ being diagonal (in the sense of \eqref{eq:S_is_diag}) and bounded it can be ignored in this discussion. 
The proof is almost the same as for Klein--Gordon. The only novelty is the extra dependence of $G$ with respect to the $L^2$ norm. We prove in the following lemma that this dependence is not an obstruction.

\begin{lemma}
  \label{gro.lem}
  Let $P:\ell^2 \to \mathbb{R}$ be a locally bounded $\mathbb{R}-$homogeneous polynomial of degree $q\geq 3$ fulfilling the estimate
\eqref{eq:estimates_thm}. Then  $Q:=P \| \cdot\|_{L^2}^2$  fulfills the estimate
\eqref{eq:estimates_thm}.  
\end{lemma}
\begin{proof} The function $Q$ being a homogeneous polynomial of degree $q+2$ it suffices to consider $d^{q+2}Q(0)$ which is equal to
$$
d^{q+2}Q(0)(v^{(1)},...,v^{(q+2)})=\sum_{\tau\in \mathfrak{S}_{q+2}}
\mathrm{d}^{q}P(0)(v^{(\tau(1))},...,v^{(\tau(q))})\int_{\mathcal{M}}v^{(\tau(q+1))}\bar
v^{(\tau(q+2))} \mathrm{d}x
$$
with $v^{(1)},\cdots, v^{(q+2)}\in \ell^2$ and $\mathfrak{S}_{q+2}$ the group of the permutations of the first $q+2$ integers.
We analyze only the identical permutation, the other being equal.

Denote $(\boldsymbol{h}_1,...,\boldsymbol{h}_q):=(\boldsymbol{k}_1,...,\boldsymbol{k}_{q})$, then, if $v^{(j)}\in E_{\boldsymbol{k}_j}$ for all $j\in \{1,\cdots,q+2\}$ one has
$$
\left|\mathrm{d}^{q}P(0)(v^{(1)},\cdots,v^{(q)})\int_{\mathcal{M}}v^{(q+1)}\bar
v^{(q+2)} \mathrm{d}x \right|
\lesssim \mathds{1}_{  \boldsymbol{k}_{q+1}=\boldsymbol{k}_{q+2}}\Gamma_{\boldsymbol{h}}\left(\frac{\boldsymbol{h}^{\star}_2}{\boldsymbol{h}^{\star}_1}\right)^n
\prod_{l=3}^{q}(\boldsymbol{h}_l^\star)^\nu
\prod_{\ell=1}^{q+2} \|v^{(\ell)}\|_{\ell^2} \ , 
$$
so it enough to show that
$$
 \mathds{1}_{  \boldsymbol{k}_{q+1}=\boldsymbol{k}_{q+2}}\Gamma_{\boldsymbol{h}}\left(\frac{\boldsymbol{h}^{\star}_2}{\boldsymbol{h}^{\star}_1}\right)^n
\prod_{l=3}^{q}(\boldsymbol{h}_l^\star)^\nu
\lesssim \Gamma_{\boldsymbol{k}}\left(\frac{\boldsymbol{k}^{\star}_2}{\boldsymbol{k}^{\star}_1}\right)^n
\prod_{l=3}^{q+2}(\boldsymbol{k}_l^\star)^\nu\ .
$$
In turn, since $\Gamma_{\boldsymbol{k}}\geq \Gamma_{\boldsymbol{h}}$ this is implied by
\begin{equation}
  \label{da_dim}
 \mathds{1}_{  \boldsymbol{k}_{q+1}=\boldsymbol{k}_{q+2}} \left(\frac{\boldsymbol{h}^{\star}_2}{\boldsymbol{h}^{\star}_1}\right)^n
\prod_{l=3}^{q}(\boldsymbol{h}_l^\star)^\nu\lesssim \left(\frac{\boldsymbol{h}^{\star}_2}{\boldsymbol{h}^{\star}_1}\right)^n
\prod_{l=3}^{q}(\boldsymbol{h}_l^\star)^\nu\ ,
\end{equation}
which is the one we now prove. To this end we denote (just for this
prove) by $K$ the l.h.s. of \eqref{da_dim}. We distinguish three cases.
\begin{itemize}
\item \underline{\emph{Case 1:  $\boldsymbol{k}_{q+1}\geq \boldsymbol{h}^\star_1$.}}
In this case we have (since $\boldsymbol{k}_{q+1}=\boldsymbol{k}_{q+2}$)
$$
K\leq \prod_{l=3}^{q}(\boldsymbol{h}_l^\star)^\nu\leq
\left(\frac{\boldsymbol{k}_{q+2}}{\boldsymbol{k}_{q+1}}\right)^n \prod_{l=1}^{q}(\boldsymbol{k}_l^\star)^\nu\ , 
$$
which is the wanted estimate.

\item \underline{\emph{Case 2:  $\boldsymbol{h}_1^\star>\boldsymbol{k}_{q+1}\geq \boldsymbol{h}_2^\star$.}} Here we have
  $$
K\leq \left(\frac{\boldsymbol{k}_{q+1}}{\boldsymbol{h}^\star_1}\right)^n
\prod_{l=3}^{q}(\boldsymbol{h}_l^\star)^\nu< \left(\frac{\boldsymbol{k}_{q+1}}{\boldsymbol{k}^\star_1}\right)^n
\prod_{l=3}^{q}(\boldsymbol{h}_l^\star)^\nu \boldsymbol{k}_{q+2}^\nu 
  $$
which is the wanted estimate.

\item \underline{\emph{Case 3:  $\boldsymbol{h}_2^\star>\boldsymbol{k}_{q+1}$.}}    In this case the estimate trivially holds.
\end{itemize}

\end{proof}

\subsection{Klein-Gordon equation on $\mathbb{R}^d$ with positive definite quadratic potential}\label{harm-sub}

We now briefly explain the proof of Theorem \ref{thm:quatumoscillator} by emphasizing the differences with the compact case.
Here, the Sobolev space naturally associated to the potential $Q$ for any $s\geq 0$ is defined as follows:
\begin{equation}
\mathcal{H}_Q^s(\mathbb{R}^d):=\mbox{Dom}((-\Delta+Q)^{\frac{s}{2}})=\{ u\in L^2(\mathbb{R}^d), \quad (-\Delta+Q)^{\frac{s}{2}}u\in L^2(\mathbb{R}^d)            \}.
\end{equation}
It is well-known that the norms
  \begin{equation}\label{equivHS}
\left\|(-\Delta+Q)^{s/2}u\right\|_{L^2}\qquad \text{and}\qquad
\left\|(1-\Delta)^{s/2}u\right\|_{L^2}+\left\| Q^{s/2}u \right\|_{L^2}\  
\end{equation}
are equivalent (see e.g. \cite{YZ04}) and
 \begin{equation*}
 \mathcal{H}_Q^s(\mathbb{R}^d)=\{u\in H^s(\mathbb{R}^d) \quad | \quad Q^{\frac{s}{2}}u\in L^2(\mathbb{R}^d)\}.
 \end{equation*}
 Since $Q$ is $2$-homogeneous, we clearly find $\mathcal{H}_Q^s(\mathbb{R}^d)=\mathcal{H}^s(\mathbb{R}^d)$ defined in \eqref{harm-sob}.

As above, we denote by $(\lambda_j^2)_{j\in \mathbb{N}}$ the nondecreasing sequence (with multiplicities) of the eigenvalues of $-\Delta+Q$ and by 
$(-\Delta+Q)e_j=\lambda_j^2 e_j$ the corresponding relation. The Weyl law here reads (see Theorem XIII.81 of \cite{RS80}):
$$
\forall y \gtrsim 1,\qquad \# \{j\in \mathbb{N} \ | \ \lambda_j \leq y \} \sim  \int_{Q(x)\leq y^2} (y^2-Q(x))^{\frac{d}{2}} dx    \sim  y^{2d}
$$
and in particular $\lambda_j\sim j^{\frac{1}{2d}}$.
Then we may repeat the same argumentation as in Subsection \ref{KG_proof} upon making a few modifications. For instance, once we have identified a $\mathbb{K}\in \{\mathbb{R},\mathbb{C}\}$ valued function $u=\sum_{j\in \mathbb{N}} u_j e_j$ with the sequence of coefficients $(u_j)_{j\in \mathbb{N}}\in \mathbb{K}^\mathbb{N}$, Lemma \ref{eq:norm_equiv} now reads
\begin{equation*}
\mathcal{H}^s(\mathbb{R}^d;\mathbb{K})=h^{\frac{s}{2d}}\cap \mathbb{K}^{\mathbb{N}}.
\end{equation*}
Let us now turn to the tame estimate \eqref{eq:lestimee_tame_pas_facile_a_pas_oublier} which, as above, relies on an analogue of \eqref{tame-u}.
Thanks to \eqref{equivHS}, it is sufficient to prove the following inequality for any smooth function $f$ vanishing at the origin and any $u\in \mathcal{H}^s(\mathbb{R}^d)$ for $s>\frac{d}{2}$:
\begin{equation*}
\|f(u)\|_{H^s(\mathbb{R}^d)}^2 +\int_{\mathbb{R}^d} \langle x\rangle^{2s} |f(u(x))|^2 \mathrm{d}x 
\lesssim_{s,f,\|u\|_{L^\infty}} \|u\|_{H^s(\mathbb{R}^d)}^2 + \int_{\mathbb{R}^d} \langle x\rangle^{2s} |u(x)|^2 \mathrm{d}x
\end{equation*}
in which we recall that the constant may depend on $s,f$ and $\|u\|_{L^\infty}$. The bound on $\|f(u)\|_{H^s(\mathbb{R}^d)}$ is already done in \cite[Thm 2.87 page 104]{BCD11} as used above. In order to bound the second term, we use the embedding $\mathcal{H}^s(\mathbb{R}^d)\subset L^\infty(\mathbb{R}^d)$ and we set 
\begin{equation*}
C:=\sup\limits_{|t|\leq \|u\|_{L^\infty(\mathbb{R}^d)}} |f'(t)|,
\end{equation*}
then we immediately get
\begin{equation*}
\int_{\mathbb{R}^d} \langle x\rangle^{2s} |f(u(x))|^2 \mathrm{d}x 
\leq C^2  \int_{\mathbb{R}^d} \langle x\rangle^{2s} |u(x)|^2 \mathrm{d}x.
\end{equation*}
The proof of the multilinear estimate \eqref{eq:estimates_thm} is done in the appendix. It mainly relies on multilinear estimates proved in \cite{Brun23}. The rest of the proof of Theorem \ref{thm:quatumoscillator} is similar to the one of Theorem \ref{thm:kg} done in Subsection \ref{KG_proof}.

\section{Functional setting} From now and until the end of this paper, we fix a non-decreasing sequence of frequencies $\omega \in (\mathbb{R}_+^*)^{\mathbb{N}}$ satisfying the Weyl law \eqref{weyl} and a cluster decomposition 
$$
\mathbb{N} = \bigsqcup_{k\in \mathbb{N}} \mathcal{C}_k
$$
 satisfying the bound \eqref{eq:inclusion_clusters}. We recall that the spaces $E_k$, $k\in \mathbb{N}$ are defined by
$$
E_k := \mathrm{Span}_{\mathbb{C}} \{ \mathds{1}_{\{j\}} \ | \ j\in \mathcal{C}_k \}.
$$
In this section, first we introduce some basic notations and definitions. Then, we introduce a class of polynomials and prove the associated multilinear estimates. Finally, we prove some estimates on the Hamiltonian flows generated by these polynomials.

\subsection{Formalism}

\begin{definition}[$\ell^2$ scalar product] We endow $\ell^2$ of the real scalar product
$$
(u,v)_{\ell^2} := \Re \sum_{k\in \mathbb{N}} u_k \overline{v_k}
$$
and of the symplectic form $(\ic \cdot,\cdot)_{\ell^2}$.
\end{definition}

\begin{definition}[Projections $\Pi_k$] For all $k\in \mathbb{N}$ and all $u\in \mathbb{C}^\mathbb{N}$, we define $\Pi_k u \in \mathbb{C}^\mathbb{N}$ by
$$
\forall j \in \mathbb{N}, \quad (\Pi_k u)_j = \mathds{1}_{j\in \mathcal{C}_k} u_j.
$$
Moreover, we set
$$
\Pi_k^{-1} u := \overline{\Pi_k u},
$$
and for all $N \in \mathbb{N}$,
$$
\Pi_{\leq N} = \sum_{k\leq N}\Pi_{k} \quad \mathrm{and} \quad \Pi_{>N} = \mathrm{Id} - \Pi_{\leq N}.
$$ 
\end{definition}
\begin{remark}
$\Pi_k$ is nothing but the orthogonal projection on $E_k$  in $\ell^2$.
\end{remark}

\begin{definition}[Super-actions $J_k$] For all $k\in \mathbb{N}$ and $u\in \mathbb{C}^\mathbb{N}$, we set
$$
J_k(u) := \| \Pi_k u \|_{\ell^2}^2 = \sum_{j\in \mathcal{C}_k} |u_j|^2.
$$
\end{definition}

\begin{definition}[$H^s$ spaces] 
\label{def:Hs_space}
For all $s\in \mathbb{R}$, we set
$$
H^s := \big\{ u\in \mathbb{C}^\mathbb{N} \ | \ \| u\|_{H^s}^2 := \sum_{k \in \mathbb{N}}  k^{2s} J_k(u) < \infty \big\}.
$$
For all $\varepsilon$, we denote by $B_{H^s}(0,\varepsilon)$ the open ball of $H^s$ of center $0$ and radius $\varepsilon$.
\end{definition}
Then we note in the following lemma that  thanks to the Weyl law (see \eqref{weyl}), the Sobolev spaces $H^s$ are equivalent to the discrete Sobolev spaces $h^s$ we used to state our abstract result. Therefore, from now, we only work with $H^s$ norms.
\begin{lemma} For all $s\geq 0$, we have
$$
H^s = h^{s/\alpha \beta} \quad \mathrm{and} \quad \| \cdot \|_{h^{s/\alpha \beta}} \sim_s \| \cdot \|_{H^s}.
$$
\end{lemma}
\begin{proof} First, we note that thanks to the bound \eqref{eq:lestimee_tame_pas_facile_a_pas_oublier}, if $j\in \mathcal{C}_k$ we have $\omega_j\sim k^{\alpha}$. Therefore, we have, for all $u\in \mathbb{C}^\mathbb{N}$
$$
\| u \|_{H^s}^2 = \sum_{k\in \mathbb{N}} \sum_{j\in \mathcal{C}_k} k^{2s} |u_j|^2 \sim_s  \sum_{j\in \mathbb{N}}  \omega_j ^{2s/\alpha} |u_j|^2.
$$
Finally, to conclude the proof, it suffices to note that, thanks to the Weyl law \eqref{weyl}, we have $\langle \omega_j \rangle \sim j^{1/\beta}$ (it suffices to consider $\lambda = \omega_j$ and $\lambda= \omega_j^-$ in \eqref{weyl} and to use that $\omega$ is non decreasing).
\end{proof}

%\begin{remark}
%{  We note that this is the norm that in Subsect. \ref{KG_proof} was
%  denoted by $\left\| .\right\|_{\tilde H^s}$ and was proved to be
%  equivalent to the standard $H^s$ norm. From now on we omit the
%  tilde from $H$.}
%
%  Thanks to the polynomial grows of the frequencies (see \eqref{weyl}), one has
%$$
%H^s = h^{s/\alpha\beta} \quad \mathrm{and} \quad \| \cdot \|_{h^{s/\alpha\beta}} \sim_s \| \cdot \|_{H^s}.
%$$
%So, from now, we only work with $H^s$ norms.
%\end{remark}

\begin{definition}[Gradient $\nabla$] Given $s\geq 0$, $\mathcal{V}$ an open subset of $H^s$ and $P \in C^1(\mathcal{V};\mathcal{R})$, $\nabla P : \mathcal{U} \to H^{-s}$ denotes the unique function satisfying 
$$
\forall u\in \mathcal{V}, \forall v\in H^s, \quad (\nabla P (u),v)_{\ell^2} = \mathrm{d}P(u)(v).
$$
\end{definition}
\begin{remark}
As usual, for all $j\in \mathbb{N}$, one has $(\nabla P (u))_j = 2 \partial_{\overline{u_j}} P(u)$.
\end{remark}

\begin{definition}[Symplectic map] Given $s\geq 0$ and $\mathcal{V}$ an open subset of $H^s$, a map $\tau \in C^1(\mathcal{V};h^s)$ is \emph{symplectic} if
$$
\forall u\in \mathcal{V}, \forall v,w\in H^s, \quad (\ic \mathrm{d}\tau(u)(v),\mathrm{d}\tau(u)(w))_{\ell^2} = (\ic v,w)_{\ell^2}.
$$
\end{definition}

\begin{definition}[Poisson bracket] Let $s\geq 0$, $\mathcal{V}$ an open subset of $H^s$, and two functions $P,Q \in C^1(\mathcal{V}; \mathbb{R})$ such that $\nabla P$ or $\nabla Q$ is $H^s$ valued. The \emph{Poisson bracket} of $P,Q$ is defined by
$$
\{ P,Q \} := (\ic \nabla P, \nabla Q)_{\ell^2}.
$$
\end{definition}
\begin{remark} As usual, one has
$$
\{ P,Q \} =2\ic  \sum_{j \in \mathbb{N}} \partial_{\overline{u_j}} P \partial_{u_j} Q -  \partial_{u_j} P \partial_{\overline{u_j}}  Q.
$$

\end{remark}

\begin{definition}[Multlinear forms $\mathscr{L}_{\boldsymbol{k}}$] Given $q\geq 3$ and $\boldsymbol{k} \in \mathbb{N}^q$, $\mathscr{L}_{\boldsymbol{k}}$ denote the space of the $\mathbb{C}$ multilinear forms on $E_{\boldsymbol{k}_1} \times \cdots E_{\boldsymbol{k}_q}$. Moreover, we endow $\mathscr{L}_{\boldsymbol{k}}$ of its canonical norm
$$
\forall M \in \mathscr{L}_{\boldsymbol{k}}, \quad \| M \|_{\mathscr{L}_{\boldsymbol{k}}} := \sup_{ \substack{u^{(1)} \in E_{\boldsymbol{k}_1}\\ \| u^{(1)}  \|_{\ell^2}\leq 1}} \cdots \sup_{ \substack{u^{(q)} \in E_{\boldsymbol{k}_q}\\ \| u^{(q)}  \|_{\ell^2}\leq 1}} |M(u^{(1)},\cdots, u^{(q)})|.
$$
\end{definition}

\subsection{Homogeneous polynomials} 

\begin{definition}[Spaces of homogeneous polynomials $\mathscr{H}_{q}^{\nu,n}$] Given $q\geq 3$, $\nu ,n\geq 0$, $\mathscr{H}_{q}^{\nu,n}$ denotes the space of the formal real valued homogeneous polynomial  $P$ of degree $q$ on $\mathbb{C}^\mathbb{N}$ of the form
\begin{equation}
\label{eq:series_P}
P(u) = \sum_{\boldsymbol{k}\in \mathbb{N}^q} \sum_{\boldsymbol{\sigma} \in \{-1 , 1\}^q } P_{\boldsymbol{k}}^{\boldsymbol{\sigma}}(\Pi_{\boldsymbol{k}_1}^{\boldsymbol{\sigma}_1} u, \cdots, \Pi_{\boldsymbol{k}_q}^{\boldsymbol{\sigma}_q} u)
\end{equation}
where $P_{\boldsymbol{k}}^{\boldsymbol{\sigma}} \in \mathscr{L}_{\boldsymbol{k}}$ satisfies the symmetry condition
$$
 P_{\boldsymbol{k}}^{\boldsymbol{\sigma}}(u^{(1)},\cdots,u^{(q)}) = P_{\varphi \boldsymbol{k}}^{ \varphi \boldsymbol{\sigma}} ( u^{(\varphi_1)},\cdots,u^{(\varphi_q)} )
$$
 for all permutation $\varphi$ of $\{1,\cdots, q\}$ and all $(u^{(\ell)})_{1\leq \ell \leq q} \in \prod_{1\leq \ell \leq q} E_{\boldsymbol{k}_{\ell}}$,
 the reality condition $P_{\boldsymbol{k}}^{\boldsymbol{\sigma}}  = \overline{P_{\boldsymbol{k}}^{-\boldsymbol{\sigma}}(\overline{\cdot},\cdots, \overline{\cdot})}  $ and the bound
 $$
\| P \|_{\mathscr{H}_{q}^{\nu,n}} := \sup_{\boldsymbol{k}\in \mathbb{N}^q} \sup_{\boldsymbol{\sigma}\in \{-1 , 1\}^q }  \left[ \Gamma_{\boldsymbol{k}} \Big(   \frac{\boldsymbol{k}_2^\star }{ \boldsymbol{k}_1^\star } \Big)^n \prod_{3\leq \ell \leq q} (\boldsymbol{k}_\ell^\star)^\nu \right]^{-1}  \| P_{\boldsymbol{k}}^{\boldsymbol{\sigma}} \|_{\mathscr{L}_{\boldsymbol{k}}} < \infty
 $$
 with $\Gamma_{\boldsymbol{k}}$ defined by \eqref{eq:def_gamma_k}.
\end{definition}

First, we note that formal polynomials are polynomials functions on some Sobolev spaces.
\begin{lemma} \label{lem:pointwise} Given $q\geq 3$, $\nu,n \geq 0$, all formal polynomial $P \in \mathscr{H}_{q}^{\nu,n}$ defines a $C^\infty$ real valued function on $H^{1+\nu}$. 
\end{lemma}
\begin{proof} It suffices to note that 
$$
\| P_{\boldsymbol{k}}^{\boldsymbol{\sigma}} \|_{\mathscr{L}_{\boldsymbol{k}}} \leq \| P \|_{\mathscr{H}_{q}^{\nu,n}} (\boldsymbol{k}_{3}^\star)^\nu \cdots (\boldsymbol{k}_{q}^\star)^\nu
$$ 
to get that for all $u\in H^{1+\nu}$
\begin{equation*}
\begin{split}
\sum_{\boldsymbol{k}\in \mathbb{N}^q} \sum_{\boldsymbol{\sigma} \in \{-1 , 1\}^q } | P_{\boldsymbol{k}}^{\boldsymbol{\sigma}}(\Pi_{\boldsymbol{k}_1}^{\boldsymbol{\sigma}_1} u, \cdots, \Pi_{\boldsymbol{k}_q}^{\boldsymbol{\sigma}_q} u) |
& \lesssim_q \| P \|_{\mathscr{H}_{q}^{\nu,n}} \sum_{\boldsymbol{k}\in \mathbb{N}^q}    (\boldsymbol{k}_{3}^\star)^\nu \cdots (\boldsymbol{k}_{q}^\star)^\nu \prod_{i=1}^q  \|\Pi_{\boldsymbol{k}_i} u\|_{\ell^2} \\
& \lesssim_q \| P \|_{\mathscr{H}_{q}^{\nu,n}} \big(\sum_{k \in \mathbb{N}} k^\nu\|\Pi_{k} u\|_{\ell^2} \big)^q \lesssim_q \| P \|_{\mathscr{H}_{q}^{\nu,n}} \| u\|_{H^{1+\nu}}^q .
\end{split}
\end{equation*}
Thus the series defining $P$ converge. The fact that $P$ is real valued is a direct consequence of the reality condition on its coefficients. Finally, since $P$ is a locally bounded homogeneous polynomial, it is smooth (see e.g. \cite{BS71}).

\end{proof}
\begin{remark}
Thanks to the symmetry condition, the polynomial functions characterize the formal polynomial: from now, we identity the formal polynomials with the associated function.
\end{remark}

Now, we prove vector field estimates.
\begin{lemma} \label{lem:vf} Given $q\geq 3$, $\nu,n \geq 0$ and $P \in \mathscr{H}_{q}^{\nu,n}$, for all $s\in [0,n]$, $\nabla P$ is a $q-1$ homogeneous polynomial from $H^s$ to $H^s$ satisfying
$$
\forall u\in H^s \cap H^{1+\nu}, \quad \| \nabla P(u)\|_{H^s} \lesssim_{q}   \| P \|_{\mathscr{H}_{q}^{\nu,n}} \| u \|_{H^s} \| u\|_{H^{1+\nu}}^{q-2}.
$$ 
\end{lemma}
\begin{proof} Without loss of generality, we assume that $\| P \|_{\mathscr{H}_{q}^{\nu,n}}=1$. First, we note that thanks to the symmetry condition
\begin{equation}
\label{eq:un_peu_algebre}
\forall u,v\in H^{1+\nu}, \quad \mathrm{d}P(u)(v) = q \sum_{\boldsymbol{k}\in \mathbb{N}^q} \sum_{\boldsymbol{\sigma} \in \{-1 , 1\}^q }   P_{\boldsymbol{k}}^{\boldsymbol{\sigma}}(\Pi_{\boldsymbol{k}_1}^{\boldsymbol{\sigma}_1} u, \cdots,\Pi_{\boldsymbol{k}_{q-1}}^{\boldsymbol{\sigma}_{q-1}} u, \Pi_{\boldsymbol{k}_q}^{\boldsymbol{\sigma}_q} v).
\end{equation}
Then we recall that by duality and density
\begin{equation}
\label{eq:jolie_id}
\| \nabla P(u)\|_{H^s} = \sup_{\substack{v \in H^{1+\nu}\\ \| v\|_{H^{-s}} =1 }} | (\nabla P(u),v)_{\ell^2}| = \sup_{\substack{v \in H^{1+\nu}\\ \| v\|_{H^{-s}} =1 }} | \mathrm{d}P(u)(v)|.
\end{equation}
So, we fix $u\in H^s \cap H^{1+\nu}$, $v \in H^{1+\nu}$ such that $\| v\|_{H^{-s}} =1$, we set $w := \sum_k k^{-2s} \Pi_k v$ and for all $p\in \mathbb{N}$ we set $p'=1 + \mathds{1}_{p=1}$. We apply the triangular inequality in \eqref{eq:un_peu_algebre} to get that (since $s\leq n$)
\begin{equation*}
\begin{split}
| \mathrm{d}P(u)(v)| &\lesssim_q    \sum_{\boldsymbol{k}\in \mathbb{N}^q}  \Gamma_{\boldsymbol{k}} \Big(   \frac{\boldsymbol{k}_2^\star }{ \boldsymbol{k}_1^\star } \Big)^n \big( \prod_{3\leq \ell \leq q} \boldsymbol{k}_\ell^\star \big)^\nu \boldsymbol{k}_q^{2s} \| \Pi_{\boldsymbol{k}_q} w \|_{\ell^2} \prod_{i=1}^{q-1} \| \Pi_{\boldsymbol{k}_i} u \|_{\ell^2} \\
&\lesssim_q \sum_{1\leq p \leq q} \sum_{ \boldsymbol{k}_1 \geq \cdots \geq \boldsymbol{k}_q  } \Gamma_{\boldsymbol{k}} \Big(   \frac{\boldsymbol{k}_2 }{ \boldsymbol{k}_1 } \Big)^n \big( \prod_{3\leq \ell \leq q} \boldsymbol{k}_\ell \big)^\nu \boldsymbol{k}_p^{2s} \| \Pi_{\boldsymbol{k}_p} w \|_{\ell^2} \prod_{i\neq p}\| \Pi_{\boldsymbol{k}_i} u \|_{\ell^2} \\
&\lesssim_q \sum_{1\leq p \leq q}\sum_{ \boldsymbol{k}_1 \geq \cdots \geq \boldsymbol{k}_q  } \Gamma_{\boldsymbol{k}} \big( \prod_{\ell \neq p,p'} \boldsymbol{k}_\ell \big)^\nu \boldsymbol{k}_p^s \boldsymbol{k}_{p'}^s \| \Pi_{\boldsymbol{k}_p} w \|_{\ell^2} \prod_{i\neq p}\| \Pi_{\boldsymbol{k}_i} u \|_{\ell^2} \\
&\lesssim_q  \sum_{\boldsymbol{\varsigma} \in \{-1,1\}^q}  \sum_{\boldsymbol{k}\in \mathbb{N}^q} \langle  \boldsymbol{\varsigma}_1 \boldsymbol{k}_1 + \cdots + \boldsymbol{\varsigma}_q \boldsymbol{k}_q \rangle^{-3} \big( \prod_{\ell \geq 3} \boldsymbol{k}_\ell \big)^\nu \boldsymbol{k}_1^s \boldsymbol{k}_{2}^s \| \Pi_{\boldsymbol{k}_1} w \|_{\ell^2} \prod_{i\geq 2}\| \Pi_{\boldsymbol{k}_i} u \|_{\ell^2} \\
&\lesssim_q \|w\|_{H^s} \| u\|_{H^s} \| u\|_{H^{1+\nu}}^{q-2}.
\end{split}
\end{equation*}
where the last estimate is just the Young convolution inequality.
Since $\|w\|_{H^s}=1$, we get the expected estimate by \eqref{eq:jolie_id}.

\end{proof}

\begin{remark} Since $\nabla P$ is locally bounded homogeneous polynomial from $H^s$ to $H^s$ (see e.g. \cite{BS71} for details about polynomials), it is $C^\infty$ and satisfies
\begin{equation}
\label{eq:bound_dP}
\forall u,v\in H^s, \quad\|  \mathrm{d}\nabla P(u)(v) \|_{H^s}\lesssim_{n,q} \| P \|_{\mathscr{H}_{q}^{\nu,n}}  \| u\|_{H^{s}}^{q-2} \|v\|_{H^{s}} .
\end{equation}
\end{remark}

Finally, we prove that these spaces of polynomials are stable by Poisson bracket.
\begin{proposition} \label{prop:poisson} Let $\nu\geq 0$, $n\geq \nu+1$, $q,q' \geq 3$. For all $P\in \mathscr{H}_q^{n,\nu}$ and $Q\in \mathscr{H}_{q'}^{n,\nu}$, their Poisson bracket $\{P,Q\}$ belongs to $\mathscr{H}_{q+q'-2}^{n,\nu}$ and satisfies
$$
\| \{P,Q\} \|_{\mathscr{H}_{q+q'-2}^{n,\nu}} \lesssim q q' \| P \|_{\mathscr{H}_{q}^{n,\nu}}\| Q\|_{\mathscr{H}_{q'}^{n,\nu}}.
$$
\end{proposition}
\begin{proof} We divide the proof in $3$ steps. Actually, the proof is done in the first step up to the technical estimates \eqref{eq:reste_a_faire} which are proved in the two last steps.

\medskip

\noindent \underline{\emph{$\triangleright$ Step $1$ : Core of the proof.}}
First, we note that, (since $n\geq \nu+1$) $\nabla P$ is smooth from $H^{1+\nu}$ into $H^{1+\nu}$ (see Lemma \ref{lem:vf}) and that $Q$ is smooth on  $H^{1+\nu}$ (see Lemma \ref{lem:pointwise}). Therefore $ \{P,Q\}$ is a well defined  function on $H^{1+\nu}$.

Then, we note that jor all $j\in \mathbb{N}$ and $\sigma \in \{1,1\}$, we have
$$
\partial_{u_j^\sigma} P(u) =  q\sum_{\boldsymbol{k}\in \mathbb{N}^{q-1}} \sum_{\boldsymbol{\sigma} \in \{-1 , 1\}^{q-1}  } P_{\boldsymbol{k},\underline{j}}^{\boldsymbol{\sigma},\sigma}(\Pi_{\boldsymbol{k}_1}^{\boldsymbol{\sigma}_1} u, \cdots, \Pi_{\boldsymbol{k}_{q-1}}^{\boldsymbol{\sigma}_{q-1}} u, \mathds{1}_{\{j\}})
$$
where $\underline{j}\in \mathbb{N}$ denotes the index such that $j \in \mathcal{C}_{\underline{j}}$.
It follows that for all $u\in H^{1+\nu}$, we have
\begin{equation}
\label{eq:rel1}
\{P,Q\}(u) = 2\ic q q'  \sum_{\boldsymbol{k}''\in \mathbb{N}^{q''}} \sum_{\boldsymbol{\sigma}'' \in \{-1 , 1\}^{q''}  } R^{\boldsymbol{\sigma}''}_{\boldsymbol{k}''}(\Pi_{\boldsymbol{k}_1}^{\boldsymbol{\sigma}_1} u,\cdots, \Pi_{\boldsymbol{k}_{q''}}^{\boldsymbol{\sigma}_{q''}} u)
\end{equation}
where $q''=q+q'-2$ and decomposing $\boldsymbol{k}''=(\boldsymbol{k},\boldsymbol{k}')\in \mathbb{N}^{q-1}\times\mathbb{N}^{q'-1} $, $\boldsymbol{\sigma}''=(\boldsymbol{\sigma},\boldsymbol{\sigma}')\in \{-1,1\}^{q-1}\times \{-1,1\}^{q'-1}$, 
$$
R^{\boldsymbol{\sigma}''}_{\boldsymbol{k}''} :=\sum_{\ell\in \mathbb{N}} \sum_{\sigma\in \{-1,1\}} \sigma \sum_{j\in \mathcal{C}_\ell} P_{\boldsymbol{k},\ell}^{\boldsymbol{\sigma},-\sigma}(\, \cdot \, ,\mathds{1}_{\{j\}})\otimes Q_{\boldsymbol{k}',\ell}^{\boldsymbol{\sigma}',\sigma}(\, \cdot \, ,\mathds{1}_{\{j\}})  .
$$
Then, it suffices to note that by duality
$$
\Big( \sum_{j\in \mathcal{C}_\ell} |  P_{\boldsymbol{k},\ell}^{\boldsymbol{\sigma},-\sigma}(\, \cdot \, ,\mathds{1}_{\{j\}}) |^2 \Big)^{1/2}= \sup_{ \substack{v\in E_\ell \\ \| v\|_{\ell^2}\leq 1} } | P_{\boldsymbol{k},\ell}^{\boldsymbol{\sigma},-\sigma}(\, \cdot \, ,v) | 
$$
to get by Cauchy--Schwarz that
\begin{equation}
\label{eq:rel2}
\| R^{\boldsymbol{\sigma}''}_{\boldsymbol{k}''} \|_{\mathscr{L}_{\boldsymbol{k}''}} \leq 2 \sum_{\ell\in \mathbb{N}} \| P \|_{\mathscr{L}_{\boldsymbol{k},\ell}} \| Q \|_{\mathscr{L}_{\boldsymbol{k}',\ell}} \leq   2  \sum_{\ell \in \mathbb{N}} \Gamma_{\boldsymbol{k},\ell}\Gamma_{\boldsymbol{k}',\ell} A_{\boldsymbol{k},\ell} A_{\boldsymbol{k}',\ell}
\end{equation}
where we have assumed by homogeneity that $\| P \|_{\mathscr{H}_{q}^{n,\nu}}=\| Q\|_{\mathscr{H}_{q'}^{n,\nu}}=1$ and used the notation
$$
\forall p\geq 2,\forall \boldsymbol{h}\in\mathbb{N}^p, \quad A_{\boldsymbol{h}} =   \big(\frac{\boldsymbol{h}_2^\star}{\boldsymbol{h}_1^\star} \big)^n \prod_{i = 3}^p  (\boldsymbol{h}_i^\star)^\nu.
$$
In the two next steps, we are going to prove that
\begin{equation}
\label{eq:reste_a_faire}
\sum_{\ell \in \mathbb{N}}  \Gamma_{\boldsymbol{k},\ell}\Gamma_{\boldsymbol{k}',\ell} \lesssim  \Gamma_{\boldsymbol{k},\boldsymbol{k}'} \quad \mathrm{and} \quad \sup_{\ell \in \mathbb{N}}  A_{\boldsymbol{k},\ell}A_{\boldsymbol{k}',\ell} \leq  A_{\boldsymbol{k},\boldsymbol{k}'}.
\end{equation}
These estimates imply that
\begin{equation*}
\| R^{\boldsymbol{\sigma}''}_{\boldsymbol{k}''} \|_{\mathscr{L}_{\boldsymbol{k}''}} \lesssim  \Gamma_{\boldsymbol{k},\boldsymbol{k}'}A_{\boldsymbol{k},\boldsymbol{k}'}.
\end{equation*}

\medskip

This estimates on $R^{\boldsymbol{\sigma}''}_{\boldsymbol{k}''}$ is almost the expected one. The only missing property is the symmetry condition on the coefficients of $\{P,Q\}$. To get it, it suffices to average \eqref{eq:rel1} by the action of the group of the permutations of $\{ 1,\cdots,q''\}$ and to note that it does not affect the multilinear estimates we proved on the coefficients $R^{\boldsymbol{\sigma}''}_{\boldsymbol{k}''}$.

\medskip

\noindent \underline{\emph{$\triangleright$ Step $2$ : Estimate on $\Gamma$.}} The first estimate in \eqref{eq:reste_a_faire} is a consequence of the more general fact: for all $b>1$, all $x,y\in\mathbb{R}$
$$\sum_{\ell \in \mathbb{N}}\frac1{\langle\ell-x\rangle^b}\frac1{\langle\ell-y\rangle^b}\lesssim_b \frac1{\langle x-y\rangle^b}$$
which in turn can be proved as follow: as $\langle x-y\rangle\leq \langle\ell-x\rangle+\langle\ell-y\rangle$ we deduce that for all $\ell\in\mathbb{N}$, either  $\langle x-y\rangle\leq 2\langle\ell-x\rangle$ or $\langle x-y\rangle\leq 2\langle\ell-y\rangle$. Thus
$$\sum_{\ell \in \mathbb{N}}\frac1{\langle\ell-x\rangle^b}\frac1{\langle\ell-y\rangle^b}\leq \frac{2^b}{\langle x-y\rangle^b} \left(\sum_{\ell\in\mathbb{N}}\frac1{\langle\ell-x\rangle^b}+ \sum_{\ell\in\mathbb{N}}\frac1{\langle\ell-y\rangle^b}\right)\lesssim_b \frac1{\langle x-y\rangle^b}.$$

\medskip

\noindent \underline{\emph{$\triangleright$ Step $3$ : Estimate on $A$.}} We want to prove that, $\forall p,p'\geq 2,\forall \boldsymbol{k}\in\mathbb{N}^p, \forall \boldsymbol{k}'\in\mathbb{N}^{p'}, \forall \ell\in\mathbb{N},$ we have
$$A_{\boldsymbol{k},\ell}A_{\boldsymbol{k}',\ell} \leq  A_{\boldsymbol{k},\boldsymbol{k}'}.
$$
Without loss of generality we can assume that $\boldsymbol{k}$ and $\boldsymbol{k}'$ are ordered and we denote by $\boldsymbol{k}''$ the ordered version of $(\boldsymbol{k},\boldsymbol{k}')$ and we set $p''=p+p'$. Then we rewrite $A_{\boldsymbol{k}}$ as
$$A_{\boldsymbol{k}}= a_{\boldsymbol{k}}^\nu b_{\boldsymbol{k}}^{n-\nu}$$
where
$$a_{\boldsymbol{k}}= \frac{\prod_{i=1}^p \boldsymbol{k}_i }{\boldsymbol{k}_1^2},\quad b_{\boldsymbol{k}}=\frac{\boldsymbol{k}_2}{\boldsymbol{k}_1}.$$
As $n\geq \nu$ it remains to prove
\begin{equation}\label{ab}a_{\boldsymbol{k},\ell}a_{\boldsymbol{k}',\ell} \leq  a_{\boldsymbol{k}''} \quad \text{and}\quad b_{\boldsymbol{k},\ell}b_{\boldsymbol{k}',\ell} \leq  b_{\boldsymbol{k}''}.\end{equation}
To begin with we have
$$a_{\boldsymbol{k},\ell}a_{\boldsymbol{k}',\ell}=\frac{\ell^2\prod_{i=1}^{p''} \boldsymbol{k}''_i }{\max(\boldsymbol{k}_1,\ell)^2\max(\boldsymbol{k}'_1,\ell)^2}\leq \frac{\ell^2\prod_{i=1}^{p''} \boldsymbol{k}''_i }{(\boldsymbol{k}''_1)^2\ell^2}=a_{\boldsymbol{k}''}.$$
For the second estimate in \eqref{ab} we can assume without loss of generality that $\boldsymbol{k}_1=\boldsymbol{k}''_1$ (and thus $\boldsymbol{k}'_1\leq\boldsymbol{k}''_2$) . Then we argue according to the place of $\ell$ with respect to $\boldsymbol{k}''_1$ and $\boldsymbol{k}''_2$:
\begin{itemize}

\item If $\ell\leq\boldsymbol{k}''_2$  then $b_{\boldsymbol{k},\ell}=\frac{\max(\boldsymbol{k}_2,\ell)}{\boldsymbol{k}''_1}\leq \frac{\boldsymbol{k}''_2}{\boldsymbol{k}''_1}=b_{\boldsymbol{k}''}.$ Thus 
$b_{\boldsymbol{k},\ell}b_{\boldsymbol{k}',\ell} \leq  b_{\boldsymbol{k}''}$ since $b_{\boldsymbol{k}',\ell}\leq1$.

\item If $\boldsymbol{k}''_2\leq\ell\leq\boldsymbol{k}''_1$ then $b_{\boldsymbol{k},\ell}=\frac{\ell}{\boldsymbol{k}''_1}$ and $b_{\boldsymbol{k}',\ell}=\frac{\boldsymbol{k}'_1}{\ell}\leq \frac{\boldsymbol{k}''_2}{\ell}$ and thus $b_{\boldsymbol{k},\ell}b_{\boldsymbol{k}',\ell} \leq  b_{\boldsymbol{k}''}$.

\item If $\boldsymbol{k}''_1\leq\ell$ then
$$b_{\boldsymbol{k},\ell}b_{\boldsymbol{k}',\ell}=\frac{\boldsymbol{k}_1}{\ell}\frac{\boldsymbol{k}'_1}{\ell}\leq \frac{\boldsymbol{k}'_1}{\ell}\leq \frac{\boldsymbol{k}''_2}{\boldsymbol{k}''_1}=b_{\boldsymbol{k}''}.$$

\end{itemize}
\end{proof}

\subsection{Inhomogeneous polynomials and flows}

\begin{definition}[Space of inhomogeneous polynomials $\mathscr{I}^{\nu,n}_{\leq r}$] For all $n,\nu\geq 0$ and $r \geq 3$, we define a space of polynomials of degree smaller or equal to $r$ by
$$
\mathscr{I}^{\nu,n}_{\leq r} := \bigoplus_{3\leq q \leq r} \mathscr{H}_q^{\nu,n}.
$$
For all $P\in \mathscr{I}^{\nu,n}_{\leq r}$, we denote by 
$$
P =: P^{(3)} + \cdots + P^{(r)} \quad \mathrm{with} \quad P^{(q)} \in \mathscr{H}_q^{\nu,n}, \ \forall q\in\{3,\cdots,r\}
$$
the associated decomposition and we set, for all $\gamma>0$,
$$
\| P \|_{\mathscr{I}^{\nu,n}_{\gamma}} := \sup_{3\leq q \leq r} \gamma^{q-3} \| P^{(q)} \|_{\mathscr{H}_q^{\nu,n}}.
$$
\end{definition}
As a consequence of Lemma \ref{lem:pointwise} and Lemma \ref{lem:vf}, we have the following tame estimate.
\begin{lemma} \label{lem:tame_inhom} Let $n,\nu\geq 0$ and $r \geq 3$. All polynomial $P\in \mathscr{I}^{\nu,n}_{\leq r}$ defines a smooth real valued function on $H^{1+\nu}$. Moreover, for all $s\in [1+\nu,n]$, $\nabla P$ is a smooth function from $H^s$ into $H^s$ and it satisfies, for all $\gamma \in (0,1]$, all $u\in H^s$ 
$$
\| \nabla P(u) \|_{H^s} \lesssim_{r,n} \| P \|_{\mathscr{I}^{\nu,n}_{\gamma}} \| u\|_{H^s} \| u\|_{H^{1+\nu}} \big(1+\gamma^{-1} \| u\|_{H^{1+\nu}} \big)^{r-3}.
$$
\end{lemma}

Now, we are going to consider Hamiltonian systems of the form
\begin{equation}
\label{eq:aux} 
\ic \partial_t v = \nabla \chi(v)
\end{equation}
where $\chi \in \mathscr{I}^{\nu,n}_{\leq r}$ some some $r,n,\nu\geq 0$. Thanks to the tame estimate of Lemma \ref{lem:tame_inhom} the local Cauchy theory of this equation is given by the Cauchy--Lipschitz Theorem. In particular, we have the following classical proposition about its flow.
\begin{proposition}[Flows $\Phi_\chi^t$] \label{prop:flow} Let $n \geq \nu+1 \geq 0$,  $r\geq 3$ and $\chi \in \mathscr{I}^{\nu,n}_{\leq r}$. There exists a unique open set $\mathcal{V}_\chi \subset \mathbb{R} \times H^{1+\nu}$ and a unique map $\Phi_\chi \equiv (t,u) \mapsto \Phi_{\chi}^t(u) \in C^\infty(\mathcal{V}_\chi; H^{1+\nu})$ such that
$$
\forall (t,u) \in \mathcal{V}_\chi, \quad \ic \partial_t \Phi_\chi^t(u) = \nabla \chi( \Phi_\chi^t(u) ) \quad \mathrm{and} \quad \Phi_\chi^0(u) = u
$$
and satisfying the following properties
\begin{itemize}
\item \emph{maximality:} $\mathcal{V}_\chi$ is of the form
$$
\mathcal{V}_\chi = \bigcup_{u \in H^{1+\nu}} (-T_u^-,T_u^+) \times \{ u\} , \quad \mathrm{with} \quad T_u^-,T_u^+ \in (0,+\infty]
$$
and for all $u\in H^{1+\nu}$ and $\sigma \in \{-1,1\}$, 
\begin{center} 
either $T_u^{\sigma} = +\infty$ or $\displaystyle \lim_{t \to T_u^{\sigma}} \| \Phi_\chi^t(u) \|_{H^{1+\nu}}=+\infty.$
\end{center}
\item \emph{symplecticity:} For all $t\in \mathbb{R}$, $\Phi_\chi^t$ is symplectic on $ \{ u \in H^{1+\nu} \ | \ (t,u)\in \mathcal{V}_\chi \}.$
\item \emph{preservation of regularity:} for all $s\in [1+\nu,n]$, $\Phi_\chi \in \mathcal{C}^\infty(\mathcal{V}_\chi \cap (\mathbb{R} \times H^s) ; H^s)$.
\end{itemize}
\end{proposition} 

Finally, we prove some useful estimates on small solutions.
\begin{lemma}\label{lem:est_flow}  Let $n,\nu\geq 0$, $s\in [\nu+1,n]$,  $r\geq 3$ and $\chi \in \mathscr{I}^{\nu,n}_{\leq r}$. There exists 
\begin{equation}
\label{eq:le_rayon_de_la_boule}
\varepsilon_{n,\gamma,\chi} \sim_{n,r} \min( \| \chi \|_{\mathscr{I}^{\nu,n}_{\gamma}} ^{-1} ,\gamma)
\end{equation}
such that, for $t\in[-1,1]$, $\Phi_\chi^t$ is well defined on $B_{H^{1+\nu}}(0,2\varepsilon_{n,\gamma,\chi})$, i.e.
\begin{equation}
\label{eq:inclusion}
[-1,1] \times B_{H^{1+\nu}}(0,2\varepsilon_{n,\gamma,\chi}) \subset \mathcal{V}_\chi,
\end{equation}
and it is close to the identity, i.e. for all $u\in H^s \cap B_{H^{1+\nu}}(0,2\varepsilon_{n,\gamma,\chi})$ and $t\in [-1,1]$ we have
\begin{equation}
\label{eq:close_id_vf}
\| \Phi_\chi^t(u) - u \|_{H^s} \leq \frac{\|u  \|_{H^{1+\nu}}}{\varepsilon_{n,\gamma,\chi}} \| u\|_{H^s},
\end{equation}
\begin{equation}
\label{eq:close_id_diff}
\| \mathrm{d} \Phi_\chi^t(u) - \mathrm{Id} \|_{H^s \to H^s} \leq \frac{\| u\|_{H^s}}{\varepsilon_{n,\gamma,\chi}} \big\langle \frac{\| u\|_{H^s}}{2\varepsilon_{n,\gamma,\chi}}  \big\rangle^{r-3}.
\end{equation}
\end{lemma}
\begin{proof} Without loss of generality, we focus only on positive times. We set
$$
\varepsilon_{n,\gamma,\chi} := c_{n,r} \min( \| \chi \|_{\mathscr{I}^{\nu,n}_{\gamma}} ^{-1} ,\gamma)
$$
where $c_{n,r}$ is a constant depending only on $r$ and $n$ that we will choose small enough. 

\medskip

We fix $u\in  B_{H^{1+\nu}}(0,2\varepsilon_{n,\gamma,\chi})$ and we set 
$$
T^* = \sup \big\{ t \leq \min(1,T_u^+) \ | \ \forall \tau \leq t, \  \| \Phi_\chi^\tau(u) \|_{H^{1+\nu}} \leq 2 \| u\|_{H^{1+\nu}} \big\}.
$$
Let $t< T_*$ and $s \in [\nu+1,n] $ such that $u\in H^s$. First, we note that, provided that $c_{n,r} \leq \frac14$ then $\|\Phi_\chi^t(u)\|_{H^{1+\nu}}<\gamma$. Thus, by Lemma  \ref{lem:tame_inhom}, we have that
$$
 \| \Phi_\chi^t(u) - u\|_{H^s}  \lesssim_{r,n}  |t| \| \chi \|_{\mathscr{I}^{\nu,n}_{\gamma}} \| u\|_{H^s} \| u\|_{H^{1+\nu}}.
$$
Therefore, provided that $c_{n,r} $ is small enough, we have that
$$
 \| \Phi_\chi^t(u) - u\|_{H^s} \leq \frac12 \frac{\|u  \|_{H^{1+\nu}}}{\varepsilon_{n,\gamma,\chi}} \| u\|_{H^{s}}.
$$
This estimate implies that $T^*= 1$ (i.e. \eqref{eq:inclusion}) and that \eqref{eq:close_id_vf} holds.

\medskip

It only remains to prove \eqref{eq:close_id_diff}. So let $u\in B_{H^{1+\nu}}(0,2\varepsilon_{n,\gamma,\chi})$, $s \in [\nu+1,n]$ and $t\leq 1$. Deriving \eqref{eq:aux}, it comes
$$
\forall w\in H^{1+\nu}, \quad \ic \partial_t \mathrm{d} \Phi_\chi^t(u)(w) = \mathrm{d}\nabla \chi( \Phi_\chi^t(u)) ( \mathrm{d} \Phi_\chi^t(u)(w) )
$$
and so
\begin{equation}
\label{eq:to_gron}
\| \mathrm{d} \Phi_\chi^t(u) - \mathrm{Id} \|_{H^s \to H^s} \leq \int_{0}^t \| \mathrm{d}\nabla \chi( \Phi_\chi^{\tau}(u)) \|_{H^s \to H^s}  \| \mathrm{d} \Phi_\chi^\tau(u)   \|_{H^s \to H^s}   \, \mathrm{d}\tau.
\end{equation}
Using the bound \eqref{eq:bound_dP} on $\mathrm{d}\nabla \chi$ it comes 
$$
\| \mathrm{d}\nabla \chi( \Phi_\chi^t(u)) \|_{H^s \to H^s} \lesssim_{n,r}  \| \chi \|_{\mathscr{I}^{\nu,n}_{\gamma}}  \sum_{3\leq q\leq r}\gamma^{-q+3} \| \Phi_\chi^t(u) \|_{H^s}^{q-2}.
$$
Therefore, provided that $c_{n,r}$ is small enough, we have
$$
\| \mathrm{d}\nabla \chi( \Phi_\chi^t(u)) \|_{H^s \to H^s} \leq  \frac18 \frac{\| u\|_{H^s}}{\varepsilon_{n,\gamma,\chi}} \big\langle \frac{\| u\|_{H^s}}{2\varepsilon_{n,\gamma,\chi}}  \big\rangle^{r-3}.
$$
Hence, it suffices to apply the Gr\"onwall estimate to \eqref{eq:to_gron} to get \eqref{eq:close_id_diff}.

\end{proof}

\section{Normal form}
In this section the setting is the same as in the previous one. We frequencies $\omega$ and the cluster decomposition $(\mathcal{C}_k)_{k\in \mathbb{N}}$ are considered as given. The aim of this section is to prove a Birkhoff normal form theorem to remove by symplectic changes of variables terms which "do not commute enough" with  
$$
Z_2(u) := \frac12 \sum_{j\in \mathbb{N}} \omega_j |u_j|^2, \quad u \in H^{\frac{\alpha}{2}}.
$$
This Hamiltonian correspond to the linear part of the abstract Hamiltonian PDE \eqref{eq:ham-pde}, i.e.
$$
\nabla Z_2(u) = \Omega u =: (\omega_j u_j)_{j\in \mathbb{N}}.
$$
In order to quantify how much terms are resonant it is useful to introduce the following operator.
\begin{definition}[Operator $\mathcal{L}_{\boldsymbol{k},\boldsymbol{\sigma}}$] For all $q\geq 3$, $\boldsymbol{k} \in \mathbb{N}^q$, $\boldsymbol{\sigma} \in \{-1,1\}^q$, we define the endomorphism $\mathcal{L}_{\boldsymbol{k},\boldsymbol{\sigma}} :  \mathscr{L}_{\boldsymbol{k}} \to  \mathscr{L}_{\boldsymbol{k}}$ by the relation
$$
\mathcal{L}_{\boldsymbol{k},\boldsymbol{\sigma}}(M)(u^{(1)},\cdots,u^{(q)})  := \sum_{1\leq i \leq q}  \boldsymbol{\sigma}_i M(u^{(1)},\cdots,u^{(i-1)}, \Omega u^{(i)}, u^{(i+1)} ,\cdots, u^{(q)} )
$$
for all $M\in \mathscr{L}_{\boldsymbol{k}}$ and all $(u^{(1)},\cdots,u^{(q)}) \in E_{\boldsymbol{k}_1} \times\cdots \times  E_{\boldsymbol{k}_q} $.
\end{definition}
As stated in the following lemma, it appears naturally when considering Poisson brackets with $Z_2$. Indeed, by a straightforward calculation, we have:
\begin{lemma} \label{lem:brack_Z_2} Let $q\geq 3$, $n,\nu\geq 0$ and $P\in \mathscr{H}_q^{n,\nu}$. Then, for all $u\in \mathbb{C}^{\mathbb{N}}$ with finite support, we have
$$
\{ Z_2,P\}(u) =  \sum_{\boldsymbol{k}\in \mathbb{N}^q} \sum_{\boldsymbol{\sigma} \in \{-1 , 1\}^q } (\ic\mathcal{L}_{\boldsymbol{k},\boldsymbol{\sigma}}P_{\boldsymbol{k}}^{\boldsymbol{\sigma}})(\Pi_{\boldsymbol{k}_1}^{\boldsymbol{\sigma}_1} u, \cdots, \Pi_{\boldsymbol{k}_q}^{\boldsymbol{\sigma}_q} u).
$$
\end{lemma} 
 Now, we define the polynomials which "almost commute" with $Z_2$.
\begin{definition}[$\gamma$-resonance] \label{def:non_res}Let $\gamma >0$, $q\geq 3$ and $n,\nu\geq 0$.  A couple $(\boldsymbol{k},\boldsymbol{\sigma})\in \mathbb{N}^q \times \{-1,1\}^q$ is non $\gamma$-resonant if 
\begin{equation}
\label{eq:sd_gen}
   \mathcal{L}_{\boldsymbol{k},\boldsymbol{\sigma}} \ \mathrm{is \ invertible} \quad \mathrm{and} \quad \forall M \in \mathscr{L}_{\boldsymbol{k}} , \quad  \|    \mathcal{L}_{\boldsymbol{k},\boldsymbol{\sigma}}^{-1} M \|_{\mathscr{L}_{\boldsymbol{k}}  } \leq \gamma^{-1} \| M \|_{\mathscr{L}_{\boldsymbol{k}}}.
\end{equation}
A homogeneous polynomial $P\in \mathscr{H}_{q}^{\nu,n}$ is \emph{$\gamma$-resonant} if for all $(\boldsymbol{k},\boldsymbol{\sigma})\in \mathbb{N}^q \times \{-1,1\}^q$ either $P_{\boldsymbol{k}}^{\boldsymbol{\sigma}} =0$ or $(\boldsymbol{k},\boldsymbol{\sigma})$ is $\gamma$-resonant. An inhomogeneous polynomial is \emph{$\gamma$-resonant} if it is the sum of \emph{$\gamma$-resonant} homogeneous polynomials.
\end{definition}
Note that \eqref{eq:sd_gen} is stronger than standard small divisor estimates. Nevertheless, it is a natural multi-dimensional extension. We provide examples of $\gamma$-non resonant terms in subsection \ref{sub:eg_nr}.

\subsection{A Birkhoff normal form theorem} Now, we prove a Birkhoff normal form theorem which allows to remove the non-$\gamma$-resonant terms (see e.g. \cite{BG25,BC24} for similar formulations).
\begin{theorem}
\label{thm:birk} Let $r\geq 3$, $\nu \geq \alpha$. For all $n \geq \nu + 1$, all $P \in \mathscr{I}^{\nu,n}_{\leq r} $ of norm $B:= \|P\|_{\mathscr{I}^{\nu,n}_{1}}$ and all $\gamma \in (0,1)$, there exists $\chi \in \mathscr{I}^{\nu,n}_{\leq r}$ such that $\|\chi\|_{\mathscr{I}^{\nu,n}_{\gamma}}\lesssim_{n,\nu,r,B } \gamma^{-1}$ and 
$$
(Z_2+P)\circ \Phi_\chi^{-1} = Z_2 + Q + R \quad \mathrm{on} \quad B_{H^{1+\nu}}(0, \varepsilon_{n,\gamma,\chi} ) 
$$
where $ \varepsilon_{n,\gamma,\chi} \gtrsim_{n,r,B,\nu} \gamma$ is given by Lemma \ref{lem:est_flow}\footnote{to ensure that  $\Phi_\chi^{-1}$ is well defined} and
\begin{itemize}
\item $Q \in \mathscr{I}^{\nu,n}_{\leq r} $ is a $\gamma$-resonant polynomial of norm $\|Q\|_{\mathscr{I}^{\nu,n}_{\gamma}}\lesssim_{n,\nu,r,B } 1$,
\item $R$ is a $C^\infty$ real valued function on $B_{H^{1+\nu}}(0, \varepsilon_{n,\gamma,\chi} ) $ satisfying, for all $s\in [1+\nu,n]$ and all $u\in B_{H^s}(0,\varepsilon_{n,\gamma,\chi})$
\begin{equation}
\label{eq:est_R}
\| \nabla R(u) \|_{H^s} \lesssim_{n,r,\nu,B} \gamma^{-r+2} \| u\|_{H^s}^r.
\end{equation}
\end{itemize}

\end{theorem}
\begin{proof}  Let $n\geq \nu+1$. Then, let $\gamma \in (0,1)$, $P\in \mathscr{I}^{\nu,n}_{\leq r} $ and set $B:= \|P\|_{\mathscr{I}^{\nu,n}_{1}}$. We divide the proof in $3$ steps.

\noindent \underline{\emph{$\triangleright$ Step 1: Expansion.}} Let $\chi \in \mathscr{I}^{\nu,n}_{\leq r}$. It will be fixed at the second step but for the moment we do not impose any constraint on it.

First, we note that if $F \in C^\infty(H^{1+\nu} ;\mathbb{R})$ then
$$
\partial_t F\circ \Phi_\chi^{-t}(u) = \{\chi , F\} \circ \Phi_\chi^{-t}(u) \quad \mathrm{for} \quad (t,u) \in \mathcal{V}_\chi.
$$
Therefore denoting $\mathrm{ad}_\chi := \{\chi,\cdot\}$ and performing the Taylor expansion of $(Z_2+P)\circ \Phi_\chi^{-1}$, it comes, on $ B_{H^{1+\nu}}(0, \varepsilon_{n,\gamma,\chi} ) $\footnote{because by Lemma \ref{lem:est_flow}, one has that $[-1,1] \times B_{H^{1+\nu}}(0,\varepsilon_{n,\gamma,\chi}) \subset \mathcal{V}_\chi$.}
$$
(Z_2+P)\circ \Phi_\chi^{-1}  = \sum_{\ell = 0}^{r-2} \frac{\mathrm{ad}_{\chi}^{\ell}}{\ell ! } (Z_2 + P) + \int_{0}^1 \frac{(1-t)^{r-2}}{(r-2) !} \mathrm{ad}_\chi^{r-1} (Z_2 + P) \circ \Phi_\chi^{-t} \mathrm{d}t.
$$
Then, we expand all the polynomial and group terms depending on their homogeneity. More precisely, we set
$$
K^{(d)} = \sum_{\ell=1}^{r-2}  \frac{K^{(\ell,d)}}{\ell !} \quad  \mathrm{with} \quad  K^{(\ell,d)} := \sum_{ \boldsymbol{q}_1 + \cdots + \boldsymbol{q}_{\ell+1} = d + 2\ell } \mathrm{ad}_{\chi^{(\boldsymbol{q}_1)}} \cdots  \mathrm{ad}_{\chi^{(\boldsymbol{q}_{\ell})}} P^{(\boldsymbol{q}_{\ell +1})},
$$
$$
H^{(d)} = \sum_{\ell=2}^{r-2}  \frac{H^{(\ell,d)}}{(\ell+1) !}  \quad  \mathrm{with} \quad  H^{(\ell,d)} := \sum_{ \boldsymbol{q}_1 + \cdots + \boldsymbol{q}_{\ell+1} = d + 2\ell } \mathrm{ad}_{\chi^{(\boldsymbol{q}_1)}} \cdots  \mathrm{ad}_{\chi^{(\boldsymbol{q}_{\ell})}} \{\chi^{(\boldsymbol{q}_{\ell+1})} , Z_2\},
$$
and it comes
$$
(Z_2+P)\circ \Phi_\chi^{-1} = Z_2 + Q + R \quad \mathrm{on} \quad B_{H^{1+\nu }}(0, \varepsilon_{n,\gamma,\chi} ) 
$$
where
\begin{equation}
\label{eq:def_Q}
Q :=  P + \{\chi,Z_2\}+ \sum_{3\leq d \leq r} K^{(d)} + H^{(d)},
\end{equation}
$$
 R := \underbrace{\sum_{r+1\leq d \leq r^2} K^{(d)} + H^{(d)} }_{=:R^{(\mathrm{pol})}} + \underbrace{\sum_{r+1\leq d \leq r^2} \int_{0}^1 \frac{(1-t)^{r}}{r !} ( K^{(r+1,d)} + H^{(r,d)} ) \circ \Phi_\chi^{-t} \mathrm{d}t}_{=:R^{(\mathrm{Tay})}}.
$$

\medskip

\noindent \underline{\emph{$\triangleright$ Step 2: Resolution of the cohomological equations.}} Now, we have to design $\chi$ so that $Q$, defined by \eqref{eq:def_Q}, belongs to  $\mathscr{I}^{\nu,n}_{\leq r}$ and is $\gamma$-resonant. Let us note that decomposing $Q$ as a sum of homogeneous polynomials, we have
$$
Q = \sum_{d = 3}^r Q^{(d)} \quad \mathrm{where} \quad Q^{(d)} := P^{(d)}  + K^{(d)}+ H^{(d)}- \{  Z_2 , \chi^{(d)} \}.
$$

\medskip

We proceed by induction. We aim at proving that, for all $d \in \{3,\cdots,r\}$
there exists $\chi^{(d)}  \in \mathscr{H}_d^{\nu,n}$ so that $Q^{(d)},H^{(d)},K^{(d)} \in \mathscr{H}_d^{\nu,n}$, $Q^{(d)}$ is $\gamma$ resonant and the following bounds holds 
\begin{equation}
\label{eq:ind_hp}
\gamma \| \chi^{(d)} \|_{ \mathscr{H}_d^{\nu,n} }  + \| H^{(d)} \|_{ \mathscr{H}_d^{\nu,n} }+\| K^{(d)} \|_{ \mathscr{H}_d^{\nu,n} }+\| Q^{(d)} \|_{ \mathscr{H}_d^{\nu,n} } \lesssim_{n,B } \gamma^{-d+3}.
\end{equation}

\medskip

To prove it, we consider $d \in \{3,\cdots,r\}$ and assume that $\chi^{(q)}$ have been designed for all $q<d$. The key point is to note that $K^{(d)}$ and $H^{(d)}$ only depend on $\chi$ through the homogeneous terms $\chi^{(q)}$ with $q < d$ (which makes the system triangular in some sense).

\medskip

First, we control $K^{(d)}$. Indeed, applying Proposition \ref{prop:poisson} to control the Poisson brackets, using the induction hypothesis \eqref{eq:ind_hp} and the bound $\|P\|_{\mathscr{I}^{\nu,n}_{\gamma}} \leq \|P\|_{\mathscr{I}^{\nu,n}_{1}} =B $, it comes that for all $\ell \leq d$, $K^{(\ell,d)}\in \mathscr{H}_d^{\nu,n} $ and
\begin{equation}
\label{eq:jai_pas_envie_de_le_refaire}
\begin{split}
\| K^{(\ell,d)} \|_{ \mathscr{H}_d^{\nu,n} } &\lesssim_{n,B} \! \! \! \!  \! \! \sum_{ \boldsymbol{q}_1 + \cdots + \boldsymbol{q}_{\ell+1} = d + 2\ell } \! \! \! \!  \! \! \| \chi^{(\boldsymbol{q}_1)} \|_{ \mathscr{H}_{\boldsymbol{q}_1}^{\nu,n}}   \cdots  \| \chi^{(\boldsymbol{q}_\ell)} \|_{ \mathscr{H}_{\boldsymbol{q}_\ell}^{\nu,n} }  \| P^{(\boldsymbol{q}_{\ell+1})} \|_{ \mathscr{H}_{\boldsymbol{q}_{\ell+1}}^{\nu,n} }\\
&\lesssim_{n,B} \! \! \! \!  \! \! \sum_{ \boldsymbol{q}_1 + \cdots + \boldsymbol{q}_{\ell+1} = d + 2\ell } \! \! \! \!  \! \! \gamma^{-\boldsymbol{q}_1+2} \cdots \gamma^{-\boldsymbol{q}_{\ell}+2} \gamma^{-\boldsymbol{q}_{\ell+1}+3} \sim_{n,B} \gamma^{-d+3}.
\end{split}
\end{equation}
Therefore, we have $\| K^{(d)} \|_{ \mathscr{H}_d^{\nu,n} } \lesssim_{n,\nu,d,B} \gamma^{-d+3}$.

\medskip

Now we focus on $H^{(d)}$. We note that, by construction, for all $q<d$
$$
\{  Z_2 , \chi^{(q)} \} =  P^{(q)}  + K^{(q)}+ H^{(q)} - Q^{(q)}.
$$
Therefore, by induction hypothesis, we get that for all $q<d$, $\{  Z_2 , \chi^{(q)} \} \in \mathscr{H}_q^{\nu,n} $ and that it satisfies the bound
$$
\| \{  Z_2 , \chi^{(q)} \}  \|_{\mathscr{H}_q^{\nu,n}}  \lesssim_{n,B} \gamma^{-q+3}.
$$
As a consequence, proceeding exactly as we did for $K^{(\ell,d)} $, we deduce that 
$$
\| H^{(d)} \|_{ \mathscr{H}_d^{\nu,n} } \lesssim_{n,B} \gamma^{-d+3}.
$$

\medskip

Now, we design $\chi^{(d)}$. For all $\boldsymbol{k}\in \mathbb{N}^d$ and $\boldsymbol{\sigma} \in \{ -1,1 \}^d$, we set 
$$
(\chi^{(d)})_{\boldsymbol{k}}^{\boldsymbol{\sigma}} :=    (\ic\mathcal{L}_{\boldsymbol{k},\boldsymbol{\sigma}})^{-1} (P^{(d)}  + K^{(d)}+ H^{(d)})_{\boldsymbol{k}}^{\boldsymbol{\sigma}} \quad \mathrm{if} \quad (\boldsymbol{k},\boldsymbol{\sigma}) \ \mathrm{is} \ \mathrm{non-}\gamma\mathrm{-resonant}
$$
and $(\chi^{(d)})_{\boldsymbol{k}}^{\boldsymbol{\sigma}}= 0$ else. By definition of $\gamma$-resonance, using the bounds we proved on $K^{(d)}$ and $H^{(d)}$, it implies that $\chi^{(d)}\in \mathscr{H}_q^{\nu,n}$ and that it satisfies the bound 
$$
\| \chi^{(d)} \|_{ \mathscr{H}_d^{\nu,n} } \lesssim_{n,B}\gamma^{-d+2}.
$$

\medskip

Finally, using the formula of Lemma \ref{lem:brack_Z_2}, it implies that
for all $\boldsymbol{k}\in \mathbb{N}^d$ and $\boldsymbol{\sigma} \in \{ -1,1 \}^d$, 
$$
(Q^{(d)})_{\boldsymbol{k}}^{\boldsymbol{\sigma}} =    (P^{(d)}  + K^{(d)}+ H^{(d)})_{\boldsymbol{k}}^{\boldsymbol{\sigma}} \quad \mathrm{if} \quad (\boldsymbol{k},\boldsymbol{\sigma}) \ \mathrm{is} \ \gamma\mathrm{-resonant}
$$
and $(Q^{(d)})_{\boldsymbol{k}}^{\boldsymbol{\sigma}}= 0$ else. Thus, we get that $Q^{(d)} \in \mathscr{H}_d^{\nu,n}$ is $\gamma$-resonant and satisfies the bound  $
\| Q^{(d)} \|_{ \mathscr{H}_d^{\nu,n} } \lesssim_{n,B} \gamma^{-d+3}.
$
This conclude the proof of the induction.

\medskip

\noindent \underline{\emph{$\triangleright$ Step 3: Bound on the remainder terms.}} Finally, we have to control the remainder term $R = R^{(\mathrm{pol})}+R^{(\mathrm{Tay})} $.

\medskip

The key point here is to note that proceeding exactly as we did at the previous step (see \eqref{eq:jai_pas_envie_de_le_refaire}), we have that for all $d\leq r^2$ and $\ell\leq r-1$, $K^{(\ell,d)}\in \mathscr{H}_d^{\nu,n} ,H^{(\ell,d)}\in \mathscr{H}_d^{\nu,n} $ satisfy the bound
\begin{equation}
\label{eq:pas_si_mal}
\| K^{(\ell,d)} \|_{\mathscr{H}_d^{\nu,n}} +  \| H^{(\ell,d)} \|_{\mathscr{H}_d^{\nu,n}} \lesssim_{n,B} \gamma^{-d+3}.
\end{equation}

\medskip

Now, we fix $u\in B_{H^s}(0,\varepsilon_{n,\gamma,\chi})$ and  $s\in [1+\nu,n]$.  We note that by construction of $\chi$, we have $\|\chi\|_{\mathscr{I}^{\nu,n}_{\gamma}}\lesssim_{n,\nu,r,B } \gamma^{-1}$ and so
$\varepsilon_{n,\gamma,\chi} \sim_{n,r,B,\nu} \gamma$ (see \eqref{eq:le_rayon_de_la_boule} and \eqref{eq:ind_hp}). Therefore, applying Lemma \ref{lem:vf} and using \eqref{eq:pas_si_mal}, we directly have that
$$
\| \nabla R^{(\mathrm{pol})}(u) \|_{H^s} \lesssim_{n,\nu,r,B}  \gamma^{-r+2} \| u\|_{H^s}^r.
$$
To estimate $R^{(\mathrm{Tay})}$, we first note that, for all $r+1\leq d\leq r^2$ and $t\in [0,1]$, since $\Phi_\chi^{-t}$ is symplectic and $\Phi_\chi^{-t} \circ \Phi_\chi^{t} = \mathrm{id}$, we have
$$
 \nabla \big[ ( K^{(r+1,d)} + H^{(r,d)} ) \circ \Phi_\chi^{-t} \big](u)= -\ic \mathrm{d} \Phi_\chi^{t} (  \Phi_\chi^{-t}(u) ) \Big( \ic \big[ \nabla ( K^{(r+1,d)} + H^{(r,d)} ) \big] \circ \Phi_\chi^{-t} (u) \Big).
$$
Then, we use that by Lemma \ref{lem:est_flow}, we have 
$$
\|   \Phi_\chi^{-t}(u) \|_{H^s} \leq 2 \|   u \|_{H^s} \lesssim_{n,B}  \varepsilon_{n,\gamma,\chi} 
$$
to deduce by Lemma \ref{lem:vf}, the estimate \eqref{eq:pas_si_mal} and the estimate on $\mathrm{d} \Phi_\chi^{t}$ given by Lemma \ref{lem:est_flow} that
$$
\big\|  \nabla \big[ ( K^{(r+1,d)} + H^{(r,d)} ) \circ \Phi_\chi^{-t} \big](u) \big\|_{H^s} \lesssim_{n,B} \gamma^{-d+3} \| u\|_{H^s}^{d-1}.
$$
Finally, using that $d\geq r+1$ and $ \| u\|_{H^s} \lesssim_{n,B} \gamma$, we deduce that
$$
\| \nabla R^{(\mathrm{Tay})} \|_{H^s} \lesssim_{n,B}  \gamma^{-r+2} \| u\|_{H^s}^r.
$$

\end{proof}

\subsection{About $\gamma$-resonant terms}
\label{sub:eg_nr}

To deduce dynamical corollary of the Birkhoff normal form theorem we have just proved, we have to know more about $\gamma$-resonant terms (defined in Definition \ref{def:non_res}). 

\begin{lemma} \label{lem:2grand} Let $q\geq 3$, $\boldsymbol{k}\in \mathbb{N}^q$ and $\boldsymbol{\sigma} \in \{-1,1\}^q$ be such that 
$$
\boldsymbol{k}_1\geq \cdots \geq \boldsymbol{k}_q \quad \mathrm{and} \quad |\boldsymbol{\sigma}_1 \boldsymbol{k}_1^{\alpha}  + \boldsymbol{\sigma}_2 \boldsymbol{k}_2^{\alpha}| \gtrsim_{q} \max(\boldsymbol{k}_3^\alpha,\boldsymbol{k}_1^{\alpha-1}).
$$
Then, $(\boldsymbol{k},\boldsymbol{\sigma})$ is non-$1$-resonant.
\end{lemma}
\begin{proof} Without loss of generality we assume that $\boldsymbol{\sigma}_1 =1$. 
First, we note that thanks to \eqref{eq:lestimee_tame_pas_facile_a_pas_oublier} and the mean value inequality
$$
\forall k\in \mathbb{N},\forall u\in E_k\cap B_{L^2}(0,1), \quad \| \Omega u_k - (\Upsilon k)^\alpha u_k \|_{\ell^2} \leq  \sup_{j \in \mathcal{C}_k } |\omega_j -  (\Upsilon k)^\alpha| \lesssim  k^{\alpha-1}.
$$
Now we consider $M\in \mathscr{L}_{\boldsymbol{k}}$. It follows that
$$
\| \mathcal{L}_{\boldsymbol{k},\boldsymbol{\sigma}}(M) - \Upsilon^\alpha (\boldsymbol{k}_1^{\alpha} + \boldsymbol{\sigma}_2 \boldsymbol{k}_2^{\alpha}) M \|_{\mathscr{L}_{\boldsymbol{k}}} \lesssim \| M\|_{\mathscr{L}_{\boldsymbol{k}}} ( \boldsymbol{k}_1^{\alpha-1} + \boldsymbol{k}_2^{\alpha-1} + \boldsymbol{k}_3^{\alpha} + \cdots +  \boldsymbol{k}_q^{\alpha} ).
$$
Thus, provided that $\boldsymbol{k}$ satisfies an assumption of the kind $|\boldsymbol{\sigma}_1 \boldsymbol{k}_1^{\alpha}  + \boldsymbol{\sigma}_2 \boldsymbol{k}_2^{\alpha}| \gtrsim_{q} \max(\boldsymbol{k}_3^\alpha,\boldsymbol{k}_1^{\alpha-1})$ we have
$$
\| \mathcal{L}_{\boldsymbol{k},\boldsymbol{\sigma}}(M) \|_{\mathscr{L}_{\boldsymbol{k}}} \geq \| M\|_{\mathscr{L}_{\boldsymbol{k}}} .
$$
Therefore, the rank-nullity theorem implies that  $(\boldsymbol{k},\boldsymbol{\sigma})$ is non-$1$-resonant.

\end{proof}

A refined version of the following lemma can be found in \cite[Proposition 10]{DI17}. 
\begin{lemma} \label{lem:on_tue_tout} Assume the non resonance condition \eqref{eq:nr_thm} on the frequencies $\omega$. Then  for all $q\geq 3$, $\boldsymbol{k}_1\geq \cdots \geq \boldsymbol{k}_q$ and $\boldsymbol{\sigma} \in \{-1,1\}^q$ such that
$$ \sum_{1\leq \ell \leq q} \boldsymbol{\sigma}_{\ell}
  \mathds{1}_{\{\boldsymbol{k}_\ell \}} \neq 0,
$$
there exists $\gamma \gtrsim_{q} \boldsymbol{k}_1^{-a_q- q\alpha\beta}$ such
that $(\boldsymbol{k},\boldsymbol{\sigma})$ is
non-$\gamma$-resonant. 
\end{lemma}
\begin{proof} Let $M\in \mathscr{L}_{\boldsymbol{k}}$ and $u^{(1)} \in E_{\boldsymbol{k}_1} ,\cdots, u^{(q)} \in E_{\boldsymbol{k}_q}$ of norm $\|u^{(1)} \|_{\ell^2} \leq 1$,... $\|u^{(q)} \|_{\ell^2}\leq 1$.
Expanding $\mathcal{L}_{\boldsymbol{k},\boldsymbol{\sigma}}(M)$ it comes
$$
\mathcal{L}_{\boldsymbol{k},\boldsymbol{\sigma}}(M)(u^{(1)}  ,\cdots, u^{(q)}) = \! \! \! \! \! \! \sum_{\boldsymbol{j} \in \mathcal{C}_{\boldsymbol{k}_1} \times \cdots \times \mathcal{C}_{\boldsymbol{k}_q} } \! \! \! \!  \! \big(  \sum_{\ell=1}^q \boldsymbol{\sigma}_\ell \omega_{\boldsymbol{j}_\ell}\big) M_{\boldsymbol{j}} \prod_{\ell=1}^q u^{(\ell)}_{\boldsymbol{j}_\ell} \quad \mathrm{where} \quad M_{\boldsymbol{j}} := M(\mathds{1}_{\{\boldsymbol{j}_1\}},\cdots,\mathds{1}_{\{\boldsymbol{j}_q\}} ).
$$ 
Therefore, $\mathcal{L}_{\boldsymbol{k},\boldsymbol{\sigma}}$ is clearly invertible of invert given by
$$
\forall \boldsymbol{j} \in \mathcal{C}_{\boldsymbol{k}_1} \times \cdots \times \mathcal{C}_{\boldsymbol{k}_q} , \quad \big(\mathcal{L}_{\boldsymbol{k},\boldsymbol{\sigma}}^{-1} (M) \big)_{\boldsymbol{j}}=  \big(  \sum_{\ell=1}^q \boldsymbol{\sigma}_\ell \omega_{\boldsymbol{j}_\ell}\big)^{-1} M_{\boldsymbol{j}}.  
$$
As a consequence, using the small divisor estimate \eqref{eq:nr_thm},
the Cauchy--Schwarz inequality and the assumptions \eqref{weyl} and \eqref{eq:inclusion_clusters} on the
clusters, it comes
\begin{equation*}
\begin{split}
| \mathcal{L}_{\boldsymbol{k},\boldsymbol{\sigma}}^{-1}(M)(u^{(1)}  ,\cdots, u^{(q)}) |&\lesssim_q \boldsymbol{k}_1^{ a_q} \Big( \sum_{\boldsymbol{j} \in \mathcal{C}_{\boldsymbol{k}_1} \times \cdots \times \mathcal{C}_{\boldsymbol{k}_q} } |M_{\boldsymbol{j}}|^2 \Big)^{1/2} \prod_{\ell=1}^q \| u^{(\ell)} \|_{\ell^2} \\
&\lesssim_q \boldsymbol{k}_1^{ a_q}  \| M\|_{\mathscr{L}_{\boldsymbol{k}}}  \prod_{\ell=1}^q (\# \mathcal{C}_\ell)^{1/2} \lesssim_q \boldsymbol{k}_1^{ a_q + q \frac{\alpha\beta}2}  \| M\|_{\mathscr{L}_{\boldsymbol{k}}}.
\end{split}
\end{equation*}
\end{proof}

 As a direct consequence of these lemmas we deduce the following lemma.
 \begin{lemma} \label{lem:tous_les_types} Assume the non resonance condition \eqref{eq:nr_thm} on the frequencies $\omega$. For all $q\geq 3$ and all $N\gtrsim_q 1$, setting
 $$
 \gamma = N^{-2 a_q - q\alpha\beta - 2} ,
 $$
 all $\gamma$-resonant couple $(\boldsymbol{k},\boldsymbol{\sigma}) \in \mathbb{N}^q \times \{-1,1\}^q$ is of one of the following type
 \begin{itemize}
 \item \emph{Type I :} $\boldsymbol{k}_3^\star >N$,
  \item \emph{Type II :} $\boldsymbol{k}_1^\star \leq N$ and $\boldsymbol{\sigma}_1 \mathds{1}_{\{ \boldsymbol{k}_1 \}} + \cdots +\boldsymbol{\sigma}_q \mathds{1}_{\{ \boldsymbol{k}_q \}} =0 $,
   \item \emph{Type III :} $\boldsymbol{k}_2^\star > N$, $\boldsymbol{k}_3^\star \leq N$ and $\boldsymbol{\sigma}_1^\star =- \boldsymbol{\sigma}_2^\star$.
 \end{itemize}
 where $\boldsymbol{\sigma}^\star $ denotes an arrangement of $\boldsymbol{\sigma}$ corresponding to $\boldsymbol{k}^{\star}$\footnote{i.e. $\boldsymbol{\sigma}^\star$ is an element of $\{-1,1\}^q$ such that there exists $\varphi\in \mathfrak{S}_q$ satisfying $\boldsymbol{\sigma}^\star_{j} = \boldsymbol{\sigma}_{\varphi{j}}$ and $\boldsymbol{k}^\star_{j} = \boldsymbol{j}_{\varphi{j}}$ for all $j\in \{1,\cdots,q\}$.}.
 \end{lemma}
 \begin{proof} Let $(\boldsymbol{k},\boldsymbol{\sigma}) \in \mathbb{N}^q \times \{-1,1\}^q$ be a $\gamma$-resonant couple such that $\boldsymbol{k}_1 \geq \cdots \geq \boldsymbol{k}_q$ and  $\boldsymbol{k}_3 \leq N$. We aim at proving that it is either of type II or of type III. We distinguish $3$ basic cases.
 
 \noindent \underline{\emph{$\triangleright$ Case 1 : $\boldsymbol{k}_1 \leq N$.}}  Provided that $N$ is large enough, we have by Lemma \ref{lem:on_tue_tout}  that $(\boldsymbol{k},\boldsymbol{\sigma})$ is of type II.
 
 \medskip
 
  \noindent \underline{\emph{$\triangleright$ Case 2 : $N<\boldsymbol{k}_1 \leq  N^2$.}}  Provided that $N$ is large enough, we have by Lemma \ref{lem:on_tue_tout}  that 
  $$
  \boldsymbol{\sigma}_1 \mathds{1}_{\{ \boldsymbol{k_1} \}} + \cdots +\boldsymbol{\sigma}_q \mathds{1}_{\{ \boldsymbol{k_q} \}} = 0.
  $$
   Since $ \boldsymbol{k}_3 <N$ it implies that $\boldsymbol{\sigma}_1 = -\boldsymbol{\sigma}_2$ and $\boldsymbol{k_1} = \boldsymbol{k_2}$.  Therefore  $(\boldsymbol{k},\boldsymbol{\sigma})$ is of type III.
 
 \medskip

 \noindent \underline{\emph{$\triangleright$ Case 3 : $N^2<\boldsymbol{k}_1$.}} Without loss of generality, we assume that $\boldsymbol{\sigma}_1=1$.  Lemma \ref{lem:2grand} implies that,  provided that $N$ is large enough, 
 $$
 \boldsymbol{\sigma}_2 \boldsymbol{k}_2^\alpha \geq  \boldsymbol{k}_1^\alpha - C_q ( \boldsymbol{k}_3^\alpha + \boldsymbol{k}_1^{\alpha-1}) \geq N^{2\alpha} - C_q  (N^{\alpha} + N^{2(\alpha -1)}) > N^{\alpha}
 $$
 where $C_q>0$ is the implicit constant in Lemma \ref{lem:2grand}. Therefore $\boldsymbol{\sigma}_2=-1$ and $\boldsymbol{k}_2>N$, i.e. $(\boldsymbol{k},\boldsymbol{\sigma})$ is of type III.
 
 \end{proof}

Finally, in the following proposition, we deduce a result about the dynamics generated by the $\gamma$-resonant terms. 
\begin{proposition} \label{prop:est_poisson} Assume the non resonance condition \eqref{eq:nr_thm} on the frequencies $\omega$. For all $q\geq 3$, all $N\gtrsim_q 1$, all $n\geq 2+\nu$ and all $s\in [\nu+2,n]$, setting
 $$
 \gamma = N^{-2 a_q - q \alpha\beta - 2}  ,
 $$
 we have that for all $\gamma$-resonant homogeneous polynomial $Q \in \mathscr{H}_{q}^{n,\nu}$ of norm $\| Q\|_{\mathscr{H}_q^{n,\nu}} = 1$ and all $u\in H^s$
 $$
 \big|\{ \| \Pi_{\leq N} \cdot \|_{H^s}^2, Q   \}(u) \big| \lesssim_{q}  \| \Pi_{>N}u \|_{H^s}^2 \|  u \|_{H^{1+\nu}}^{q-2},
 $$
 $$
  \big| \{ \| \Pi_{> N} \cdot \|_{\ell^2}^2,Q \}(u)  \big| \lesssim_{q}  \| \Pi_{>N} u \|_{\ell^2}^2 \| \Pi_{> N} u \|_{H^{1+\nu}} \|  u \|_{H^{1+\nu}}^{q-3},
 $$
 $$
   \big| \{ \| \Pi_{> N} \cdot \|_{H^s}^2,Q \}(u)  \big|  \lesssim_{q}  \| \Pi_{>N}  u \|_{H^s}^{2-\frac1s} \| \Pi_{> N} u \|_{\ell^2}^{\frac1s}   \|  u \|_{H^{2+\nu}}^{q-2} +  \| \Pi_{>N}u  \|_{H^s}^2 \| \Pi_{>N}u  \|_{H^{1+\nu}} \|  u \|_{H^{2+\nu}}^{q-3}.
 $$
\end{proposition}
\begin{proof}  By Lemma \ref{lem:tous_les_types}, it suffices to consider all the types separately. 

\medskip

 \noindent \underline{\emph{$\triangleright$ Case 1 : $Q$ is supported on indices of type {\rm II}.}} In that case $Q$ commutes with all the super actions $J_k$ for $k \in \mathbb{N}$. A fortiori it commutes with $\| \Pi_{\leq N} \cdot \|_{H^s}^2,\| \Pi_{>N} \cdot \|_{\ell}^2$ and $\| \Pi_{>N} \cdot \|_{H^s}^2$ (i.e. the considered Poisson bracket are identically equal to $0$).

\medskip

 \noindent \underline{\emph{$\triangleright$ Case 2 : $Q$ is supported on indices of type {\rm I}.}} A minor extension of the proof of Lemma \ref{lem:vf} (i.e. paying attention to the support of the functions) provides the estimate
 $$
 \| \nabla Q(u) \|_{H^{s_*}}  \lesssim_{q}  \|\Pi_{>N} u\|_{H^{s_*}} \| \Pi_{>N} u\|_{H^{1+\nu}} \| u\|_{H^{1+\nu}}^{q-3}, \quad \mathrm{for} \quad s_* \in [0,n].
 $$
 Thus all the Poisson bracket estimates follow by Cauchy-Schwarz inequality.
 
\medskip

  \noindent \underline{\emph{$\triangleright$ Case 3 : $Q$ is supported on indices of type {\rm III}.}} First we note that $ \{ \| \Pi_{> N} \cdot \|_{\ell^2}^2,Q \} = 0$. Then, using the  mean value inequality and the Young convolution inequality, we have
  \begin{equation*}
  \begin{split}
  |\{ \| \Pi_{> N} \cdot \|_{H^s}^2,Q \}(u)|
   =& \big| \sum_{\boldsymbol{k}\in \mathbb{N}^q} \sum_{\boldsymbol{\sigma} \in \{-1 , 1\}^q }  (\boldsymbol{\sigma}_1^\star (\boldsymbol{k}_1^\star)^{2s} + \boldsymbol{\sigma}_2^\star (\boldsymbol{k}_2^\star)^{2s}  ) Q_{\boldsymbol{k}}^{\boldsymbol{\sigma}}(\Pi_{\boldsymbol{k}_1}^{\boldsymbol{\sigma}_1} u, \cdots,  \Pi_{\boldsymbol{k}_q}^{\boldsymbol{\sigma}_q} u ) \big| \\
  \lesssim_q&  \sum_{\boldsymbol{k}_1 \geq \boldsymbol{k}_2 >N \geq  \boldsymbol{k}_3 \cdots \geq  \boldsymbol{k}_q}  \Gamma_{\boldsymbol{k}} \Big(   \frac{\boldsymbol{k}_2 }{ \boldsymbol{k}_1 } \Big)^n \big( \prod_{3\leq \ell \leq q} \boldsymbol{k}_\ell \big)^\nu    (\boldsymbol{k}_1^{2s} - \boldsymbol{k}_2^{2s}) \prod_{\ell=1}^q \| \Pi_{\boldsymbol{k}_\ell}^{\boldsymbol{\sigma}_\ell} u \|_{\ell^2}    \\
  \lesssim_q& \sum_{ \substack{\boldsymbol{k}_1 \geq \boldsymbol{k}_2 >N \geq  \boldsymbol{k}_3 \cdots \geq  \boldsymbol{k}_q \\ \boldsymbol{\varsigma} \in \{-1,1\}^q}}  \frac{(\boldsymbol{k}_1  - \boldsymbol{k}_2)    \boldsymbol{k}_1^{2s-1}}{\langle  \boldsymbol{\varsigma}_1  \boldsymbol{k}_1 + \cdots +  \boldsymbol{\varsigma}_q  \boldsymbol{k}_q \rangle^3}  \Big(   \frac{\boldsymbol{k}_2 }{ \boldsymbol{k}_1 } \Big)^n \big( \prod_{3\leq \ell \leq q} \boldsymbol{k}_\ell \big)^\nu    \prod_{\ell=1}^q \| \Pi_{\boldsymbol{k}_\ell}^{\boldsymbol{\sigma}_\ell} u \|_{\ell^2} \\
  \lesssim_q& \sum_{ \substack{\boldsymbol{k}_1 \geq \boldsymbol{k}_2 >N \geq  \boldsymbol{k}_3 \cdots \geq  \boldsymbol{k}_q \\ \boldsymbol{\varsigma} \in \{-1,1\}^q}} \frac{ \boldsymbol{k}_3 \boldsymbol{k}_1^{s-\frac12} \boldsymbol{k}_2^{s-\frac12} }{ \langle  \boldsymbol{\varsigma}_1  \boldsymbol{k}_1 + \cdots +  \boldsymbol{\varsigma}_q  \boldsymbol{k}_q \rangle^{2}}    \big( \prod_{3\leq \ell \leq q} \boldsymbol{k}_\ell \big)^{\nu}    \prod_{\ell=1}^q \| \Pi_{\boldsymbol{k}_\ell}^{\boldsymbol{\sigma}_\ell} u \|_{\ell^2} \\
  \lesssim_q&  \| \Pi_{> N} u \|_{H^{s-1/2}}^{2}   \|  u \|_{H^{2+\nu}}^{q-2} \\
  \lesssim_{q}& \| \Pi_{>N} u\|_{H^s}^{2 - \frac1{s}}   \| \Pi_{>N} u\|_{\ell^2}^{\frac1s}  \|  u \|_{H^{2+\nu}}^{q-2}
\end{split}  
  \end{equation*}
where the last estimate comes from the H\"older inequality.
Finally, applying the Young convolution inequality as previously, we have
  \begin{equation*}
  \begin{split}
  |\{ \| \Pi_{\leq N} \cdot \|_{H^s}^2,Q \}(u)|& = \big| \sum_{\boldsymbol{k}\in \mathbb{N}^q} \sum_{\boldsymbol{\sigma} \in \{-1 , 1\}^q }  (\boldsymbol{\sigma}_3^\star (\boldsymbol{k}_3^\star)^{2s} + \cdots + \boldsymbol{\sigma}_q^\star (\boldsymbol{k}_q^\star)^{2s}  ) Q_{\boldsymbol{k}}^{\boldsymbol{\sigma}}(\Pi_{\boldsymbol{k}_1}^{\boldsymbol{\sigma}_1} u, \cdots,  \Pi_{\boldsymbol{k}_q}^{\boldsymbol{\sigma}_q} u )\big| \\
  &\lesssim_q  \sum_{\boldsymbol{k}_1 \geq \boldsymbol{k}_2 >N \geq  \boldsymbol{k}_3 \cdots \geq  \boldsymbol{k}_q} \Gamma_{\boldsymbol{k}}   \boldsymbol{k}_3^{2s} \big( \prod_{3\leq \ell \leq q} \boldsymbol{k}_\ell \big)^{\nu} \prod_{\ell=1}^q \| \Pi_{\boldsymbol{k}_\ell}^{\boldsymbol{\sigma}_\ell} u \|_{\ell^2} \\
  &\lesssim_{q}\sum_{\boldsymbol{k}_1 \geq \boldsymbol{k}_2 >N \geq  \boldsymbol{k}_3 \cdots \geq  \boldsymbol{k}_q} \Gamma_{\boldsymbol{k}} \boldsymbol{k}_1^{s} \boldsymbol{k}_2^{s}  \big( \prod_{3\leq \ell \leq q} \boldsymbol{k}_\ell \big)^{\nu} \prod_{\ell=1}^q \| \Pi_{\boldsymbol{k}_\ell}^{\boldsymbol{\sigma}_\ell} u \|_{\ell^2} \\
  &\lesssim_q \| \Pi_{>N} u\|_{H^{s}}^2   \|  u \|_{H^{1+\nu}}^{q-2} . 
\end{split}  
  \end{equation*}

\end{proof}

\section{Dynamical consequences : proof of Theorem \ref{thm:main}}

From now, we are in the setting of Theorem \ref{thm:main}. In particular, we assume all the related hypothesis. The only difference is that we work with $H^s$ spaces instead of $h^s$ spaces, we recall that since for all $s\geq 0$
$$
\| \cdot \|_{h^{s/\alpha\beta}} \sim_s \| \cdot \|_{H^s}
$$
this is equivalent. We just set
$$
\mathfrak{s}_0 := s_0\alpha\beta, \quad \mathfrak{s}_{\mathrm{min}} =   s_{\mathrm{min}}\alpha\beta
$$
so that the indices $s_0,s_{\mathrm{min}}$ of the theorem becomes $\mathfrak{s}_0,\mathfrak{s}_{\mathrm{min}}$ in this setting. We note that by definition of $s_{\mathrm{min}}$, we have 
$$
\mathfrak{s}_{\mathrm{min}} \geq \max(\mathfrak{s}_{0} ,\nu +2).
$$
  Without loss of generality, we assume that  that  the sequence $(a_q)_q$ (related to the assumptions \eqref{eq:nr_thm}) is increasing and that $\mathcal{U}$ is a ball, i.e. their exists $\rho_0>0$ such that
$$
\mathcal{U} = B_{H^{\mathfrak{s}_0}}(0,\rho_0).
$$
Moreover, we note that thanks to the tame estimate  \eqref{eq:lestimee_tame_pas_facile_a_pas_oublier}, it suffices to prove Theorem \ref{thm:main} for one special value of $s$ to get it for all larger $s$ (see Step $2$ below for more details).

\medskip

\noindent \underline{\emph{$\triangleright$ Step 1: Parameters.}} Let $r\geq 17$, $\mathfrak{s}_c\geq \mathfrak{s}_{\mathrm{min}}$ and set
\begin{equation*}
 \mathfrak{s}  :=  \mathfrak{s}_c + 9 r^2 \big(2a_r + r \alpha\beta + 2 \big) + \alpha.
\end{equation*}

\medskip

Let $\varepsilon \in (0,1)$. All along the proof we will add smallness assumptions on $\varepsilon$ with respect to $r$ and $\mathfrak{s}_c$. We set
$$
N := \varepsilon^{-\frac{8r}{\mathfrak{s} -\mathfrak{s}_c }} \quad \mathrm{and} \quad \gamma := N^{-(2a_r + r \alpha\beta + 2 )}.
$$
We note that, we have the useful estimate
\begin{equation}
\label{eq:sympa_non}
\gamma^{-r} = N^{r(2a_r + r \alpha\beta + 2 )} \leq N^{\frac{\mathfrak{s}  - \mathfrak{s}_c}{9r}} = \varepsilon^{- \frac{8}9}.
\end{equation}
We define for $u\in H^{\mathfrak{s}_0}$
$$
P(u) := \sum_{q=3}^r P^{(q)}(u) \quad \mathrm{where} \quad P^{(q)}(u):=  \frac1{q !} \mathrm{d}^q G(0)(u,\cdots,u). 
$$
Since $g$ is a smooth function from $B_{H^{\mathfrak{s}_c}}(0,\rho_0)$ into $H^{\mathfrak{s}_c}$, it implies that provided that $5\varepsilon<\rho_0$ 
\begin{equation}
\label{eq:estime_terme_de_reste}
\forall u\in B_{H^{\mathfrak{s}_c}}(0,5\varepsilon), \quad \| \nabla (G-P)(u) \|_{H^{\mathfrak{s}_c}} \lesssim_{r,s_c} \varepsilon^r.
\end{equation}
Then, we set $n\in \mathbb{N}$ as the integer satisfying
$$
n +1\leq \mathfrak{s} < n+2.
$$

 Finally, we have to prove that $P\in \mathscr{I}_{\leq r}^{n,\nu}$ (i.e. $P^{(q)} \in \mathscr{H}_q^{n,\nu}$ for all $q\geq 3$). We define the coefficients $ (P^{(q)})_{\boldsymbol{k}}^{\boldsymbol{\sigma}}$ by the polarization formula\footnote{coming from the Fourier inversion formula in the Abelian group $(\mathbb{Z}/2\mathbb{Z})^q$.}
 $$
 (P^{(q)})_{\boldsymbol{k}}^{\boldsymbol{\sigma}} ( \Pi_{\boldsymbol{k}_1} u^{(1)}, \cdots, \Pi_{\boldsymbol{k}_q} u^{(q)} ) =2^{-q}  \sum_{\boldsymbol{\eta} \in \{0,1\}^q }\big( \prod_{\ell=1}^q( -\ic \boldsymbol{\sigma}_\ell)^{\boldsymbol{\eta}_\ell} \big) \frac{\mathrm{d}^q G(0)}{q!}( \ic^{\boldsymbol{\eta}_1} \Pi_{\boldsymbol{k}_1}^{\boldsymbol{\sigma}_1}  u^{(1)},\cdots, \ic^{\boldsymbol{\eta}_q} \Pi_{\boldsymbol{k}_q}^{\boldsymbol{\sigma}_q}  u^{(q)})
 $$
for $q \geq 3$, $\boldsymbol{k}\in \mathbb{N}$, $\boldsymbol{\sigma}\in \{-1,1\}^q$ and $u^{(1)},\cdots,u^{(q)} \in \mathbb{C}^\mathbb{N}$. By a tedious calculation, it can be checked that $(P^{(q)})_{\boldsymbol{k}}^{\boldsymbol{\sigma}} \in \mathscr{L}_{\boldsymbol{k}}$ (i.e. that it is $\mathbb{C}-$multilinear) and that 
\begin{equation}
\label{eq:dec_P}
P (u)= \sum_{q=3}^r \sum_{\boldsymbol{k}\in \mathbb{N}^q } \sum_{\boldsymbol{\sigma} \in \{-1,1\}^q }   (P^{(q)})_{\boldsymbol{k}}^{\boldsymbol{\sigma}}  (  \Pi_{\boldsymbol{k}_1}^{\boldsymbol{\sigma}_1}  u,\cdots, \Pi_{\boldsymbol{k}_q}^{\boldsymbol{\sigma}_q}  u) 
\end{equation}
 when $u\in \mathbb{C}^\mathbb{N}$ with finite support. Moreover, we note that the reality and the symmetry condition are obvious. Finally, we note that thanks to the polarization formula above and the bound \eqref{eq:estimates_thm} we assumed on $\mathrm{d}^q G(0)$, we have that $\| P \|_{\mathscr{I}_{1}^{n,\nu}} \lesssim_{n,q} 1$ (and so that \eqref{eq:dec_P} holds in $H^{1+\nu}$ by density).

\medskip

\noindent \underline{\emph{$\triangleright$ Step 2: Setting of the bootstrap.}} Without loss of generality,  we only focus on positive times. Let $u^{(0)} \in H^{\mathfrak{s}}$ such that $\| u^{(0)} \|_{H^{\mathfrak{s}_c}} = \varepsilon$ and $\| u^{(0)} \|_{H^{\mathfrak{s}}} \leq 1$. Let 
$$
u\in C^0([0,T^{\max}); H^{\mathfrak{s}}) \cap C^1([0,T^{\max}); H^{\mathfrak{s}-\alpha}) 
$$
be the maximal solution of \eqref{eq:ham-pde} with initial datum $u(0) = u^{(0)}$. We note that thanks to the tame estimate \eqref{eq:lestimee_tame_pas_facile_a_pas_oublier}, we have 
\begin{center}
either $T^{\max}=+\infty$ or $\displaystyle \lim_{t \to T^{\max}} \| u(t) \|_{H^{\mathfrak{s}_0}} = \rho_0$.
\end{center}

\medskip

Then, we set
\begin{equation}
\label{eq:on_voit_bien_le_temps}
T_\varepsilon = \sup \Big\{t\leq \min(T^{\max},\varepsilon^{-\frac{r}{9 \mathfrak{s}_c}}) \ | \ \sup_{\tau \leq t} \| u(\tau) \|_{H^{\mathfrak{s}_c}} < 5 \varepsilon\Big\}.
\end{equation} 
Since $r$ can be chosen arbitrarily large, in order to prove Theorem \ref{thm:main}, it suffices\footnote{note that here we use that $\mathfrak{s}_2$ does not depend on $r$.} to prove that $T_\varepsilon = \varepsilon^{-\frac{r}{9\mathfrak{s}_c}}$.
By a continuity argument, it means that it suffices to prove that for all $t< T_\varepsilon$, we have $\| u(t) \|_{H^{\mathfrak{s}_c}} \leq 4\varepsilon$.

\medskip

\noindent \underline{\emph{$\triangleright$ Step 3: Normal form.}} Now, we apply the Birkhoff normal form Theorem \ref{thm:birk}. We note that by \eqref{eq:sympa_non}, we have that
$$
\varepsilon_{n,\gamma,\chi}\gtrsim \gamma \geq \varepsilon^{\frac{8}{9r}} \gg \varepsilon.
$$
So, by Lemma \ref{lem:est_flow},  it makes senses to define, for all $t\leq T_\varepsilon$,
$$
v(t) := \Phi_\chi^1(u(t)).
$$
Since $u\in C^1((0,T_\varepsilon); H^{\mathfrak{s}_c})$ (because $\mathfrak{s} -\alpha \geq \mathfrak{s}_c$) and $\Phi_\chi^1$ is smooth on $B_{H^{\mathfrak{s}_c}}(0,\varepsilon_{n,\gamma,\chi})$ (because $\mathfrak{s}_c \geq \nu +1 $), we have that 
$$
v\in C^1((0,T_\varepsilon); H^{\mathfrak{s}_c}).
$$
Moreover, since $\chi$ is symplectic, it satisfies
\begin{equation}
\label{eq:eq_on_v}
\ic \partial_t v= \nabla (Z_2+Q+R)(v) - \underbrace{\ic \mathrm{d}\Phi_\chi^1 (u)( \ic \nabla (G-P)(u) )}_{=: f}. 
\end{equation}
Using the estimates of Lemma \ref{lem:est_flow} on $\Phi_\chi^1$, the bootstrap assumption and the estimate \eqref{eq:estime_terme_de_reste} on $\nabla (G-P)(u)$, we have that for all $t<T_\varepsilon$
\begin{equation}
\label{eq:eq_tot1}
\| v(0) \|_{H^{\mathfrak{s}}} \leq 2 , \quad \| v(0) \|_{H^{\mathfrak{s}_c}} \leq 2 \varepsilon, \quad \| v(t) \|_{H^{\mathfrak{s}_c}} \leq 6 \varepsilon, \quad  \|f(t)\|_{H^{\mathfrak{s}_c}} \lesssim_{r,\mathfrak{s}_c} \varepsilon^{r} .
\end{equation} 
We note that in particular
\begin{equation}
\label{eq:init_grand}
\| \Pi_{>N} v(0) \|_{H^{\mathfrak{s}_c}} \leq \| v(0) \|_{H^{\mathfrak{s}}}  N^{-(\mathfrak{s} - \mathfrak{s}_c)} \leq 2   N^{-(\mathfrak{s} - \mathfrak{s}_c)} = 2  \varepsilon^{8r}.
\end{equation} 
Moreover, using the bound on $\| v(t) \|_{H^{\mathfrak{s}_c}} $, the estimate \eqref{eq:est_R} on $\nabla R$ and the estimate \eqref{eq:sympa_non} on $\gamma^{-r}$, we have that for all $t<T_\varepsilon$
\begin{equation}
\label{eq:eq_rem2}
\|\nabla R(v(t)) \|_{H^{\mathfrak{s}_c}} \leq \varepsilon^{r-1}.
\end{equation}
Finally, using that $\Phi_\chi^1$ is close to the identity as proven in Lemma \ref{lem:est_flow}, we point out that it suffices to prove that $\| v(t) \|_{H^{\mathfrak{s}_c}} \leq 3 \varepsilon$ for $t<T_\varepsilon$ to deduce that $\| u(t) \|_{H^{\mathfrak{s}_c}} \leq 4\varepsilon$ and so to conclude the proof.

\medskip

\noindent \underline{\emph{$\triangleright$ Step 4: system of estimates.}} Since $v\in C^1((0,T_\varepsilon); H^{\mathfrak{s}_c})$ solve the equation \eqref{eq:eq_on_v} and $Z_2$ commutes with $\| \cdot \|_{H^{\mathfrak{s}_c }}^2 $, we have
$$
\partial_t \| \Pi_{>N} v \|_{H^{\mathfrak{s}_c }}^2 = \{  \| \Pi_{>N} \cdot \|_{H^{\mathfrak{s}_c }}^2, Q \} +2 (\ic D^{\mathfrak{s}_c} \Pi_{>N} v,D^{\mathfrak{s}_c} (\nabla R(v) + f))_{\ell^2}.
$$
where for all $z\in \mathbb{C}^{\mathbb{N}}$, $D^{\mathfrak{s}_c}z := (k^{\mathfrak{s}_c}z_k)_{k\in \mathbb{N}}$. Then by using the Cauchy--Schwarz inequality and the estimates \eqref{eq:eq_tot1}, \eqref{eq:eq_rem2} on the remainder terms (and that $\|v(t) \|_{H^{\mathfrak{s}_c }} \lesssim \varepsilon$; see \eqref{eq:eq_tot1}), it comes
$$
 |(\ic D^{\mathfrak{s}_c} \Pi_{>N} v,D^{\mathfrak{s}_c} (\nabla R(v) + f))_{\ell^2}| \lesssim  \varepsilon^{r}.
$$
Since $\varepsilon \leq \gamma$ (by \eqref{eq:sympa_non}), $Q$ is $\gamma$-resonant and satisfies $\| Q\|_{\mathscr{I}_{\gamma}^{n,\nu}} \lesssim_{r,n} 1 $ (by Theorem \ref{thm:birk}),  we get by Proposition \ref{prop:est_poisson} that
$$
 |\{  \| \Pi_{>N} \cdot \|_{H^{\mathfrak{s}_c }}^2, Q \}(v) |  \lesssim_{r,n} \varepsilon \| \Pi_{>N}  v \|_{H^{\mathfrak{s}_c }}^{2-\frac1{\mathfrak{s}_c }} \| \Pi_{> N} v \|_{\ell^2}^{\frac1{\mathfrak{s}_c }}   +  \| \Pi_{>N}v  \|_{H^{\mathfrak{s}_c }}^3 .
$$
Thus we have proven that
\begin{equation}
\label{eq:s0_grand}
\partial_t \| \Pi_{>N} v \|_{H^{\mathfrak{s}_c }}^2  \lesssim_{r,n} \varepsilon \| \Pi_{>N}  v \|_{H^{\mathfrak{s}_c }}^{2-\frac1{\mathfrak{s}_c }} \| \Pi_{> N} v \|_{\ell^2}^{\frac1{\mathfrak{s}_c }}   +  \| \Pi_{>N}v  \|_{H^{\mathfrak{s}_c }}^3 + \varepsilon^r.
\end{equation}
Then, by using the same arguments we get that
\begin{equation}
\label{eq:s0_petit}
\partial_t \| \Pi_{\leq N} v \|_{H^{\mathfrak{s}_c }}^2 \lesssim_{r,n} \varepsilon  \| \Pi_{>N}  v \|_{H^{\mathfrak{s}_c }}^{2} +  \varepsilon^r .
\end{equation}
\begin{equation}
\label{eq:l2_grand}
\partial_t \| \Pi_{>N} v \|_{\ell^2}^2 \lesssim_{r,n} \| \Pi_{>N} v \|_{\ell^2}^2   \| \Pi_{>N}  v \|_{H^{\mathfrak{s}_c }} +  \varepsilon^r.
\end{equation}

\medskip

\noindent \underline{\emph{$\triangleright$ Step 5: last bootstrap and conclusion.}} In order to fully exploit the system of estimates (\eqref{eq:s0_grand}, \eqref{eq:s0_petit}, \eqref{eq:l2_grand}) we do a second bootstrap argument. We set
$$
T^* = \sup \big\{ t \leq T_\varepsilon \ | \ \forall \tau \leq t, \quad  \| \Pi_{\leq N} v(\tau) \|_{H^{\mathfrak{s}_c }}^2 \leq 3\varepsilon \quad \mathrm{and} \quad \| \Pi_{> N} v(\tau) \|_{H^{\mathfrak{s}_c }}^2 \leq \varepsilon^{\frac{r}8} \}
$$
and we note that by \eqref{eq:eq_tot1} and \eqref{eq:init_grand} we have $T^*>0$. We note that now, to conclude the proof, it suffices to prove that if $t<T^*$ then $\| \Pi_{\leq N} v(\tau) \|_{H^{\mathfrak{s}_c }}^2 < 3\varepsilon$ and $\| \Pi_{> N} v(\tau) \|_{H^{\mathfrak{s}_c }}^2 < \varepsilon^{\frac{r}8} $.

\medskip

So let $t<T^*$. We note that 
$$
t<T^* \leq T_\varepsilon \leq \varepsilon^{-\frac{r}{9 \mathfrak{s}_c}} \leq \varepsilon^{-\frac{r}{18}}.
$$

\medskip

 First, by using \eqref{eq:s0_petit}, we get that
 $$
\partial_t \| \Pi_{\leq N} v(t) \|_{H^{\mathfrak{s}_c }}^2 \lesssim_{r,n} \varepsilon^{1+\frac{r}4}
$$
and so that (since $r\geq 17$)
$$
\| \Pi_{\leq N} v(t) \|_{H^{\mathfrak{s}_c }}^2 \leq T^*  \varepsilon^{\frac{r}4}\leq \varepsilon^{\frac{r}8} < \varepsilon^2 .
$$
Thus it suffices to focus on estimating $\| \Pi_{> N} v(t) \|_{H^{\mathfrak{s}_c }}^2$.

\medskip

Then by applying the Gr\"ownwall inequality in \eqref{eq:l2_grand}, we get a constant $C_{q,n}>0$ depending only on $q$ and $n$ such that
$$
\| \Pi_{>N} v(t) \|_{\ell^2}^2  \leq \| \Pi_{>N} v(0) \|_{\ell^2}^2 e^{ C_{q,n} t \varepsilon^{\frac{r}8}} + C_{q,n}\int_{0}^t  e^{ C_{q,n} (t-\tau) \varepsilon^{\frac{r}8}}  \varepsilon^r   \mathrm{d}\tau .
$$
Now recalling that that by \eqref{eq:init_grand} $\| \Pi_{>N} v(0) \|_{\ell^2}^2 \leq \varepsilon^r$, we get that (provided that $\varepsilon$ is small enough)
$$
\| \Pi_{>N} v(t) \|_{\ell^2}^2  \leq 2   \varepsilon^r + C_{q,n} T^*  \varepsilon^r \leq \varepsilon^{\frac{r}2}. 
$$
Then plugging this estimate in \eqref{eq:s0_grand} and using the bootstrap assumption, we get that
$$
\partial_t \| \Pi_{>N} v (t)\|_{H^{\mathfrak{s}_c }}^2  \lesssim_{r,n} \varepsilon^{\frac{r}8 (2-\frac1{\mathfrak{s}_c })} \varepsilon^{\frac{r}{4\mathfrak{s}_c }}   + \varepsilon^{\frac{3r}{8}}.
$$
and so using the bound \eqref{eq:init_grand} on $\| \Pi_{>N} v(0) \|_{H^{\mathfrak{s}_c }}^2$, we get that
$$
\| \Pi_{>N} v(t) \|_{H^{\mathfrak{s}_c }}^2 \lesssim_{r,n} \varepsilon^{r} + T^* \varepsilon^{\frac{r}{4}} ( \varepsilon^{\frac{r}{8}} + \varepsilon^{\frac{r}{8\mathfrak{s}_c }}  ) \lesssim_{r,n} \varepsilon^{r} +  \varepsilon^{-\frac{r}{9\mathfrak{s}_c }} \varepsilon^{\frac{r}{4}} ( \varepsilon^{\frac{r}{8}} + \varepsilon^{\frac{r}{8\mathfrak{s}_c }}  ) \ll \varepsilon^{\frac{r}4}.
$$
This last estimate ensures that, provided that $\varepsilon$ is small enough, $\| \Pi_{>N} v(t) \|_{H^{\mathfrak{s}_c }}< \varepsilon^{\frac{r}8}$ which concludes the proof of Theorem \ref{thm:main}.

\appendix

\section{Multilinear estimates} Let $\mathcal{M}$ be a
 Riemannian manifold which is either $\mathbb{R}^d$ or a boundaryless compact manifold. Let us consider an unbounded operator of the form $-\Delta+V$ on 
$L^2(\mathcal{M})$ in which $V$ is a positive smooth potential. We assume that $-\Delta+V$ is essentially self-adjoint and has pure-point spectrum. 

\medskip

Let us write a few words about Sobolev spaces. There is no issue to define the Sobolev space $W^{N,\infty}(\mathcal{M})$ (eventually using a finite atlas in the manifold framework) as the space of functions such that the $N$ first derivatives are bounded. For any fixed number $s_0>\frac{d}{2}$, due to the Sobolev embedding $H^{s_0}(\mathcal{M})\subset L^\infty(\mathcal{M})$, we may deduce the following one
\begin{equation}\label{soboi}
H^{N+s_0}(\mathcal{M}) \subset W^{N,\infty}(\mathcal{M}).
\end{equation}

\medskip

By positivity and self-adjointness, one may define the following operator by functional calculus
$$
P:=\sqrt{-\Delta+V}.
$$
  In all the examples of our paper, the following inequalities hold true
\begin{equation}\label{sobo-comp}
\|u\|_{H^s(\mathcal{M})} \lesssim \|P^s u\|_{L^2(\mathcal{M})}, 
\qquad \forall s\in \mathbb{N}, \ \forall u\in \mbox{Dom}(P^s).
\end{equation}
We shall set an abstract definition of admissible operators but before stating it, let us write two examples of admissible operators we have in mind :
\begin{itemize}
\item $-\Delta+V$ on a boundaryless Riemannian compact manifold (see \cite{DS06}). For \eqref{sobo-comp}, this is actually and equivalence. For complements about checking Definition \ref{abs}, one may 
 see a proof due to J.M. Delort presented in the thesis of F. Monzani \cite[Section 2.3]{Mon24}.
\item $-\Delta+V$ on $\mathbb{R}^d$ and $V$ being a positive-definite quadratic form (see \eqref{equivHS} and \cite[Section 3]{Brun23}). Note that, in this case, in the rest of the paper, the quadratic potential is denoted by $Q$ instead of $V$.
\end{itemize}

\begin{definition}\label{abs}
The linear operator $-\Delta+V$ is admissible if the following assumptions is true : there exists an integer $\nu>0$ such that for any $a\in \mathcal{C}^\infty (\mathcal{M})$ with bounded derivatives, the operator $A$ of multiplication by $a$ satisfies for any $n\in \mathbb{N}$
\begin{equation}\label{adP}
\sup\limits_{0\leq \ell\leq n} \| \mathrm{Ad}_P^\ell (A)\|_{L^2(\mathcal{M})\rightarrow L^2(\mathcal{M})}
\leq C_n \|a\|_{W^{n+\nu,\infty}(\mathcal{M})}
\end{equation}
in which $\mathrm{Ad}_P:L \mapsto [P,L]$ is the adjoint mapping.
%\item $P$ satisfies a semi-classical Bernstein inequality : there exists an integer $\nu'$ such that for any $\chi\in C^\infty(\R)$ with compact support, we may write
%\begin{equation*}
%\|\chi (h P)\|_{L^2(\mathcal{M})\rightarrow W^{N,\infty}(\mathcal{M})} \leq C h^{-N-\nu'}                              \qquad    0<h\leq 1.
%\end{equation*}
%\end{enumerate}
\end{definition}

Let us now explain the consequences for multilinear estimates. Let us denote $(I_k)_{k\in \mathbb{N}}$ a sequence of bounded intervals of $\mathbb{R}$ which satisfy (for some constant $C>0$):
\begin{equation}\label{thet0}     \min(I_k) = Ck+O(1)   \qquad k\rightarrow +\infty 
\end{equation}
\begin{equation}\label{thet}
\max(I_k)-\min(I_k)=O(1).
\end{equation}
We shall need to select once and for all an element of $I_k$, say $\theta_k\in I_k$. We also have the asymptotic
 \begin{equation*}
 \theta_k=Ck+O(1)
 \end{equation*} and thus there exists a positive constant $\tau>0$ fulfilling
\begin{equation}\label{diff}
 \boldsymbol{k}_1- \boldsymbol{k}_2 \geq \tau \qquad \implies \qquad \theta_{ \boldsymbol{k}_1}-\theta_{ \boldsymbol{k}_2} \gtrsim  \boldsymbol{k}_1- \boldsymbol{k}_2.
\end{equation}
The goal of this part is to prove the next result in which, for all $k\in \mathbb{N}$, $\Pi_{\mathcal E_k}:L^2(\mathcal{M})\rightarrow L^2(\mathcal{M})$ is the orthonormal projector on
$\mathcal E_k$ the subspace spanned by the eigenvectors of $P$ whose associated eigenvalue belong to $I_k$,
i.e.
$$
\Pi_{\mathcal{E}_{k}}=\sum_{\lambda \in I_{k}} \Pi_{\lambda} 
$$
where $\Pi_{\lambda}$ denotes the spectral projector associated with the eigenvalue $\lambda$ of $P$.
\begin{proposition}\label{prop:multi}
There exists $\nu \geq 0$ such that for all $q \geq 3$, for all $n\geq 4$, all $\boldsymbol{k} \in \mathbb{N}^q$ and all $u^{(1)},\cdots,u^{(q)}\in L^2(\mathcal M)$, the following two sets of inequalities hold true
\begin{equation}
\label{eq:DS_estim-KG}
   \big| \int_{\mathcal M} (\Pi_{\mathcal E_{\boldsymbol{k}_{1}}} u^{(1)})\cdots (\Pi_{\mathcal E_{\boldsymbol{k}_{q}}}u^{(q)})\mathrm{d} x \big|  \lesssim_{n,q}  \frac{(\boldsymbol{k}_{3}^\star)^{\nu +n} (\boldsymbol{k}_{4}^\star)^{\nu } \dots (\boldsymbol{k}_{q}^\star)^{\nu }  }{ (\boldsymbol{k}_{1}^\star - \boldsymbol{k}_{2}^\star + \boldsymbol{k}_{3}^\star )^{n}} \prod_{\ell=1}^q \| u^{(\ell)} \|_{L^2}
\end{equation}
and
\begin{equation}
\label{eq:DS_estim-KG2}
   \big| \int_{\mathcal M} (\Pi_{\mathcal E_{\boldsymbol{k}_{1}}} u^{(1)})\cdots (\Pi_{\mathcal E_{\boldsymbol{k}_{q}}}u^{(q)})\mathrm{d} x \big|  
\lesssim_{n,q}  \Gamma_{\boldsymbol{k}} \Big(\frac{\boldsymbol{k}_{2}^\star}{\boldsymbol{k}_{1}^\star}\Big)^{n}
\Big(\prod_{3\leq \ell\leq q} \boldsymbol{k}_\ell^\star \Big)^{\nu}
 \prod_{\ell=1}^q \| u^{(\ell)} \|_{L^2}
\end{equation}
where $\mathrm{d} x$ is the Riemannian volume on $\mathcal M$
and $\Gamma_{\boldsymbol{k}}$ 
is defined in \eqref{eq:def_gamma_k}. 
\end{proposition}

\begin{remark} $\bullet$ A proof of \eqref{eq:DS_estim-KG} in the compact setting is done in \cite{DS06} and involves the Helffer-Sj\"{o}strand formula (precisely Lemma 1.2.3 of \cite{DS06}). We give here an elementary alternative proof by exploiting that $-\Delta+V$ and hence its square root $P$ have pure-point spectrum.\\
 $\bullet$ In \cite{Brun23}, a Bernstein inequality is also proved in order to follow the proof of Delort and Szeftel. But it turns out that we may overcome such an inequality thanks to the Sobolev embedding \eqref{soboi}.
\end{remark}

Let us begin the proof of Proposition \ref{prop:multi}. We first need the following result.
\begin{lemma} Assume $\boldsymbol{k}_1-\boldsymbol{k}_2\geq \tau$ as in \eqref{diff}. For any bounded self-adjoint operator $A:L^2(\mathcal{M})\rightarrow L^2(\mathcal{M})$, the following bound holds true
\begin{equation}\label{adp2}
|\langle A \Pi_{\mathcal{E}_{\boldsymbol{k}_1}} u^{(1)}, \Pi_{\mathcal{E}_{\boldsymbol{k}_2}}u^{(2)}\rangle | \lesssim_n 
\sup\limits_{0\leq \ell\leq n} \| (\mathrm{Ad}_P^\ell (A)\|_{L^2(\mathcal{M})\rightarrow L^2(\mathcal{M})} 
  \frac{ \| u^{(1)}\|_{L^2(\mathcal{M})}\|u^{(2)}\|_{L^2(\mathcal{M})} }{(    \boldsymbol{k}_1 - \boldsymbol{k}_2  )^n}  
\end{equation}
where $\langle \cdot,\cdot \rangle $ denotes the $L^2(\mathcal{M};\mathbb{R})$ scalar product. 
\end{lemma}
\begin{proof}

We shall prove \eqref{adp2} by induction on $n$. For $n=0$, \eqref{adp2} is merely the Cauchy-Schwarz inequality. Let us assume that \eqref{adp2} is true for $n$ and let us prove it for $n+1$. For any eigenvalue $\lambda$ of $P=\sqrt{-\Delta+V}$, let us denote by $\Pi_{\lambda}$ the spectral projector associated with the eigenvalue $\lambda$ of $P$. By definition, we may write
\begin{equation*}
\Pi_{\mathcal{E}_{\boldsymbol{k}_1}}=\sum_{\lambda \in I_{\boldsymbol{k}_1}} \Pi_{\lambda} \qquad \mbox{and} \qquad P\Pi_{\lambda}=\lambda \Pi_{\lambda}
\end{equation*}
which may be used to reformulate 
\begin{eqnarray*}
\langle [A,P] \Pi_{\mathcal{E}_{\boldsymbol{k}_1}}u^{(1)}, \Pi_{\mathcal{E}_{\boldsymbol{k}_2}}u^{(2)} \rangle  & =& 
\langle AP \Pi_{\mathcal{E}_{\boldsymbol{k}_1}}u^{(1)}, \Pi_{\mathcal{E}_{\boldsymbol{k}_2}}u^{(2)} \rangle-
\langle  PA \Pi_{\mathcal{E}_{\boldsymbol{k}_1}}u^{(1)}, \Pi_{\mathcal{E}_{\boldsymbol{k}_2}}u^{(2)} \rangle\\
& =& \langle P \Pi_{\mathcal{E}_{\boldsymbol{k}_1}}u^{(1)}, A \Pi_{\mathcal{E}_{\boldsymbol{k}_2}}u^{(2)} \rangle-
\langle A \Pi_{\mathcal{E}_{\boldsymbol{k}_1}}u^{(1)}, P\Pi_{\mathcal{E}_{\boldsymbol{k}_2}}u^{(2)} \rangle\\
& = & \sum_{ \substack{\lambda \in I_{\boldsymbol{k}_1} \\ \mu \in I_{\boldsymbol{k}_2}}} (\lambda-\mu) \langle \Pi_{\lambda}u^{(1)},A\Pi_{\mu} u^{(2)} \rangle.
\end{eqnarray*}
We now formalize the idea that the eigenvalues in $I_{\boldsymbol{k}_1}$ and $I_{\boldsymbol{k}_2}$ can be compared to $\theta_{\boldsymbol{k}_1}$ and $\theta_{\boldsymbol{k}_2}$. More precisely, the last right-hand side can be written $S_0+S_1+S_2$ with 
\begin{eqnarray*}
S_0 &:= &\sum_{ \substack{\lambda \in I_{\boldsymbol{k}_1} \\ \mu \in I_{\boldsymbol{k}_2}}} (\theta_{\boldsymbol{k}_1}-\theta_{\boldsymbol{k}_2}) \langle \Pi_{\lambda}u^{(1)},A\Pi_{\mu} u^{(2)} \rangle=(\theta_{\boldsymbol{k}_1}-\theta_{\boldsymbol{k}_2}) \langle \Pi_{\mathcal{E}_{\boldsymbol{k}_1}}u^{(1)},A\Pi_{\mathcal{E}_{\boldsymbol{k}_2}}u^{(2)}\rangle, \\
S_1 &:=& \sum_{ \substack{\lambda \in I_{\boldsymbol{k}_1} \\ \mu \in I_{\boldsymbol{k}_2}}} (\lambda-\theta_{\boldsymbol{k}_1}) \langle \Pi_{\lambda}u^{(1)},A\Pi_{\mu} u^{(2)} \rangle =\Big\langle  \underbrace{\sum_{\lambda \in I_{\boldsymbol{k}_1}}  (\lambda-\theta_{\boldsymbol{k}_1})  \Pi_{\lambda}u^{(1)} }_{:=v^{(1)}},    A \Pi_{\mathcal{E}_{\boldsymbol{k}_2}}u^{(2)}  \Big\rangle,  \\
S_2&  := &\sum_{ \substack{\lambda \in I_{\boldsymbol{k}_1} \\ \mu \in I_{\boldsymbol{k}_2}}} (\theta_{\boldsymbol{k}_2}-\mu) \langle \Pi_{\lambda}u^{(1)},A\Pi_{\mu} u^{(2)} \rangle = \Big\langle    A \Pi_{\mathcal{E}_{\boldsymbol{k}_1}}u^{(1)} ,  \underbrace{ \sum_{\mu \in I_{\boldsymbol{k}_2}} (\theta_{\boldsymbol{k}_2}-\mu) \Pi_{\mu} u^{(2)} }_{:=v^{(2)}}\Big\rangle .
\end{eqnarray*}
In the previous formulas, we remark that $S_1$ and $S_2$ can be written as 
\begin{eqnarray*}
S_1& =& \langle v^{(1)},  A \Pi_{\mathcal{E}_{\boldsymbol{k}_2}}u^{(2)} \rangle \qquad \mbox{with} \qquad \Pi_{\mathcal{E}_{\boldsymbol{k}_1}} v^{(1)} = v^{(1)}, \\
S_2 & =& \langle  A \Pi_{\mathcal{E}_{\boldsymbol{k}_1}}u^{(1)}, v^{(2)}\rangle \qquad \mbox{with} \qquad 
\Pi_{\mathcal{E}_{\boldsymbol{k}_2}} v^{(2)} = v^{(2)}.
\end{eqnarray*}
Moreover,  thanks to \eqref{thet}, the following bounds hold true
\begin{equation}\label{boun1}
\|v^{(1)}\|_{L^2(\mathcal{M})}  \lesssim  \|u^{(1)}\|_{L^2(\mathcal{M})} \qquad \mbox{and}\qquad 
\|v^{(2)}\|_{L^2(\mathcal{M})} \lesssim \|u^{(2)}\|_{L^2(\mathcal{M})}.
\end{equation} 
Since $S_0$ contains the inner product we want to control in \eqref{adp2}, we bound as follows
\begin{eqnarray*}
S_0 & =& -S_1 -S_2+\langle [A,P] \Pi_{\mathcal{E}_{\boldsymbol{k}_1}}u^{(1)}, \Pi_{\mathcal{E}_{\boldsymbol{k}_2}}u^{(2)} \rangle \\
|S_0| & \leq &  |\langle \Pi_{\mathcal{E}_{\boldsymbol{k}_1}}v^{(1)},  A \Pi_{\mathcal{E}_{\boldsymbol{k}_2}}u^{(2)} \rangle|
+|\langle  A \Pi_{\mathcal{E}_{\boldsymbol{k}_1}}u^{(1)}, \Pi_{\mathcal{E}_{\boldsymbol{k}_2}} v^{(2)}\rangle|+|\langle [A,P] \Pi_{\mathcal{E}_{\boldsymbol{k}_1}}u^{(1)}, \Pi_{\mathcal{E}_{\boldsymbol{k}_2}}u^{(2)} \rangle|.
\end{eqnarray*}
Note that if the commutator operator $[A,P]$ is not bounded then \eqref{adp2} is of no interest since the upper bound is infinite. So there is no loss of generality to assume that $[A,P]$ is bounded.
Then we invoke the induction assumption \eqref{adp2} at rank $n$ to get the upper bound
\begin{eqnarray*}
|S_0| & \lesssim_n & 
\sup\limits_{0\leq \ell\leq n} \| (\mathrm{Ad}_P^\ell (A)\|_{L^2(\mathcal{M})\rightarrow L^2(\mathcal{M})}   \frac{ \|v^{(1)}\|_{L^2(\mathcal{M})} \|u^{(2)}\|_{L^2(\mathcal{M})} + \|u^{(1)}\|_{L^2(\mathcal{M})} \|v^{(2)}\|_{L^2(\mathcal{M})} }{(   \boldsymbol{k}_1 - \boldsymbol{k}_2 )^n} \\
 &  & +  \sup\limits_{0\leq \ell\leq n} \| (\mathrm{Ad}_P^\ell ([A,P])\|_{L^2(\mathcal{M})\rightarrow L^2(\mathcal{M})}   \frac{ \|u^{(1)}\|_{L^2(\mathcal{M})} \|u^{(2)}\|_{L^2(\mathcal{M})} }{(   \boldsymbol{k}_1 - \boldsymbol{k}_2  )^n} .
\end{eqnarray*}
We get the conclusion (namely \eqref{adp2} for $n+1$) by looking at \eqref{boun1}, \eqref{diff} and 
the equality 
\begin{equation*}
|S_0|=|(\theta_{\boldsymbol{k}_1}-\theta_{\boldsymbol{k}_2})|| \langle \Pi_{\mathcal{E}_{\boldsymbol{k}_1}}u^{(1)},A\Pi_{\mathcal{E}_{\boldsymbol{k}_2}}u^{(2)}\rangle|.
\end{equation*}
\end{proof}

\begin{lemma}\label{prod3}
Set the function $a$ given by the following product
\begin{equation}\label{defa}
a=(\Pi_{\mathcal E_{\boldsymbol{k}_{3}}} u^{(3)})\cdots (\Pi_{\mathcal E_{\boldsymbol{k}_{q}}}u^{(q)})
\end{equation}
Then following bound holds true 
\begin{equation}\label{defaa}
\|a\|_{W^{n,\infty}(\mathcal{M})} \lesssim_{n,q} \boldsymbol{k}_{3}^{\nu +n} \boldsymbol{k}_{4}^{\nu } \dots \boldsymbol{k}_{q}^{\nu } \times \prod_{\ell=3}^q \|\Pi_{\mathcal{E}_{k_\ell} }u^{(\ell)}\|_{L^2} . 
\end{equation}
\end{lemma}
\begin{proof}
Note that, thanks to \eqref{soboi} and upon increasing $\nu$, it is sufficient to prove
\begin{equation}\label{defaaa}
\|a\|_{H^{n}(\mathcal{M})} \lesssim_{n,q} \boldsymbol{k}_{3}^{\nu +n} \boldsymbol{k}_{4}^{\nu } \dots \boldsymbol{k}_{q}^{\nu } \times \prod_{\ell=3}^q \|\Pi_{\mathcal{E}_{k_\ell} }u^{(\ell)}\|_{L^2} . 
\end{equation}
For any $s>0$ and  $(u,w)\in (H^{s}(\mathcal{M})\cap L^\infty(\mathcal{M}))^2$, let us recall the following bound of a product of two functions (see  \cite[Proposition 2.1.1]{AG07}):
\begin{equation*}
\|uw\|_{H^s(\mathcal{M})}\lesssim_s 
\|u\|_{H^s(\mathcal{M})}\|w\|_{L^\infty(\mathcal{M}) }+
\|u\|_{L^\infty(\mathcal{M})}\|w\|_{H^s(\mathcal{M}) }.
\end{equation*}
Remembering the Sobolev embedding $H^{s_0}(\mathcal{M})\subset L^\infty(\mathcal{M})$ once 
we fix a real number $s_0>\frac{d}{2}$ (say $s_0=d$), we deduce 
\begin{equation*}
\|uw\|_{H^s(\mathcal{M})}\lesssim_s 
\|u\|_{H^s(\mathcal{M})}\|w\|_{H^{d}(\mathcal{M}) }+
\|u\|_{H^{d}(\mathcal{M})}\|w\|_{H^s(\mathcal{M}) }.
\end{equation*}
By using \eqref{sobo-comp},\eqref{thet0} and \eqref{thet}, we get 
$$
\|   (\Pi_{\mathcal E_{k}} u)\|_{H^{s}(\mathcal{M})}\lesssim (1+k)^s \|u\|_{L^2(\mathcal{M})}.
$$
We then understand that $\|a\|_{H^{n+\nu}(\mathcal{M})}$ is bounded (up to a multiplicative  constant depending on $n$) by
\begin{equation*}
\Big(\sum_{\ell=3}^{q} (1+\boldsymbol{k}_{\ell})^{n+\nu}  \prod_{j\neq \ell} (1+\boldsymbol{k}_j)^{d} \Big)\times \prod_{\ell=3}^q \|\Pi_{\mathcal{E}_{\boldsymbol{k}_\ell} }u^{(\ell)}\|_{L^2}  
\end{equation*}
and we immediately get \eqref{defaaa}.
\end{proof}

In view of Definition \ref{abs}, we have the following corollary
\begin{corollary}
The first set of inequalities \eqref{eq:DS_estim-KG} hold true.
\end{corollary}
\begin{proof}
Upon permuting the terms, one may assume $\boldsymbol{k}_1\geq \dots \geq \boldsymbol{k}_q$. We distinguish two cases.

\medskip

\noindent \underline{\emph{$\triangleright$ Case $1$: $\boldsymbol{k}_{1}-\boldsymbol{k}_{2} \geq \max(\boldsymbol{k}_3,\tau)$.}} 
In particular, we have  $\boldsymbol{k}_1-\boldsymbol{k}_2\geq \tau$.
Let $A$ be the bounded operator by the product function $a$ given in \eqref{defa}.
We now use \eqref{adp2} and \eqref{adP} to get
\begin{eqnarray}\nonumber
\Big| \int_{\mathcal M} (\Pi_{\mathcal E_{\boldsymbol{k}_{1}}} u^{(1)})\cdots (\Pi_{\mathcal E_{\boldsymbol{k}_{q}}}u^{(q)})\mathrm{d} x \Big|  & =& |\langle A \Pi_{\mathcal{E}_{\boldsymbol{k}_1}} u^{(1)}, \Pi_{\mathcal{E}_{\boldsymbol{k}_2}}u^{(2)}\rangle | \|\Pi_{\mathcal{E}_{\boldsymbol{k}_1}} u^{(1)}\|_{L^2(\mathcal{M})}\|\Pi_{\mathcal{E}_{\boldsymbol{k}_2}} u^{(2)}\|_{L^2(\mathcal{M})} \\ \label{bobo}
& \lesssim_n &  \frac{\|a\|_{W^{n+\nu,\infty}(\mathcal{M}) }}{(\boldsymbol{k}_{1}-\boldsymbol{k}_{2}  )^n} \|\Pi_{\mathcal{E}_{\boldsymbol{k}_1}} u^{(1)}\|_{L^2(\mathcal{M})}\|\Pi_{\mathcal{E}_{\boldsymbol{k}_2}} u^{(2)}\|_{L^2(\mathcal{M})} .
\end{eqnarray}
Again upon modifying $\nu$ and looking at \eqref{defaa}, we then obtain
\begin{equation}\label{bobo2}
\Big| \int_{\mathcal M} (\Pi_{\mathcal E_{\boldsymbol{k}_{1}}} u^{(1)})\cdots (\Pi_{\mathcal E_{\boldsymbol{k}_{q}}}u^{(q)})\mathrm{d} x \Big|   \lesssim_{n,q} 
\frac{\boldsymbol{k}_{3}^{\nu +n} \boldsymbol{k}_{4}^{\nu } \dots \boldsymbol{k}_{q}^{\nu }  }{ (\boldsymbol{k}_{1} - \boldsymbol{k}_{2})^{n}} \prod_{\ell=1}^q \| u^{(\ell)} \|_{L^2}.
\end{equation}
Since our assumption implies $\boldsymbol{k}_{1}-\boldsymbol{k}_{2} \geq \boldsymbol{k}_{3}$, we may write
\begin{equation*}
1\leq 
\frac{\boldsymbol{k}_3+\boldsymbol{k}_{1} - \boldsymbol{k}_{2}}{\boldsymbol{k}_{1} - \boldsymbol{k}_{2}} \leq 2.
\end{equation*}
Hence \eqref{bobo2} is exactly the expected upper bound \eqref{eq:DS_estim-KG}.

\medskip

\noindent \underline{\emph{$\triangleright$ Case $2$: $\max(\boldsymbol{k}_{3},\tau)> \boldsymbol{k}_{1}-\boldsymbol{k}_{2}$.}} 
By using the notation \eqref{defa}, we may invoke the Cauchy-Schwarz inequality:
\begin{equation*}
\big| \int_{\mathcal M} (\Pi_{\mathcal E_{\boldsymbol{k}_{1}}} u^{(1)})\cdots (\Pi_{\mathcal E_{\boldsymbol{k}_{q}}}u^{(q)})\mathrm{d} x \big|  \leq 
\|(\Pi_{\mathcal E_{\boldsymbol{k}_{1}}} u^{(1)})\|_{L^2(\mathcal{M})}
\|(\Pi_{\mathcal E_{\boldsymbol{k}_{2}}} u^{(2)})\|_{L^2(\mathcal{M})}
\|a\|_{L^\infty(\mathcal{M})}.
\end{equation*}
The computations made in the proof of Lemma \ref{prod3} are again available (see \eqref{defaa} for $n=0$) and we obtain the following inequality 
\begin{equation*}
\big| \int_{\mathcal M} (\Pi_{\mathcal E_{\boldsymbol{k}_{1}}} u^{(1)})\cdots (\Pi_{\mathcal E_{\boldsymbol{k}_{q}}}u^{(q)})\mathrm{d} x \big|  \lesssim_{n,q} \boldsymbol{k}_{3}^{\nu} \boldsymbol{k}_{4}^{\nu } \dots \boldsymbol{k}_{q}^{\nu }  \prod_{\ell=1}^q \| u^{(\ell)} \|_{L^2}.
\end{equation*}
Under the assumption $\max(\tau,\boldsymbol{k}_{3})> \boldsymbol{k}_{1}-\boldsymbol{k}_{2}$, the last upper bound is equivalent to 
\begin{equation*}
\left( 1+\frac{\boldsymbol{k}_{1}-\boldsymbol{k}_{2}}{\boldsymbol{k}_{3}}\right)^{-n}   \boldsymbol{k}_{3}^{\nu } \cdots \boldsymbol{k}_{q}^{\nu }\prod_{\ell=1}^q \| u^{(\ell)} \|_{L^2}
\end{equation*}
which is exactly the right-hand side of \eqref{eq:DS_estim-KG}.
\end{proof}

The second set of inequalities \eqref{eq:DS_estim-KG2} is finally a consequence of the following lemma (by keeping in mind that $n$ may be chosen arbitrary large in \eqref{eq:DS_estim-KG}).
\begin{lemma}\label{A.6} Let us consider $q\geq 3$, $n\geq 4$, $\nu \geq 0$, $\boldsymbol{k}_1 \geq \cdots \geq \boldsymbol{k}_q$, then the following inequalities hold
$$
 \frac{\boldsymbol{k}_{3}^{\nu +n}  \boldsymbol{k}_{4}^{\nu } \cdots \boldsymbol{k}_{q}^{\nu }}{ (\boldsymbol{k}_{1} - \boldsymbol{k}_{2} + \boldsymbol{k}_{3} )^{n}} \lesssim_{q,n} \Gamma_{\boldsymbol{k}} \Big( \frac{\boldsymbol{k}_2}{\boldsymbol{k}_1} \Big)^{n-3} \boldsymbol{k}_{3}^{\nu' }   \cdots \boldsymbol{k}_{q}^{\nu' }
$$
where $\nu ' = \nu +3$.
\end{lemma}
\begin{proof}Without loss of generality, we assume that $\nu = 0$. First, we note that
$$
 \frac{\boldsymbol{k}_{3} }{ \boldsymbol{k}_{1} - \boldsymbol{k}_{2} + \boldsymbol{k}_{3} } \lesssim_q  \frac{\boldsymbol{k}_{3} }{ \boldsymbol{k}_{1} - \boldsymbol{k}_{2} + q\boldsymbol{k}_{3} } \lesssim_q  \frac{\boldsymbol{k}_{3} }{\boldsymbol{k}_{1} - \boldsymbol{k}_{2} + \boldsymbol{k}_{3} + \cdots +  \boldsymbol{k}_{q} } .
$$
Looking at the definition of $\Gamma_{\boldsymbol{k}}$ in \eqref{eq:def_gamma_k}, it follows that 
\begin{eqnarray*}
 \frac{\boldsymbol{k}_{3}^{n} }{ (\boldsymbol{k}_{1} - \boldsymbol{k}_{2} + \boldsymbol{k}_{3} )^{n}} & \lesssim_{q,n} & \Big(\frac{\boldsymbol{k}_{3}}{\boldsymbol{k}_{1} - \boldsymbol{k}_{2} + \boldsymbol{k}_{3} }\Big)^3 \Big(  \frac{\boldsymbol{k}_{3} }{ \boldsymbol{k}_{1} - \boldsymbol{k}_{2} + \boldsymbol{k}_{3} } \Big)^{n-3}
\\& \lesssim_{q,n} & \boldsymbol{k}_{3}^3  \Gamma_{\boldsymbol{k}} \Big(  \frac{\boldsymbol{k}_{3} }{ \boldsymbol{k}_{1} - \boldsymbol{k}_{2} + \boldsymbol{k}_{3} } \Big)^{n-3}.
\end{eqnarray*}
Since the function $t\mapsto \frac{t}{\boldsymbol{k}_{1} - \boldsymbol{k}_{2}+t}$ is non-decreasing for $t\in [0, \boldsymbol{k}_{2}]$, its highest value is $\frac{ \boldsymbol{k}_{2}}{ \boldsymbol{k}_{1}}$.
\end{proof}

\section{Proof of Corollary \ref{cor:smooth}} 
\label{sec:appendix_B}

Without loss of generality, we focus on positive times and we assume that $\mathcal{U} = B_{h^{s_0}}(0,\rho_0)$ for some $\rho_0>0$. To avoid any possible confusion, we use the letter $\delta$ instead of $\epsilon$. We fix\footnote{i.e. from now most of the constants will depend on $v$ but we will not try to track this dependency.} $v\in h^\infty \setminus \{0\}$ and for all $\delta \in (0,1)$, we denote by $u^{(\delta)} \in C^0([0,T_\delta^{\mathrm{max}});\mathcal{U})$ the maximal solution of \eqref{eq:ham-pde} in $h^{s_0}$ with initial datum $u^{(\delta)}(0) =\delta v$. Since this solution is maximal, it satisfies
$$
T_\delta^{\mathrm{max}} = +\infty \quad \mathrm{or} \quad \lim_{t\to T_\delta^{\mathrm{max}}} \| u^{(\delta)}(t) \|_{h^{s_0}} = +\infty.
$$
First, we note that since $u^{(\delta)}(0) \in h^\infty$, thank to the tame estimate \eqref{eq:lestimee_tame_pas_facile_a_pas_oublier}, we have that $u^{(\delta)} \in C^\infty([0,T_\delta^{\mathrm{max}});h^{\infty})$.
So it suffices to focus on estimates.

\medskip

We start by deducing the following corollary of Theorem \ref{thm:main}.
\begin{lemma} \label{lem:last}
For all $s_1\geq s_{\mathrm{min}}$, for all $\mathfrak{r}\geq 1$, there exists $\delta_{\mathfrak{r},s_1}\in (0,1)$ such that for all $\delta \leq \delta_{r,s_1}$ we have $ \delta^{-\mathfrak{r}} \leq T_\delta^{\mathrm{max}}$ and
$$
\forall t\leq \delta^{-\mathfrak{r}}, \quad \| u^{(\delta)}(t) \|_{h^{s_1}} \lesssim_{s_1}  \| u^{(\delta)}(0) \|_{h^{s_1}}.
$$
\end{lemma}
\begin{proof} We apply Theorem \ref{thm:main}, with $s_c = s_1$, $u^{(0)} = u^{(\delta)}(0)$, $\mathfrak{r} = r+1$ and $s\equiv s(r,s_c)$ minimal so that the assumption $s\gtrsim_{r,s_c} 1$ is satisfied. By definition we have $\varepsilon := \|u^{(\delta)}(0)\|_{h^{s_1}} = \delta \|v\|_{h^{s_1}} \sim_{s_1} \delta$ and, since $s$ is minimal, $\|u^{(\delta)}(0)\|_{h^{s}} \lesssim_{r,s_1} \delta$. So provided that $\delta \lesssim_{s_1,r} 1$, Theorem \ref{thm:main} ensures that $\varepsilon^{-\mathfrak{r}-1} \leq T_\delta^{\mathrm{max}}$ and that for all $t\leq \varepsilon^{-\mathfrak{r}-1}$, we have $\| u^{(\delta)}(t) \|_{h^{s_1}} \lesssim_{s_1}  \| u^{(\delta)}(0) \|_{h^{s_1}}$. Then it suffices to assume that $\delta$ satisfies another bound of the kind $\delta \lesssim_{s_1,r} 1$ to deduce that $\delta^{-\mathfrak{r}} \leq \varepsilon^{-\mathfrak{r}-1}$. Thus we have proven that, provided that $\delta \lesssim_{s_1,r} 1$, we have
$$
\delta^{-\mathfrak{r}} \leq T_\delta^{\mathrm{max}} \quad \mathrm{and} \quad \forall t\leq \delta^{-\mathfrak{r}}, \quad \| u^{\delta}(t) \|_{s_1} \lesssim_{s_1} \| u^{\delta}(0) \|_{s_1}.
$$
\end{proof}

Now we apply Lemma \ref{lem:last}. For all $r\geq 1$, we set
$$
\eta_r := 2^{-r} \min_{s_1 \in s_{\mathrm{min}} +  \{ 0,\cdots,r \} } \min_{\mathfrak{r} \in \{1,\cdots,r\}} \delta_{r,s_1} .
$$
We note that, thanks to the $2^{-r}$ factor, $\eta$ is decreasing and goes to $0$ as $r$ goes to $+\infty$. As a consequence, it makes sense to set, for all $\delta \in (0,1)$,
$$
T_\delta := \delta^{-r} \quad \mathrm{if} \quad \eta_{r+1}< \delta \leq \eta_r \quad \mathrm{with} \quad r\in \mathbb{N} \quad \mathrm{and} \quad T_\delta := 0 \quad \mathrm{if} \quad \delta > \eta_1.
$$
Thus, when $\delta>\eta_1$, there is nothing to prove. Moreover, we note that $T_\delta<T_\delta^{\mathrm{max}}$ and that
$$
\lim_{\delta\to 0} \delta^{r} T_\delta = +\infty. 
$$
As a consequence, to conclude the proof, it suffices to prove that if $\delta \leq \eta_1$, $t\leq T_\delta$ and $s_2\geq 0$ then $\| u^{\delta}(t) \|_{s_2} \lesssim_{s_2} \| u^{\delta}(0) \|_{s_2}$.

\medskip

So let $r\geq 1$, $\delta \in (\eta_{r+1},\eta_r]$ and $s_2\geq 0$. We have to distinguish $2$ cases.

\medskip

\noindent \emph{\underline{Case 1: $s_2 \leq s_{\mathrm{min}} +r$.}}
Let $s_1 \in s_{\mathrm{min}} +  \{ 0,\cdots,r \}$ be such that $s_2\leq s_1< 1+s_2$. By Lemma \ref{lem:last}, we have that
$$
\forall t\leq T_\delta, \quad  \| u^{(\delta)}(t) \|_{h^{s_1}} \lesssim_{s_1}   \| u^{(\delta)}(0) \|_{h^{s_1}}.
$$
Now, using that for all $z\in h^\infty$, $s_3\mapsto \| z \|_{h^{s_3}}$ is an increasing positive function, we deduce that for all $t\leq T_\delta$,
$$
 \| u^{(\delta)}(t) \|_{h^{s_2}} \leq   \| u^{(\delta)}(t) \|_{h^{s_1}}  \lesssim_{s_1}  \| u^{(\delta)}(0) \|_{h^{s_1}} 
$$
and
$$
\| u^{(\delta)}(0) \|_{h^{s_1}}  = \| u^{(\delta)}(0) \|_{h^{s_2}}  \frac{\| v \|_{h^{s_1}}}{\| v \|_{h^{s_2}}} \leq \| u^{(\delta)}(0) \|_{h^{s_2}} \frac{\| v \|_{h^{s_1}}}{\| v \|_{\ell^2}} \sim_{s_1}  \| u^{(\delta)}(0) \|_{h^{s_2}}.
$$
Since here $s_1$ is a function of $s_2$, we have proven that $\| u^{\delta}(t) \|_{s_2} \lesssim_{s_2} \| u^{\delta}(0) \|_{s_2}$.

\medskip

\noindent \emph{\underline{Case 2: $s_2 > s_{\mathrm{min}} +r$.}} Since $T_\delta<T_\delta^{\mathrm{max}}$, we know that for all $t\leq T_\delta$, $ u^{\delta}(t) \in \mathcal{U} \cap h^{s_2}$ and so using the tame estimate \eqref{eq:lestimee_tame_pas_facile_a_pas_oublier} on $g$ and the Gr\"onwall inequality, we have that 
$$
\forall t\leq T_\delta, \quad \| u^{\delta}(t) \|_{s_2} \leq \| u^{\delta}(0) \|_{s_2}  e^{C_{s_2} t}
$$
where $C_2>0$ is a constant depending only on $s_2$. Now using that $T_\delta =  \delta^{-r}$, $\delta>\eta_{r+1}$ and $s_2>r+ s_{\mathrm{min}}$, we get that 
$$
\forall t\leq T_\delta, \quad \| u^{\delta}(t) \|_{s_2} \leq \| u^{\delta}(0) \|_{s_2}  e^{C_{s_2} \eta_{r+1}^{-r}} \leq \| u^{\delta}(0) \|_{s_2}  e^{C_{s_2} \eta_{ \lfloor s_2 \rfloor  +2}^{-s_2}} \sim_{s_2} \| u^{\delta}(0) \|_{s_2} .
$$

\end{document}